\documentclass[11pt,oneside]{gtpart} 
\usepackage{amscd,amssymb,amsxtra}
\usepackage{mathabx} \let\breve\widecheck

 \title[Differential index theorem]{An Index Theorem in Differential
$K$-Theory} 
 \author[D. S. Freed]{Daniel S.~Freed} \thanks{The first author was partially
 supported by NSF grant DMS-0603964.  The second author was
partially supported by NSF grant DMS-0903076}
 \address{Department of Mathematics \\ University of Texas \\ 1 University
Station C1200\\ Austin, TX 78712-0257}
 \email{dafr@math.utexas.edu}
 
 \author[J.~Lott ]{John~Lott}
 \address{Department of Mathematics \\ University of California, Berkeley \\
970 Evans Hall \#3840\\ Berkeley, CA 94720-3840}
 \email{lott@math.berkeley.edu}

 \dedicatory{To our teacher Isadore Singer on the occasion of his 85th
birthday} 
 \date{July 17, 2009}

\theoremstyle{definition}

\newtheorem*{example*}{Example}
\newtheorem*{problem*}{Problem}
\newtheorem*{exercise*}{Exercise}
\newtheorem*{question*}{Question}

\theoremstyle{remark}

\newtheorem{remark}[equation]{Remark}

\newtheorem*{note*}{Note}
\newtheorem*{notation*}{Notation}
\newtheorem*{remark*}{Remark}

\theoremstyle{plain}
\newtheorem{definition}[equation]{Definition}
\newtheorem{theorem}[equation]{Theorem}
\newtheorem{corollary}[equation]{Corollary}
\newtheorem{lemma}[equation]{Lemma}
\newtheorem{proposition}[equation]{Proposition}

\newtheorem*{definition*}{Definition}
\newtheorem*{theorem*}{Theorem}
\newtheorem*{corollary*}{Corollary}
\newtheorem*{lemma*}{Lemma}
\newtheorem*{proposition*}{Proposition}
\newtheorem*{conjecture*}{Conjecture}
\newtheorem*{claim*}{Claim}

\numberwithin{equation}{section}

\renewcommand{\:}{\colon}
\newcommand{\Ahat}{{\hat A}}
\DeclareMathOperator{\Aut}{Aut}
\newcommand{\CC}{{\mathbb C}}

\DeclareMathOperator{\End}{End}

\DeclareMathOperator{\Hom}{Hom}

\DeclareMathOperator{\Index}{Index}

\DeclareMathOperator{\pt}{pt}
\newcommand{\QQ}{{\mathbb Q}}

\newcommand{\RR}{{\mathbb R}}

\DeclareMathOperator{\Spin}{Spin}
\DeclareMathOperator{\STr}{STr}
\DeclareMathOperator{\str}{str}

\DeclareMathOperator{\Tr}{Tr}
\DeclareMathOperator{\tr}{tr}
\DeclareMathOperator{\trs}{Str}
\newcommand{\ZZ}{{\mathbb Z}}
\DeclareMathOperator{\ch}{ch}
\newcommand{\chiup}{\raise.5ex\hbox{$\chi$}}
\newcommand{\cir}{S^1}

\DeclareMathOperator{\ind}{ind}
\newcommand{\inv}{^{-1}}
\newcommand{\mstrut}{^{\vphantom{1*\prime y}}}

\newcommand{\temsquare}{\raise3.5pt\hbox{\boxed{ }}}

\newcommand{\zmod}[1]{\ZZ/#1\ZZ}

\newcommand{\zt}{\zmod2}
\newcommand{\RL}{\R_\Lambda }
\newcommand{\RQ}{\R_\QQ }

\newcommand{\pprod}{\pi ^{prod}}

\numberwithin{subsection}{section}
\usepackage[all,2cell,dvips]{xy}\renewcommand{\cir}{\ensuremath{S^1}}
\usepackage{color}

\newif\ifpdf 
\ifx\pdfoutput\undefined\pdffalse\else\pdftrue\fi 
 \ifpdf
 \fi

\newcommand{\cK}{\breve{K}}

\DeclareMathOperator{\Image}{Image}
\DeclareMathOperator{\Ker}{Ker}
\DeclareMathOperator{\Td}{Todd}

\DeclareMathOperator{\Det}{Det}
\DeclareMathOperator{\Quillen}{an}

\newcommand{\RZ}{\RR/\ZZ}
\renewcommand{\R}{\mathcal{R}}
\newcommand{\Spinc}{\textnormal{Spin}^c}
\newcommand{\Sp}{\mathcal{S}}
\newcommand{\WFO}{{}\mstrut _{WF} \Omega }
\newcommand{\WFcK}{{}\mstrut _{WF} \breve{K}}

\newcommand{\cb}{\check{b}}

\newcommand{\dO}{{}\mstrut _\delta \Omega }

\newcommand{\dcK}{{}\mstrut _\delta \breve{K}}

\newcommand{\sB}{\mathcal{B}}
\newcommand{\sF}{\mathcal{F}}
\newcommand{\sH}{\mathcal{H}}
\newcommand{\snablas}{\nabla \kern -.75 em \nabla }
\newcommand{\snabla}{\nabla \kern -1 em \nabla }
\newcommand{\spinc}{\textnormal{spin}^c}
\renewcommand{\AA}{\mathbb{A}}

\begin{document}

\begin{abstract} 
Let $\pi : X \rightarrow B$ be a proper submersion with a Riemannian
structure.  Given a differential $K$-theory class on $X$, we define its
analytic and topological indices as differential $K$-theory classes on
$B$.  We prove that the two indices are the same.
\end{abstract}
\maketitle







\section{Introduction}

Let $\pi \:X\to B$ be a proper submersion of relative dimension~$n$.  The
Atiyah-Singer families index theorem \cite{Atiyah-Singer} equates the
analytic and topological index maps, defined on the topological $K$-theory of
the relative tangent bundle.  Suppose that the relative tangent bundle has a
$\spinc$-structure.  This orients the map~$\pi $ in $K$-theory, and the index
maps may be expressed as pushforwards $K^0(X; \ZZ) \to K^{-n}(B; \ZZ)$.  The
topological index map $\pi_*^{top}$, which preceded the index theorem, is due
to Atiyah and Hirzebruch \cite{Atiyah-Hirzebruch}.  The analytic index
map~$\pi^{an}_*$ is defined in terms of Dirac-type operators as follows.  Let
$E\to X$ be a complex vector bundle representing a class $[E] \in K^0(X;
\ZZ)$.  Choose a Riemannian structure on the relative tangent bundle, a
$\spinc$-lift of the resulting Levi-Civita connection, and a connection on
$E$.  This geometric data determines a family of fiberwise Dirac-type
operators, parametrized by~$B$.  The analytic index $\pi^{an}_*[E] \in
K^{-n}(B; \ZZ)$ is the homotopy class of that family of Fredholm operators;
it is independent of the geometric data.  The families index theorem asserts
$\pi^{an}_* = \pi^{top}_*$.  In the special case when $B$ is a point, one
recovers the original integer-valued Atiyah-Singer index theorem
\cite{Atiyah-Singer0}.

The work of Atiyah and Singer has led to many other index theorems. In this
paper we prove a geometric extension of the Atiyah-Singer theorem, in which
$K$-theory is replaced by \emph{differential} $K$-theory.  Roughly speaking,
differential $K$-theory combines topological $K$-theory with differential
forms.  We define analytic and topological pushforwards in differential
$K$-theory.  The analytic pushforward~$\ind^{an}$ is constructed using the
Bismut superconnection and local index theory techniques.  The topological
pushforward~$\ind^{top}$ is constructed as a refinement of the
Atiyah-Hirzebruch pushforward in topological $K$-theory.  Our main result is
the following theorem.

 	\begin{theorem*} \label{maintheorem0}
  Let $\pi : X \rightarrow B$ be a proper submersion of relative dimension
$n$ equipped with a Riemannian structure and a differential
$\spinc$-structure. Then $\ind^{an} = \ind^{top}$ as homomorphisms
$\cK^0(X)\to\cK^{-n}(B)$ on differential $K$-theory.
   \end{theorem*}

This theorem provides a topological formula (with differential forms)
for geometric invariants of Dirac-type operators. We illustrate
this for the determinant line bundle (Section \ref{sec:7}) and the
reduced eta-invariant (Section \ref{sec:8}).

In the remainder of the introduction we describe the theorem and its proof in
more detail. We also give some historical background.

\subsection{Differential $K$-theory}

Let $u$ be a formal variable of degree $2$ and put $\R = \RR[u, u^{-1}]$.
For any manifold~$X$, its differential $K$-theory~$\cK^{\bullet }(X)$ fits
into the commutative square
  \begin{equation}\label{eq:112}
     \xymatrix{\cK^{\bullet }(X) \ar[r]\ar[d] & \Omega ^{\bullet
     }(X;\R)_K\ar[d]\\ K^{\bullet }(X;\ZZ) \ar[r] &H ^{\bullet }(X;\R)}. 
  \end{equation}
The bottom map is the Chern character $\ch : K^{\bullet}(X; \ZZ) \rightarrow
H(X;\R)^\bullet$; the formal variable~$u$ encodes Bott periodicity.  Also,
$\Omega(X;\R)_K^\bullet$ denotes the closed differential forms whose
cohomology class lies in the image of the Chern character.  The right
vertical map is defined by the de Rham theorem.  One can define differential
$K$-theory by positing that \eqref{eq:112}~be a \emph{homotopy} pullback
square~\cite{Hopkins-Singer}, which is the precise sense in which
differential $K$-theory combines topological $K$-theory with differential
forms.

We use a geometric model for differential $K$-theory, defined by generators
and relations.  A generator~ ${\mathcal E}$ of~$\cK^0(X)$ is a quadruple $(E,
h^E, \nabla^E, \phi)$, where $E\to X$ is a complex vector bundle, $h^E$ is a
Hermitian metric, $\nabla^E$ is a compatible connection, and $\phi \in
\Omega(X; \R)^{-1}/\Image(d)$.  The relations come from short exact sequences
of Hermitian vector bundles.  There is a similar description of
$\cK^{-1}(X)$, in which $(E,h^E,\nabla^E, \phi)$ is additionally equipped
with a unitary automorphism $U^E \in \Aut(E)$. One can then use periodicity
to define $\cK^r(X)$ for any integer $r$.

\subsection{Pushforwards for geometric submersions}

Let $\pi \:X\to B$ be a proper submersion.  A Riemannian structure on $\pi$
consists of an inner product on the vertical tangent bundle and a horizontal
distribution on $X$.  A differential $\spinc$-structure on~$\pi $ is a
topological $\spinc$-structure together with a unitary connection on the
characteristic line bundle associated to the $\spinc$-structure.  This
geometric data determines a local index form $\Td(X/B)\in \Omega (X;\R)^0$,
the $\spinc$-version of the $\Ahat$-form.

Suppose first that the fibers of~$\pi $ have even dimension~$n$.  Then there
is a diagram
  \begin{equation} \label{commintro}
     \xymatrix{ 0 \ar[r]&
           K^{-1}(X;\RZ)\ar[d]_{\ind^{an}}\ar@<1ex>[d]^{\ind^{top}}\ar[r]^{\quad
     j}&\cK^0(X)\ar[d]_{?}\ar@<1ex>[d]^{?}\ar[r]^{\omega\quad} &
           \Omega(X;\R)_K^0\ar[d]^{\int_{X/B}\Td(X/B)\wedge -}\ar[r] &0 \\
     0 \ar[r]& 
           K^{-n-1}(B;\RZ)\ar[r]^{\quad j}&\cK^{-n}(B)\ar[r]^{\omega\quad} &
           \Omega(B;\R)_K^{-n}\ar[r]& 0 } 
  \end{equation}
in which the rows are exact sequences closely related to~\eqref{eq:112},
easily derived from the definition of differential $K$-theory.  The left
vertical arrows are the topological index $\ind^{top}$, defined by a
construction in generalized cohomology theory, and the analytic index
$\ind^{an}$, defined in~ \cite{Lott}.  The main theorem of \cite{Lott} is the
equality of these arrows.  Our analytic and topological indices are defined
to fill in the~ ``?'''s in the middle vertical arrows subject to the
condition that the resulting two diagrams (with analytic and topological
indices) commute. Note that the right vertical arrow depends on the geometric
structures; hence the same is true for the middle vertical arrows.

The analytic index is based on Quillen's notion of a superconnection
\cite{Quillen2} as generalized to the infinite-dimensional setting by Bismut
\cite{Bismut}.  To define it in our finite-dimensional model of differential
$K$-theory we use the Bismut-Cheeger eta form \cite{Bismut-Cheeger}, which
mediates between the Chern character of the Bismut superconnection and the
Chern character of the finite-dimensional index bundle.  The resulting
Definition~\ref{analindex} is then a simple extension of the
$\RR/\ZZ$~analytic index in~\cite{Lott}.

As in topological $K$-theory, to define the topological index we factor~$\pi
$ as the composition of a fiberwise embedding $X \rightarrow S^N \times B$
and the projection $S^N \times B \rightarrow B$, broadly basing our
construction on~\cite{Hopkins-Singer,Klonoff,Ortiz}.  However, our
differential $K$-theory pushforward for an embedding, given in
Definition~\ref{pushforwarddef}, is new and of independent interest.  For our
proof of the main theorem we want the image to be defined in terms of
currents instead of differential forms, so the embedding pushforward lands in
the ``currential'' $K$-theory of $S^N \times B$.  The definition uses the
Bismut-Zhang current~\cite{Bismut-Zhang}, which essentially mediates between
the Chern character of a certain superconnection and a cohomologous current
supported on the image of the embedding.  A K\"unneth decomposition of the
currential $K$-theory of $S^N \times B$ is used to give an explicit formula
for the projection pushforward.  The topological index is the composition of
the embedding and projection pushforwards, with a modification to account for
a discrepancy in the horizontal distributions.  The topological index does
not involve any spectral analysis, but does involve differential forms, so
may be more appropriately termed the ``differential topological index''.

For proper submersions with odd fiber dimension, we introduce suspension and
desuspension maps between even and odd differential $K$-theory groups.  We
use them to define topological and analytic pushforward maps for the odd case
in terms of those for the even case.

The preceding constructions apply when $B$~is compact. 
To define the index maps for
noncompact~$B$, we take a limit over
an exhaustion of~$B$ by compact submanifolds.  
This depends on a result of independent interest which we
prove in the appendix: $\cK^{\bullet }(B)$ is isomorphic to the inverse limit
of the differential $K$-theory groups of compact submanifolds.

\subsection{Method of proof}

The commutativity of the right-hand square of (\ref{commintro}) for both the
analytic and topological pushforwards, combined with the exactness of the
bottom row, implies that for any~ ${\mathcal E} \in \cK^0(X)$ we have
$\ind^{an}({\mathcal E}) - \ind^{top}({\mathcal E})=j({\mathcal T})$ for a
unique $\mathcal{T}\in K^{-n-1}(B; \RR/\ZZ)$. We now use the basic method of
proof in~ \cite{Lott} to show that ${\mathcal T}$ vanishes.  Namely, it
suffices to demonstrate the vanishing of the pairing of ${\mathcal T}$ with
any element of the $K$-homology group $K_{-n-1}(B; \ZZ)$. Such pairings are
given by reduced eta-invariants, assuming the family of Dirac-type operators
has vector bundle kernel.  After some rewriting we are reduced to proving an
identity involving a reduced eta-invariant of $S^N \times B$, a reduced
eta-invariant of $X$, and the eta form on $B$.  The relation between the
reduced eta-invariant of $X$ and the eta form on $B$ is an adiabatic limit
result of Dai \cite{Dai}. The new input is a theorem of Bismut-Zhang which
relates the reduced eta-invariant of $X$ to the reduced eta-invariant of $S^N
\times B$ \cite{Bismut-Zhang}.  To handle the case when the rank of the
kernel is not locally constant we follow a perturbation argument from
\cite{Lott}, which uses a lemma of
Mischenko-Fomenko~\cite{Mischenko-Fomenko}.

\subsection{Historical discussion}

Karoubi's description of $K$-theory with coefficients \cite{Karoubi} combines
vector bundles, connections, and differential forms into a topological
framework.  His model of $K^{-1}(X; \CC/\ZZ)$ is essentially the same as the
kernel of the map~$\omega $ in~\eqref{commintro}.  (Hermitian metrics can be
added to his model to get $\RR/\ZZ$-coefficients.)  Inspired by this work,
Gillet and Soul\'e~ \cite{Gillet-Soule} defined a group $\hat{K}^0(X)$ in the
holomorphic setting which is a counterpart of differential $K$-theory in the
smooth setting. Faltings~ \cite{Faltings} and Gillet-R\"ossler-Soul\'e~
\cite{Gillet-Rossler-Soule} proved an arithmetic Riemann-Roch theorem about
these groups.  Using Karoubi's description of $K$-theory with coefficients,
the second author proved an index theorem in $\RR/\ZZ$-valued $K$-theory~
\cite{Lott}.  Based on the Gillet-Soul\'e work, he also considered what could
now be called differential flat $K$-theory $\widehat{K}^0_R(X)$ and
differential L-theory $\widehat{L}^0_\epsilon(X)$ \cite{Lott3}.

Differential $K$-theory has an antecedent in the differential character
groups of Cheeger-Simons~\cite{Cheeger-Simons}, which are isomorphic to
integral differential cohomology groups.  Independent of the developments in
the last paragraph, and with physical motivation, the first author sketched a
notion of differential $K$-theory $\cK^0(X)$~\cite{Freed,Freed-Hopkins}.  In
retrospect, differential K-theory can also be seen as a case of Karoubi's
multiplicative K-theory~\cite{Karoubinew} for a particular choice of
subcomplexes.  Generalized differential cohomology was developed by Hopkins
and Singer~\cite{Hopkins-Singer}.  In particular, they defined differential
orientations and pushforwards in their setting, and constructed a pushforward
in differential $K$-theory for such a map \cite[Example 4.85]{Hopkins-Singer}
(but with a more elaborate notion of ``differential $\spinc$-structure'').
Klonoff \cite{Klonoff} constructed an isomorphism between the differential
$K$-theory group $\cK^0(X)$ defined by Hopkins-Singer and the one given by
the finite-dimensional model used in this paper.  (We remark that the
argument in \cite{Klonoff} relies on a universal connection which is not
proved to exist; Ortiz~\cite[\S3.3 ]{Ortiz} gives a modification of the
argument which bypasses this difficulty.) A topological index map for proper
submersions is also developed in~\cite{Klonoff,Ortiz} and it fits into a
commutative diagram~\eqref{commintro}.  It has many of the same ingredients
as the topological index defined here and most likely agrees with it, but we
have not checked the details.  One notable difference is our use of currents,
a key element in our proof of the main theorem.  In \cite{Klonoff}, Klonoff
proves a version of Proposition~\ref{determinant} and
Corollary~\ref{etapush}.

There are many other recent works on differential $K$-theory, among which we
only mention two.  Bunke and Schick~\cite{Bunke-Schick} use a different
definition of $\cK^{\bullet}(X)$ in which the generators are fiber bundles
over $X$ with a Riemannian structure and a differential $\spinc$-structure.
They prove a rational Riemann-Roch-type theorem with value in rational
differential cohomology.  This theorem is the result of applying the
differential Chern character to our index theorem; see Subsection
\ref{rational}. In a different direction, Simons and
Sullivan~\cite{Simons-Sullivan} prove that the differential form~ $\phi$ in
the definition of $\cK^0(X)$ can be removed, provided that one modifies the
relations accordingly.  In this way, differential $K$-theory really becomes a
$K$-theory of vector bundles with connection.  However, in our approach to
the analytic and topological indices it is natural to include the form~
$\phi$.

The differential index theorem, or rather its consequence for determinant
line bundles, is used in Type~I string theory to prove the anomaly
cancellation known as the ``Green-Schwarz mechanism''~\cite{Freed}, at least
on the level of isomorphism classes.  Indeed, this application was one
motivation to consider a differential K-theory index theorem, for the first
author.  The differential $K$-theory formula for the determinant line bundle
reduces in special low dimensional cases to a formula in a simpler
differential cohomology theory.  There is a two-dimensional example relevant
to worldsheet string theory~\cite[\S5]{Freed3} and an example in
four-dimensional gauge theory~\cite[\S2]{Freed2}.

\subsection{Outline}

We begin in Section \ref{background} with some basic material about
characteristic forms, differential $K$-theory and reduced eta-invariants.  We
also establish our notation.  In Section \ref{sec:3} we define the analytic
index for differential $K$-theory, in the case of vector bundle kernel.  In
Section~ \ref{sec:2} we construct the pushforward for differential $K$-theory
under an embedding which is provided with a Riemannian structure and a
differential $\spinc$-structure on its normal bundle. It lands in the
currential $K$-theory of the image manifold.  In Section~ \ref{sec:4} we
construct the topological index.  In Section \ref{sec:5} we prove our main
theorem in the case of vector bundle kernel. The general case is covered in
Section \ref{sec:6}.  The relationship of our index theorem to determinant
line bundles, differential Riemann-Roch theorems, and indices in Deligne
cohomology is the subject of Section \ref{sec:7}. In Section~\ref{sec:8} we
describe how to extend the results of the preceding sections, which were for
even differential $K$-theory and even relative dimension, to the odd case by
means of suspensions and desuspensions.  In the appendix we prove that the
differential $K$-theory of a noncompact manifold may be computed as a limit
of the differential $K$-theory of compact submanifolds.

More detailed explanations appear at the beginnings of the individual
sections.
 
\subsection*{} \ \ The first author thanks Michael Hopkins and Isadore Singer
for early explorations, as well as Kiyonori Gomi for more recent
discussions.  We also thank the referee for constructive suggestions which
improved the paper.

\section{Background material} \label{background}

In this section we review some standard material and clarify
notation. 

In Subsection~\ref{subsec:1.0} we describe the Chern character and the
relative Chern-Simons form.  One slightly nonstandard point is that we
include a formal variable $u$ of degree~$2$ so that the Chern character
preserves integer cohomological degree as opposed to only a mod~2 degree.

In Subsection~\ref{subsec:1.1} we define differential $K$-theory 
in even degrees using a model which involves vector bundles, connections and
differential forms.  We will also need a slight extension of differential
$K$-theory in which differential forms are replaced by de Rham currents. This
``currential'' K-theory is introduced in Subsection~\ref{subsec:1.2}.

In Subsection~\ref{subsec:1.25} we show that on a compact odd-dimensional
$\spinc$-manifold, the Atiyah-Patodi-Singer reduced 
$\eta$-invariant gives an invariant of
currential $K$-theory.
Finally, in Subsection~\ref{subsec:1.3} we
recall Quillen's definition of superconnections and their
associated Chern character forms.

\subsection{Characteristic classes} \label{subsec:1.0}

Define the $\ZZ$-graded real algebra 
  \begin{equation}\label{eq:2}
     \R^{\bullet }=\RR[u,u\inv ],\qquad \deg u=2. 
  \end{equation}
It is isomorphic to~$K^{\bullet }(\pt;\RR)$.  

Let $X$ be a smooth manifold.  Let $\Omega (X;\R)^{\bullet}$~denote the
$\ZZ$-graded algebra of differential forms with coefficients in~$\R$; we use
the total grading.  Let $H (X;\R)^{\bullet}$ denote the corresponding
cohomology groups.  Let $R_u : \Omega (X;\R)^{\bullet} \otimes \CC
\rightarrow \Omega (X;\R)^{\bullet} \otimes \CC$ be the map which multiplies
$u$ by $2\pi i$. 

We take $K^{0}(X; \ZZ)$ to be the homotopy-invariant $K$-theory of $X$,
i.e. $K^{0}(X; \ZZ) = [X, \ZZ \times BGL(\infty, \CC)]$.
Note that one can carry out all of the usual $K$-theory constructions without
any further assumption on the manifold $X$, such as compactness or
finite topological type. For example,
given a complex vector bundle over $X$, there is always another
complex vector bundle on $X$ so that the direct sum is a trivial bundle
\cite[Problem 5-E]{Milnor}.

We can describe $K^{0}(X; \ZZ)$ as an abelian group generated by 
complex vector bundles $E$ over $X$ equipped with Hermitian metrics $h^E$.
The relations are that $E_2 = E_1 + E_3$ whenever there is
a short exact sequence of Hermitian vector bundles, meaning that
there is a short exact sequence
 \begin{equation} \label{shortexact}
0 \longrightarrow E_1 \stackrel{i}{\longrightarrow} E_2 
\stackrel{j}{\longrightarrow} E_3 \longrightarrow 0
 \end{equation}
so that $i$ and $j^*$ are isometries.
Note that in such a case, we get an orthogonal
splitting $E_2 = E_1 \oplus E_3$.

Let $\nabla^E$ be a compatible connection on $E$.  The corresponding Chern
character form is
 \begin{equation}
\omega(\nabla^E) = R_u 
\tr \left( e^{- u^{-1} (\nabla^E)^2}
\right) \in \Omega (X;\R)^0.
 \end{equation}
It is a closed form whose de Rham cohomology class $\ch(E) \in H (X;\R)^0$ is
independent of $\nabla^E$.  The map $\ch : K^{0}(X; \ZZ) \longrightarrow H
(X;\R)^0$ becomes an isomorphism after tensoring the left-hand side with~
$\RR$.  We also put
 \begin{equation}
c_1(\nabla^E) \, = \, - \frac{1}{2\pi i} \, u^{-1} \, 
\tr \left( (\nabla^E)^2
\right) \in \Omega (X;\R)^0.
 \end{equation}

We can represent $K^0(X; \ZZ)$ using $\ZZ/2\ZZ$-graded vector bundles.  A
generator of $K^0(X; \ZZ)$ is then a $\ZZ/2\ZZ$-graded complex vector bundle
$E = E_+ \oplus E_-$ on $X$, equipped with a Hermitian metric $h^E = h^{E_+}
\oplus h^{E_-}$.  Choosing unitary connections $\nabla^{E_\pm}$ and letting
$\str$ denote the supertrace, we put
 \begin{equation} \label{Chern2}
\omega(\nabla^E) = R_u \str \left( e^{- u^{-1} (\nabla^E)^2}
\right) \in \Omega (X;\R)^0.
 \end{equation}

More generally, if $r$ is even then by Bott periodicity, we can
represent a generator of $K^r(X; \ZZ)$ by a 
complex vector bundle
$E$ on $X$, equipped with a Hermitian metric $h^E$.
Again, we choose a compatible connection $\nabla^E$.
In order to define the Chern character form, we put
 \begin{equation} \label{Chern3}
\omega(\nabla^E) = u^{{r}/{2}} R_u \tr \left( e^{- u^{-1} (\nabla^E)^2}
\right) \in \Omega (X;\R)^r,
 \end{equation}
and similarly for $\ZZ/2\ZZ$-graded generators of $K^r(X; \ZZ)$.

\begin{remark} \label{newremark1}
Note the factor of $u^{{r}/{2}}$.
It would perhaps be natural to insert the formal variable~$u^{r/2}$ in
front of~$E$ but we will refrain from doing so.
In any given case, it should be clear from the context what the
degree is.
\end{remark}

If $\nabla^E_1$ and $\nabla^E_2$ are two connections on a vector
bundle $E$ then
there is an explicit relative Chern-Simons form
$CS(\nabla^E_1,\nabla^E_2) \in 
\Omega (X;\R)^{-1}/\Image(d)$. It satisfies
 \begin{equation}
dCS(\nabla^E_1,\nabla^E_2) =
\omega(\nabla^E_1) - \omega(\nabla^E_2).
 \end{equation}
More generally, if 
 \begin{equation}
0 \longrightarrow E_1 \longrightarrow E_2 \longrightarrow E_3 
\longrightarrow 0
 \end{equation}
is a short exact sequence of vector bundles with connections
$\{\nabla^{E_i}\}_{i=1}^3$ then there is an explicit 
relative Chern-Simons form
$CS(\nabla^{E_1},\nabla^{E_2},\nabla^{E_3}) \in 
\Omega (X;\R)^{-1}/\Image(d)$.
It satisfies
 \begin{equation}
dCS(\nabla^{E_1},\nabla^{E_2},\nabla^{E_3}) =
\omega(\nabla^{E_2}) - \omega(\nabla^{E_1}) - \omega(\nabla^{E_3}).
 \end{equation}
To construct $CS(\nabla^{E_1},\nabla^{E_2},\nabla^{E_3})$,
put $W = [0,1] \times X$ and let $p : W \rightarrow X$ be the
projection map.  Put $F = p^* E_2$. Let $\nabla^F$ be a
unitary connection on $F$
which equals $p^* \nabla^{E_2}$ near
$\{1\} \times X$ and which equals $p^*(\nabla^{E_1} \oplus \nabla^{E_3})$
near $\{0\} \times X$. Then
 \begin{equation}
CS(\nabla^{E_1},\nabla^{E_2},\nabla^{E_3}) = \int_0^1
\omega(\nabla^F) \in \Omega (X;\R)^{-1}/\Image(d).
 \end{equation}
 
If $W$ is a real vector bundle on $X$ with connection $\nabla^W$ then
we put
  \begin{equation}\label{eq:10}
     \Ahat(\nabla^{W})= R_u \sqrt{\det\left(\frac{u\inv \Omega^{W}/2}
     {\sinh\,u\inv \Omega^{W}/2}\right)} \quad \in \Omega (X;\R)^0,
   \end{equation}
where $\Omega^{W}$ is the curvature of~$\nabla^{W}$.

Suppose that $W$ is an oriented $\RR^n$-vector bundle on $X$ with a Euclidean
metric $h^W$ and a compatible connection $\nabla^W$.
Let $\sB \rightarrow X$ 
denote the principal $SO(n)$-bundle on $X$ to
which $W$ is associated. We say that 
$W$ has a
$\spinc$-structure if the principal $SO(n)$-bundle $\sB \rightarrow X$
lifts to a principal $\Spinc(n)$-bundle $\sF \rightarrow X$. 
Let $\Sp^W \rightarrow X$ 
be the complex spinor bundle on $X$ that is associated to
$\sF$.
It is $\zt$-graded if $n$~is even and ungraded if $n$~is odd.
Let $L^W \rightarrow X$ 
denote the characteristic line bundle on $X$ that is
associated to $\sF\to X$ by the homomorphism
$\Spinc(n)\to U(1)$.  (Recall that 
$\Spinc(n)=\Spin(n)\times _{\ZZ/2\ZZ}U(1)$;
the indicated homomorphism is trivial on the~$\Spin(n)$ factor and is the
square on
the $U(1)$~factor.)
Choose a unitary connection $\nabla^{L^W}$ on $L^W$.
Then $\nabla^{W}$ and $\nabla^{L^W}$ combine to give a connection on
$\sF$ and hence an associated 
connection $\widehat{\nabla}^W$ on $\Sp^W$.  We write
 \begin{equation}
 \Td( \widehat{\nabla}^W) = \Ahat(\nabla^{W}) \wedge
e^{\frac{c_1 \left( \nabla^{L^W} \right)}{2}} \;\in \Omega (X;\R)^0
 \end{equation}
The motivation for our notation comes from the case when $W$ is the
underlying real vector bundle of a complex vector bundle $W^\prime$. If
$W^\prime$ has a unitary structure then $W$ inherits a $\spinc$-structure. If
$\nabla^{W^\prime}$ is a unitary connection on $W^\prime$ then $\Sp^W \cong
\Lambda^{0,*}(W)$ inherits a connection $\widehat{\nabla}^W$ and $\Td (
\widehat{\nabla}^W )$ equals the Todd form of $\nabla^{W^\prime}$
\cite[Chapter 1.7]{Hirzebruch}.

\subsection{Differential K-theory}\label{subsec:1.1}

 \begin{definition}\label{def:1.14}
The differential K-theory group $\cK^0(X)$ is the abelian group 
coming from the following generators and relations.
The generators are quadruples ${\mathcal E} = 
(E, h^E, \nabla^E, \phi)$ where
 \begin{itemize}
 \item $E$ is a complex vector bundle on $X$.
 \item $h^E$ is a Hermitian metric on $E$.
 \item $\nabla^E$ is an $h^E$-compatible connection on $E$.
 \item $\phi \in \Omega (X;\R)^{-1}/\Image(d)$.
 \end{itemize}

The relations are 
${\mathcal E}_2 = {\mathcal E}_1 + {\mathcal E}_3$
whenever there is a short exact sequence (\ref{shortexact}) of
Hermitian vector bundles and
$\phi_2 = \phi_1 + \phi_3 - CS(\nabla^{E_1},\nabla^{E_2},\nabla^{E_3})$.
 \end{definition}

\noindent
 Hereafter, when we speak of a generator of $\cK^0(X)$, we will mean a
quadruple ${\mathcal E} = (E, h^E, \nabla^E, \phi)$ as above.

There is a homomorphism 
$\omega : \cK^0(X) \longrightarrow \Omega (X; \R)^0$, given on
generators by
$\omega({\mathcal E}) = \omega(\nabla^E) + d\phi$.

There is an evident extension of the definition of $\cK^0$ to 
manifolds-with-boundary. 

We can also represent $\cK^0(X)$ using $\ZZ/2\ZZ$-graded vector bundles.
A generator of $\cK^0(X)$ is then a quadruple consisting of a
$\ZZ/2\ZZ$-graded complex vector bundle $E$ on $X$, a Hermitian metric 
$h^E$ on $E$, a compatible connection $\nabla^E$ on $E$ and an element
$\phi \in \Omega(X, \R)^{-1}/\Image(d)$.
 
One can define $\cK^{\bullet}(X)$ by a general
construction~\cite{Hopkins-Singer};  it is a $2$-periodic generalized
differential cohomology theory.   

        \begin{remark}[]\label{thm:3}
The abelian group defined in Definition~\ref{def:1.14} is isomorphic to that
defined in~\cite{Hopkins-Singer}; see~\cite{Klonoff,Ortiz} for a
proof.
        \end{remark}

We use the following model in arbitrary even degrees.  For any even $r$, a
generator of $\cK^{r}(X)$ is a quadruple ${\mathcal E} = (E, h^E, \nabla^E,
\phi)$ where $\phi \in \Omega(X, \R)^{r-1}/\Image(d)$ has total degree~$r-1$.
For such a quadruple, we put
  \begin{equation}\label{eq:101}
     \omega({\mathcal E}) = u^{r/2} R_u \tr \left( e^{- u^{-1} (\nabla^E)^2}
     \right) + d\phi \in \Omega (X;\R)^r, 
  \end{equation}
and similarly for $\ZZ/2\ZZ$-graded generators of $\cK^{r}(X)$. 

\begin{remark} \label{newremark2} 
As in Remark \ref{newremark1}, it would
perhaps be natural to insert the formal variable~$u^{r/2}$ in front
of~$(E,h^E,\nabla^E)$ but we will refrain from doing so.
\end{remark}

Let $\Omega(X;\R)_K^\bullet$ denote the union of 
affine subspaces of closed forms
whose de Rham cohomology class lies in the image of
$\ch\:K^\bullet(X;\ZZ)\longrightarrow H(X;\R)^\bullet$.
There are
exact sequences
  \begin{equation}\label{eq:4}
     0 \longrightarrow
     K^{\bullet -1}(X;\RZ)\xrightarrow{\;\;i\;\;}\cK^\bullet
     (X)\xrightarrow{\;\;\omega \;\;} 
     \Omega(X;\R)_K^\bullet \longrightarrow 0 
   \end{equation}
and 
  \begin{equation}\label{eq:5}
     0\longrightarrow \frac{\Omega (X;\R)^{\bullet -1}}{\Omega
     (X;\R)_K^{\bullet -1}}\xrightarrow{\;\;j\;\;}\cK^\bullet
     (X)\xrightarrow{\;\;c\;\;} K^\bullet (X;\ZZ)\longrightarrow 0.
   \end{equation}
 
Also,
$\cK^{\bullet }(X)$~is an algebra, with the product on
$\cK^{0}(X)$ given by
 \begin{align} \label{product}
& \left( E_1, h^{E_1}, \nabla^{E_1}, \phi_1 \right) \cdot
\left( E_2, h^{E_2}, \nabla^{E_2}, \phi_2 \right) = \\
&\left( E_1 \otimes E_2, h^{E_1} \otimes h^{E_2},
\nabla^{E_1} \otimes I + I \otimes \nabla^{E_2},
\phi_1 \wedge \omega(\nabla^{E_2}) + \omega(\nabla^{E_1}) \wedge \phi_2 +
\phi_1 \wedge d\phi_2 \right). \notag
 \end{align}
Then with respect to the exact sequences~\eqref{eq:4} and~\eqref{eq:5},
  \begin{equation}\label{eq:19}
     \begin{aligned} i(x)\cb &= i\bigl(xc(\cb) \bigr), \\ j(\alpha )\cb &=
      j\bigl(\alpha \wedge \omega(\cb) \bigr). \end{aligned} 
   \end{equation}

We now describe how a differential K-theory class changes
under a deformation of
its Hermitian metric, its unitary connection and its differential form.

\begin{lemma} \label{restriction}
For $i \in \{0,1\}$, let 
$A_i : X \rightarrow [0,1] \times X$ be the embedding $A_i(x) = (i,x)$.
Given ${\mathcal E}^\prime = (E^\prime, h^{E^\prime}, \nabla^{E^\prime}, 
\phi^\prime) \in \cK^0([0,1] \times X)$,
put ${\mathcal E}_i = A_i^*  {\mathcal E}^\prime \in \cK^0(X)$.
Then ${\mathcal E}_1 = {\mathcal E}_0 + 
j(\int_0^1 \omega({\mathcal E}^\prime))$.
 \end{lemma}
 \begin{proof}
We can write $E^\prime$ as the pullback of a vector bundle on $X$, under
the projection map $[0,1] \times X \rightarrow X$. Thereby,
$E_0$ and $E_1$ get identified with a single vector bundle $E$.
After performing an automorphism of $E^\prime$, we can
also assume that $h^{E^\prime}$ is the pullback of a Hermitian
metric $h^E$ on $E$.
Since 
 \begin{equation}
\omega({\mathcal E}^\prime) = \omega(\nabla^{E^\prime}) + d\phi^\prime =
\omega(\nabla^{E^\prime}) + dt \wedge \partial_t \phi^\prime + d_X \phi^\prime,
 \end{equation}
we have
 \begin{equation}
\int_0^1 \omega({\mathcal E}^\prime) = CS(\nabla^{E}_1, \nabla^{E}_0) + 
\phi_1 - \phi_0 -
d_X  \int_0^1 \phi^\prime,
 \end{equation}
from which the lemma follows.
 \end{proof}

\begin{remark} \label{restriction2}
There is an evident extension of Lemma \ref{restriction} to the case when
${\mathcal E}^\prime = (E^\prime, h^{E^\prime}, \nabla^{E^\prime}, 
\phi^\prime) \in \cK^r([0,1] \times X)$ for $r$ even.
\end{remark}

\subsection{Currential K-theory}\label{subsec:1.2}

Let $\dO^p(X)$ denote the $p$-currents on $X$, meaning
$\dO^p(X) = \left( \Omega_c^{\dim(X)-p}(X; o) \right)^*$, where
$o$ is the flat orientation $\RR$-bundle on $X$. We
think of an element of $\dO^p(X)$ as a 
$p$-form on $X$ whose components, in a local
coordinate system, are distributional.
Consider the cocomplex
$\dO(X;\R)^{\bullet } = \dO(X) \otimes \R$ equipped with
the differential $d$ of degree $1$.
In the definition of $\cK^0(X)$, suppose that we
take $\phi \in \dO(X)^{-1}/\Image(d)$.
Let $\dcK^{\bullet }(X)$ denote the ensuing ``currential'' K-theory groups.
With an obvious meaning for $\dO(X;\R)_K^n$,
there are exact sequences 
  \begin{equation}\label{eq:4.5}
     0 \longrightarrow K^{\bullet -1}(X;\RZ)\xrightarrow{\;\;i\;\;}\cK^\bullet
     (X)\xrightarrow{\;\;\omega \;\;} \dO(X;\R)_K^\bullet \longrightarrow 0
   \end{equation}
and 
  \begin{equation}\label{eq:5.5}
     0\longrightarrow \frac{\dO (X;\R)^{\bullet -1}}{\dO
     (X;\R)_K^{\bullet -1}}\xrightarrow{\;\;j\;\;}\cK^\bullet
     (X)\xrightarrow{\;\;c\;\;} K^\bullet (X;\ZZ)\longrightarrow 0. 
   \end{equation}
However, $\dcK^{\bullet }(X)$ is not an algebra, since we can't multiply
currents.

\subsection{Reduced eta-invariants}\label{subsec:1.25}

Suppose that $X$ is a closed odd-dimensional $\spinc$ manifold. Let $L^X$
denote the characteristic line bundle of the $\spinc$ structure. We assume
that $X$ is equipped with a Riemannian metric $g^{TX}$ and a unitary
connection $\nabla^{L^X}$ on $L^X$. Let $\Sp^X$ denote the spinor bundle on
$X$.  Given a generator ${\mathcal E} = \left(E, h^E, \nabla^E, \phi \right)$
for $\dcK^0(X)$, let $D^{X,E}$ be the Dirac-type operator acting on smooth
sections of $\Sp^X \otimes E$.  Let $\overline{\eta}(D^{X,E})$ denote its
reduced eta-invariant, i.e.
 \begin{equation}
\overline{\eta}(D^{X,E}) = \frac{\eta(D^{X,E}) + \dim(\Ker(D^{X,E}))}{2} 
\, \, \, \, (mod \, \, \ZZ).
 \end{equation}

 \begin{definition} \label{bigreducedeta}

Given a generator ${\mathcal E}$ for $\dcK^0(X)$, 
define $\overline{\eta}(X, {\mathcal E}) \in 
u^{-\frac{\dim(X)+1}{2}} \cdot (\RR/\ZZ)$ by

\begin{equation}
\overline{\eta}(X, {\mathcal E}) = 
u^{-\frac{\dim(X)+1}{2}} \overline{\eta}(D^{X,E}) +
\int_X \Td \left( \widehat{\nabla}^{TX} \right)
\wedge \phi \, \, \, \, (mod \, \, u^{-\frac{\dim(X)+1}{2}} \cdot \ZZ).
 \end{equation} 
 \end{definition}

Note that $\int_X \Td \left( \widehat{\nabla}^{TX} \right) \wedge \phi$ is a
real multiple of $u^{- \frac{\dim(X)+1}{2}}$ for dimensional reasons.  
Note also that
$\overline{\eta}(X, {\mathcal E})$ generally depends on the geometric
structure of $X$.

We now prove some basic properties of 
$\overline{\eta}(X, {\mathcal E})$. 

 \begin{proposition} \label{etapairing}
 \begin{enumerate}
 \item Let $W$ be an even-dimensional
compact $\spinc$-manifold-with-boundary. 
Suppose that $W$ is equipped with a Riemannian metric $g^{TW}$ and a unitary
connection $\nabla^{L^W}$, which are products near $\partial W$.
Let ${\mathcal F}$ be a generator for
$\cK^0(W)$ which is a product near $\partial W$
and let ${\mathcal E}$ be its pullback to $\partial W$.
Then
 \begin{equation}
\overline{\eta}(\partial W, {\mathcal E}) = \int_W 
\Td \left( \widehat{\nabla}^{TW} \right)
\wedge \omega({\mathcal F})
\, \, \, \, (mod \, \, u^{-\frac{\dim(\partial W)+1}{2}} \cdot \ZZ).
 \end{equation}
 \item The assignment ${\mathcal E} \longrightarrow \overline{\eta}(X,
{\mathcal E})$ factors through a homomorphism $\overline{\eta} : \dcK^0(X)
\rightarrow u^{-\frac{\dim(X)+1}{2}} \cdot (\RR/\ZZ)$.
\item If $a\in K\inv (X;\RZ)$, then 
 \begin{equation} \label{pairing}
\overline{\eta}\left(X, i(a)\right) = u^{-\frac{\dim(X)+1}{2}} 
\langle [X],a\rangle,
 \end{equation} 
where $[X] \in K_{-1}(X; \ZZ)$ is the (periodicity-shifted) fundamental class
in K-homology and $\langle [X],a\rangle \in \RR/\ZZ$ is the result of the
pairing between $K_{-1}(X; \ZZ)$ and $K^{-1}(X; \RR/\ZZ)$.
 \end{enumerate}
 \end{proposition}
 \begin{proof}
For part (1), write ${\mathcal F} = (F, h^F, \nabla^R, \Phi)$.
By the Atiyah-Patodi-Singer
index theorem \cite{Atiyah-Patodi-Singer},
 \begin{equation}
u^{\frac{\dim(W)}{2}} \int_W \Td \left( \widehat{\nabla}^{TW} \right)
\wedge \omega(\nabla^F) - \overline{\eta}(D^{X,E}) \in \ZZ.
 \end{equation}
(Note that $\int_W \Td \left( \widehat{\nabla}^{TW} \right)
\wedge \omega(\nabla^E)$ is a real multiple of
$u^{- \frac{\dim(W)}{2}}$ for dimensional reasons.)
As 
 \begin{align}
\int_W \Td \left( \widehat{\nabla}^{TW} \right)
\wedge \omega({\mathcal F}) = &  
\int_W \Td \left( \widehat{\nabla}^{TW} \right)
\wedge ( \omega(\nabla^F) + d \Phi) \\
= &
\int_W \Td \left( \widehat{\nabla}^{TW} \right)
\wedge  \omega(\nabla^F) + \notag \\
& \int_{\partial W} \Td \left( \widehat{\nabla}^{T\partial W} \right)
\wedge \Phi, \notag
 \end{align}
part (1) follows.

To prove part (2), suppose first that we have a relation
${\mathcal E}_2 = {\mathcal E}_1 + {\mathcal E}_3$ for
$\cK^0(X)$.
Put $W = [0,1] \times X$, with a product metric.
If $p : W \rightarrow X$ is the projection map,
put $F = p^* E_2$  and $h^F = p^* h^{E_2}$.
Let $\nabla^F$ be a unitary connection on $F$
which equals $p^* \nabla^{E_2}$ near
$\{1\} \times X$ and which equals $p^*(\nabla^{E_1} \oplus \nabla^{E_3})$
near $\{0\} \times X$.  Choose $\Phi \in \Omega(W; \R)^{-1}/Im(d)$
which equals
$p^* \phi_2$ near $\{1\} \times X$ and which equals 
$p^*(\phi_1 + \phi_3)$ near $\{0\} \times X$.
Using part (1),
 \begin{align}
\overline{\eta}(X, {\mathcal E}_2) - \overline{\eta}(X, {\mathcal E}_1) -
\overline{\eta}(X, {\mathcal E}_3) = & 
\int_W 
\Td \left( \widehat{\nabla}^{TW} \right)
\wedge (\omega(\nabla^F) + d\Phi) \\
= &
\int_X \int_0^1 
\Td \left( \widehat{\nabla}^{TX} \right)
\wedge (\omega(\nabla^F) + d\Phi) \notag \\
= &
\int_X 
\Td \left( \widehat{\nabla}^{TX} \right)
\wedge 
\left(
CS \left( \nabla^{E_1}, \nabla^{E_2}, \nabla^{E_3} \right)
+ \phi_2 - \phi_1 - \phi_3 \right) \notag \\
= & 0 \notag
 \end{align}
in $u^{-\frac{\dim(X)+1}{2}} \cdot (\RR/\ZZ)$. 
This shows that $\overline{\eta}$ extends to a map
$\cK^0(X) \rightarrow u^{-\frac{\dim(X)+1}{2}} \cdot (\RR/\ZZ)$. 
The argument easily extends if we use
currents instead of forms, thereby proving part (2) of the proposition.

Part (3) follows from \cite[Proposition 3]{Lott}.
 \end{proof}

 \begin{remark}
To prove part (2) of Proposition \ref{etapairing}, we could have
used the variational formula for $\overline{\eta}$
\cite{Atiyah-Patodi-SingerIII}, which is more elementary than the
Atiyah-Patodi-Singer index theorem.
 \end{remark}

More generally, if $(E, h^E, \nabla^E, \phi)$ is a generator
for $\dcK^r(X)$ then we define $\overline{\eta}(X, {\mathcal E}) \in
u^{\frac{r-\dim(X)-1}{2}} \cdot (\RR/\ZZ)$ by
\begin{equation}
\overline{\eta}(X, {\mathcal E}) = 
u^{\frac{r-\dim(X)-1}{2}} \overline{\eta}(D^{X,E}) +
\int_X \Td \left( \widehat{\nabla}^{TX} \right)
\wedge \phi \, \, \, \, (mod \, \, u^{\frac{r-\dim(X)-1}{2}} \cdot \ZZ).
\end{equation} 

\subsection{Superconnections} \label{subsec:1.3}

Define the auxiliary ring
  \begin{equation}\label{eq:7}
     \R'=\RR[u^{1/2},u^{-1/2}] ,
   \end{equation}
where $u^{1/2}$ is a formal variable of degree~1 and $u^{-1/2}$ its inverse.
Then $\R\subset \R'$.  If $E$ is a $\frac{\ZZ}{2\ZZ}$-graded 
vector bundle on $X$ then the 
$\Omega (X; \R')$-module 
$\Omega (X;E\otimes \R')$ of differential forms with values in~$E\otimes \R'$
is $(\ZZ\times \ZZ\times \frac{\ZZ}{2\ZZ})$-graded: 
by form degree, degree in~$\R'$, and
degree in~$E$.  We use a quotient~$(\ZZ\times \frac{\ZZ}{2\ZZ})$-grading: 
the integer
degree is the sum of the form degree and the degree in~$\R'$, 
while the mod~2
degree is the degree in~$E$ plus the mod two form degree.

        \begin{definition}[]\label{thm:1}
 A \emph{superconnection}~$A$ on~$E$ is a 
graded $\Omega (X; \R')$-derivation of  $\Omega
(X;E\otimes \R')$  of degree~$(1,1)$. 
         \end{definition}

 Note that we can uniquely write
  \begin{equation}\label{eq:8}
     A = u^{1/2}\omega_0 + \nabla + u^{-1/2}\omega_2 + u^{-1}\omega_3+\cdots,
   \end{equation}
where $\nabla$ is an ordinary connection on~$E$ (which preserves degree) and
$\omega _j$ is an $\End(E)$-valued  $j$-form 
on~$X$ which is an even endomorphism if
$j$~is odd and an odd endomorphism if $j$~is even.  The powers of~$u$ are
related to the standard scaling of a superconnection.  The Chern character of
$A$ is defined by
  \begin{equation}\label{eq:9}
     \omega(A) = R_u \trs e^{u\inv \,A^2}\quad \in \Omega (M;\R)^0. 
   \end{equation}
Notice that the curvature~$A\,^2$ has degree~$(2,0)$ so $u\inv
A\,^2$ of degree~$(0,0)$ can be exponentiated.  Also, there are no
fractional powers of~$u$ in the result since the supertrace of an odd
endomorphism of~$E$ vanishes.

\section{Analytic index}\label{sec:3}

In this section we define the analytic pushforward of a
differential K-theory class under a proper submersion. This is
an extension of the analytic pushforward in $\RR/\ZZ$-valued
K-theory that was defined in 
\cite[Section 4]{Lott}. The geometric assumptions are that we have
a proper submersion $\pi : X \rightarrow B$ of
relative dimension $n$, with $n$ even,
which is
equipped with a Riemannian structure on the fibers and a
differential $\spinc$-structure (in a sense that will be made precise
below). 

Given a differential K-theory class ${\mathcal E} =
(E, h^E, \nabla^E, \phi)$ on $X$, there is an ensuing family $D^V$ of
vertical Dirac-type operators. In this section we assume that $\Ker(D^V)$ forms
a vector bundle on $B$.  (This assumption will be lifted in
Section \ref{sec:6}). In Definition \ref{analindex} 
we define the analytic pushforward
$\check{\pi}_*({\mathcal E}) \in \cK^{-n}(B)$ of ${\mathcal E}$, using
the Bismut-Cheeger eta form.  

For later purposes, we will want to extend the definition of the
analytic pushforward to certain currential K-theory classes. To
do so, we have to make a compatibility assumption between the singularities of
the current $\phi$ and the fibration $\pi$. This is phrased in terms
of the wave front set of the current $\phi$, which is a subset of $T^*X$
that microlocally measures the
singularity locus of $\phi$.  For a fiber bundle $\pi \:X\to B$ we define an
analog  
$\WFcK^0(X)$ of $\cK^0(X)$ using currents $\phi$ whose
wave front set has zero
intersection with the conormal bundle of the fibers.  Roughly speaking,
this means that the singularity locus of $\phi$ meets the fibers of
$\pi$ transversely, so we can integrate $\phi$ fiberwise to get
a smooth form on $B$. We then define the analytic pushforward on
$\WFcK^0(X)$. 


\subsection{Construction of the analytic index}\label{subsec:3.0}

Let $\pi \:X\to B$ be a  proper submersion of relative
dimension $n$, with $n$ even.  Recall that this is the same as
saying that $\pi \: X \to B$ is a
smooth fiber bundle with compact fibers of even
dimension $n$.
Let $T^VX = \Ker(d\pi)$ denote the relative tangent bundle on $X$. 

We define a \emph{Riemannian structure on~$\pi $} to be a pair consisting of
a vertical metric $g^{T^VX}$ and a horizontal distribution $T^HX$ on~$X$.  This
terminology is justified by the existence of a certain
connection on~$T^VX$ which
restricts to the Levi-Civita connection on each fiber of~$\pi $~
\cite[Definition 1.6]{Bismut}.  We recall the definition.  Let $g^{TB}$ be a
Riemannian metric on $B$. Using $g^{TB}$ and the Riemannian structure on
$\pi$, we obtain a Riemannian metric $g^{TX}$ on $X$. Let $\nabla^{TX}$ be
its Levi-Civita connection. Let $P : TX \rightarrow T^VX$ be orthogonal
projection.

 \begin{definition} \label{Bismutconnection}
The connection $\nabla^{T^VX}$ on 
$T^VX$ is $\nabla^{T^VX} = P \circ \nabla^{TX}
\circ P$. It is independent of the choice of $g^{TB}$.
 \end{definition}
 
Suppose the map $\pi$ is $\spinc$-oriented in the sense that $T^VX$ has a
$\spinc$-structure, with characteristic hermitian line bundle $L^VX\to X$.  A
\emph{differential $\spinc$-structure} on~$\pi $ is in addition a unitary
connection on~$L^VX$.  Let $\Sp^VX$ denote the associated spinor bundle
on~$X$.  The connections on~$T^VX$ and~$L^VX$ induce a connection
$\widehat{\nabla}^{T^VX}$ on $\Sp^VX$.

Define
$\pi_* : \Omega (X;\R)^{\bullet} \rightarrow
\Omega (B;\R)^{\bullet - n}$ by
 \begin{equation} \label{fiberint}
 \pi_*(\phi) = \int_{X/B} \Td \left( \widehat{\nabla}^{T^VX} \right) \wedge
\phi.
 \end{equation}
Note that our $\pi_*$ differs from the de Rham pushforward
by the factor of $\Td \left( \widehat{\nabla}^{T^VX} \right)$.
It will simplify later formulas if we use our slightly
unconventional definition.

We recall that there is a notion of the wave front set of a
current on $X$; it is the union of the wave front sets of its local
distributional coefficients \cite[Chapters 8.1 and 8.2]{Hormander}. The wave
front set is a subset of $T^*X$.  Let $N^*_V X = \pi^* T^*B \subset T^*X$ be
the conormal bundle of the fibers.  Let $\WFO(X;\R)^{\bullet}$ denote the
subspace of $\dO(X;\R)^{\bullet}$ consisting of elements whose wave front set
intersects $N^*_V X$ only at the zero section of $N^*_V X$.  By \cite[Theorem
8.2.12]{Hormander}, equation (\ref{fiberint}) defines a map \begin{equation}
\label{canintegrate} \pi_* : \WFO (X;\R)^{\bullet} \rightarrow \Omega
(B;\R)^{\bullet - n}.
\end{equation}

Let $\WFcK^0(X)$ be the abelian group whose generators are quadruples
${\mathcal E} = \left(E, h^E, \nabla^E, \phi \right)$ with
$\phi \in \WFO(X;\R)^{-1}/\Image(d)$, and with relations as before.
Then there are exact sequences
   \begin{equation}\label{eq:4.52}
     0 \longrightarrow K^{\bullet
     -1}(X;\RZ)\xrightarrow{\;\;i\;\;}\WFcK^\bullet (X)\xrightarrow{\;\;\omega
     \;\;} \WFO(X;\R)_K^\bullet \longrightarrow 0
   \end{equation}
and 
   \begin{equation}\label{eq:5.52}
     0\longrightarrow \frac{\WFO (X;\R)^{\bullet -1}}{\WFO (X;\R)_K^{\bullet
     -1}}\xrightarrow{\;\;j\;\;}\WFcK^\bullet (X)\xrightarrow{\;\;c\;\;}
     K^\bullet (X;\ZZ)\longrightarrow 0.
   \end{equation}
Here we use the fact that if $\alpha \in \WFO(X;\R)^{\bullet}$ and
$\alpha \in \Image
(d \, : \, \dO(X;\R)^{\bullet-1} \rightarrow \dO(X;\R)^{\bullet})$ 
then $\alpha \in \Image
(d \, : \, \WFO(X;\R)^{\bullet -1} \rightarrow \WFO(X;\R)^{\bullet})$.

Given a Riemannian structure on $\pi$ and a generator ${\mathcal E}$ for
$\cK^0(X)$, we want to define a pushforward of ${\mathcal E}$ that lives
in $\cK^{-n}(B)$. Write
${\mathcal E} = \left( E, h^E, \nabla^E, \phi \right)$.  Let $\sH$ denote the
(possibly infinite dimensional) vector bundle on $B$ whose fiber $\sH_b$
at~$b\in B$ is the space of smooth sections of $(\Sp^VX \otimes E)
\big|_{X_b}$. The bundle $\sH$ is $\ZZ/2\ZZ$-graded.  For $s > 0$, the
\emph{Bismut superconnection} $A_s$ is
   \begin{equation}\label{eq:23}
     A_s= s u^{1/2} D^V + 
\nabla^{\sH}  - s^{-1} u^{-1/2}\frac{c(T)}{4}. 
   \end{equation}
Here $D^V$ is the Dirac-type operator acting on $\sH_b$, $\nabla^{\sH}$ is a
certain unitary connection on $\sH$ constructed from
$\widehat{\nabla}^{T^VX}$, $\nabla^E$ and the mean curvature of the fibers,
and $c(T)$ is the Clifford multiplication by the curvature $2$-form $T$ of
the fiber bundle. For more information, see \cite[Proposition
10.15]{Berline-Getzler-Vergne}. We use powers of $s$ in
(\ref{eq:23}) in order to simplify calculations, as compared to the
powers of $s^{1/2}$ used by some other authors, but there is no
essential difference.

Now assume that $\Ker(D^V)$ forms a smooth vector bundle on $B$, necessarily
$\ZZ/2\ZZ$-graded. There are an induced $L^2$-metric $h^{\Ker(D^V)}$ and a
compatible projected connection $\nabla^{\Ker(D^V)}$.  
Note that
$[\Ker(D^V)]$ lies in $K^{-n}(B)$.
Then
 \begin{equation} \label{short}
 \lim_{s \rightarrow 0} u^{-n/2}R_u \STr \left( e^{- u^{-1} A_s^2} \right) =
\pi_*(\omega(\nabla^E)),
 \end{equation}
while
 \begin{equation} \lim_{s \rightarrow \infty} u^{-n/2} R_u \STr \left( e^{-
u^{-1} A_s^2} \right) = \omega(\nabla^{\Ker(D^V)});
 \end{equation}
see \cite[Chapter 10]{Berline-Getzler-Vergne}.  
Note that
the preceding two equations lie in forms of
total degree~$-n$.

The Bismut-Cheeger eta-form \cite{Bismut-Cheeger} is
 \begin{equation} \label{etaform}
 \tilde{\eta} = u^{-n/2} R_u \int_0^\infty \STr \left( u^{-1}
\frac{dA_s}{ds} e^{- u^{-1} A_s^2} \right) ds \in \Omega
(B;\R)^{-n-1}/\Image(d).
 \end{equation}
It satisfies
 \begin{equation}\label{eq:2.10}
d\tilde{\eta} = \pi_*(\omega(\nabla^E)) - \omega(\nabla^{\Ker(D^V)}).
 \end{equation}

 \begin{definition}\label{analindex}
 Given a generator ${\mathcal E} = \left( E, h^E, \nabla^E, \phi \right)$ for
$\cK^0(X)$, and assuming $\Ker(D^V)$ is a vector bundle, we define the
\emph{analytic index} $\ind^{an}({\mathcal E}) \in \cK^{-n}(B)$ by
 \begin{equation} \label{submerformula}
 \ind^{an}({\mathcal E}) = \left( \Ker(D^V), h^{\Ker(D^V)},
\nabla^{\Ker(D^V)}, \pi_*(\phi) + \widetilde{\eta} \right).
 \end{equation}
 \end{definition}

It follows from Theorem \ref{maintheorem} below that the assignment
${\mathcal E} \rightarrow \ind^{an}({\mathcal E})$ factors through a map
from $\cK^{0}(X)$ to $\cK^{-n}(B)$.

Given a generator ${\mathcal E}$ of $\WFcK^0(X)$, we define
$\ind^{an}({\mathcal E}) \in \cK^{-n}(B)$ by the same formula
(\ref{submerformula}).

 \begin{lemma} \label{omegapush}
If ${\mathcal E}$ is a generator for $\WFcK^0(X)$ then
$\omega(\ind^{an}({\mathcal E})) = \pi_*(\omega({\mathcal E}))$
in $\Omega(B; \R)^{-n}$.
 \end{lemma}
 \begin{proof}
From \eqref{eq:2.10},
 \begin{equation}
\omega(\ind^{an}({\mathcal E})) =
\omega(\nabla^{\Ker(D^V)}) + d(\pi_*(\phi) + \widetilde{\eta}) =
\pi_*(\omega(\nabla^E) + d\phi) = 
\pi_*(\omega({\mathcal E})),
 \end{equation}
which proves the lemma.
 \end{proof}

  \section{Pushforward under an embedding}\label{sec:2}

In this section we define a pushforward on differential $K$-theory under a
proper embedding $\iota : X \rightarrow Y$ of manifolds.  The definition uses
the data of a generator ${\mathcal E} = (E, h^E, \nabla^E, \phi)$ of
$\cK^0(X)$ and a Riemannian structure on the normal bundle $\nu$ of the
embedding.

To motivate our definition, let us recall how to push forward ordinary
$K$-theory under $\iota$
\cite{Atiyah-Hirzebruch0}.  Suppose that the normal bundle
$p : \nu \rightarrow X$ has even dimension $r$ and
is endowed with a $\spinc$-structure. Let $\Sp^\nu \rightarrow X$ 
denote the corresponding $\ZZ/2\ZZ$-graded spinor
bundle on $X$. 
Clifford multiplication by an element in $\nu$
gives an isomorphism,
on the complement of the zero-section in $\nu$,
between $p^* \Sp^\nu_+$ and
$p^* \Sp^\nu_-$. 
The $K$-theory Thom class $U_K$
is the corresponding relative class in
$K^r(D(\nu), S(\nu); \ZZ)$, where $D(\nu)$ denotes the 
closed disk bundle of $\nu$ and
$S(\nu) = \partial D(\nu)$ is the sphere bundle.

Given a vector bundle $E$ on $X$, the Thom homomorphism $K^0(X; \ZZ)
\rightarrow K^r(D(\nu), S(\nu); \ZZ)$ sends $[E]$ to $p^*[E] \cdot U_K$.
Transplanting this to a closed tubular neighborhood $T$ of $X$ in $Y$, we
obtain a relative $K$-theory class in $K(T, \partial T; \ZZ)$. Then excision
in $K$-theory defines an element $\iota_*[E] \in K^r(Y; \ZZ)$, which is the
$K$-theory pushforward of $[E]$. Applying the Chern character, one finds that
$\ch(\iota_*[E])$ is the extension to $Y$ of the cohomology class $\frac{p^*
\ch([E]) \cup U_H}{\Td(\nu)} \in H^{\bullet}(D(\nu), S(\nu); \QQ)$, where
$U_H \in H^r(D(\nu), S(\nu); \ZZ)$ is the Thom class in cohomology.

In order to push forward classes in differential $K$-theory, we will need to
carry along differential form information in the $K$-theory pushforward.
There are differential form descriptions of the Thom class in cohomology, but
they are not very convenient for our purposes.  Instead we pass to currents
and simply write the Thom homomorphism in real cohomology as the map which
sends a differential form $\omega$ on $X$ to the current $\omega \wedge
\delta_X$ on $Y$.  Following this line of reasoning, the pushforward under
$\iota$ of a differential $K$-theory class on $X$ is a currential $K$-theory
class on $Y$.  An important ingredient in its definition is a certain current
$\gamma$ defined by Bismut-Zhang \cite{Bismut-Zhang}.

 \subsection{Construcgtion of the embedding pushforward}\label{subsec:2.2}

Let $\iota\:X\hookrightarrow Y$ be a proper embedding of manifolds.  Let $r$
be the codimension of $X$ in $Y$. We assume that $r$ is even.  Let $\delta
\mstrut _X \in \dO(Y)^{r}$ denote the current of integration on $X$.

Let $\nu = \iota^*TY/TX$ be the normal bundle to $X$.  We define a
\emph{Riemannian structure on~$\iota$} to be a metric $g^\nu$ on $\nu$ and a
compatible connection $\nabla^\nu$ on $\nu$.  Suppose the map $\iota$ carries
a \emph{differential $\spinc$-structure}, in the sense that $\nu$ has a
$\spinc$-structure with characteristic hermitian line bundle $L^{\nu} \to X$
and that the 
line bundle is endowed with a unitary connection~ $\nabla^{L^\nu}$.
Let $\Sp^\nu\to X$ be the spinor bundle of $\nu$.  Then $\Sp^\nu$ inherits a
connection $\widehat{\nabla}^{\nu}$.  Let $c(\xi)$ denote Clifford
multiplication by $\xi \in \nu$ on $\Sp^\nu$.  Let $p : \nu \rightarrow X$ be
the vector bundle projection.  Then there is a self-adjoint odd endomorphism
$c \in \End(p^* \Sp^{\nu})$ which acts on $(p^* \Sp^{\nu})_\xi \cong
\Sp^{\nu}$ as Clifford multiplication by $\xi \in \nu$.

There is a pushforward map $\iota_* : \Omega(X;\R)^{\bullet} \rightarrow
\dO(Y;\R)^{\bullet + r}$ given by
 \begin{equation}
 \iota_*(\phi) = \frac{\phi}{ \Td \left( \widehat{\nabla}^{\nu} \right) }
\wedge \delta \mstrut _X.
 \end{equation}
Note that our $\iota_*$ differs from the de Rham pushforward
by the factor of $\Td \left( \widehat{\nabla}^{\nu} \right)$.
It will simplify later formulas if we use our slightly
unconventional definition.

Given a Riemannian structure on $\iota$, we want to define a map
$\check{\iota}_*\:\cK^0(X) \longrightarrow \dcK^{r}(Y)$.  To do so, we use a
construction of Bismut and Zhang \cite{Bismut-Zhang}.  Let $F$ be a
$\ZZ/2\ZZ$-graded vector bundle on $Y$ equipped with a Hermitian metric
$h^F$. We assume that we are given an odd self-adjoint endomorphism $V$ of
$F$ which is invertible on $Y- X$, and that $\Ker(V)$ has locally constant
rank along $X$.  Then $\Ker(V)$ restricts to a $\zt$-graded vector bundle on
$X$. It inherits a Hermitian metric $h^{\Ker(V)}$ from $F$.  Let
$P^{\Ker(V)}$ denote orthogonal projection from $F \big|_X$ to $\Ker(V)$. If
$F$ has an $h^F$-compatible connection $\nabla^F$ then $\Ker(V)$ inherits an
$h^{\Ker(V)}$-compatible connection given by $\nabla^{\Ker(V)} = P^{\Ker(V)}
\nabla^F P^{\Ker(V)}$.  Given a connection $\nabla^F$ on $F$, a point $x \in
X$ and a vector $\xi \in \nu_x$, lift $\xi$ to an element $\hat{\xi} \in
T_xY$ and put
 \begin{equation} \partial_\xi V = P^{\Ker(V)} \left( \nabla^F_{\hat{\xi}} V
\right) P^{\Ker(V)}.
 \end{equation}
Then $\partial_\xi V$ is an odd self-adjoint endomorphism of
$\Ker(V)$ which is independent of the choices of $\nabla^F$ and $\hat{\xi}$.
There is a well-defined odd self-adjoint endomorphism 
$\partial V \in \End(p^* \Ker(V))$ so that $\partial V$ acts on
$(p^* \Ker(V))_\xi \cong \Ker(V)$ by $\partial_\xi V$.

 \begin{lemma} \label{AH} \cite[Remark 1.1]{Bismut-Zhang}
Given a $\ZZ/2\ZZ$-graded vector bundle $E$ on $X$,
equipped with a Hermitian metric $h^E$ and a compatible
connection $\nabla^E$, there are
$F, h^F, \nabla^F, V$ on $Y$ so that
 \begin{equation}\label{eq:3.5}
    \left( \Sp^{\nu} \otimes E, h^{\Sp^{\nu}} \otimes h^E,
\widehat{\nabla}^{\nu} \otimes Id + Id \otimes \nabla^F, {c} \otimes Id
\right) \cong \left( \Ker(V), h^{\Ker(V)}, \nabla^{\Ker(V)}, \partial V
\right).
 \end{equation}
 \end{lemma} 
 \begin{proof}
Let $D(\nu)$ denote the closed unit disk bundle of $\nu$. Put
$S(\nu) = \partial D(\nu)$. Then there is a diffeomorphism
$\sigma : T \rightarrow D(\nu)$ between $D(\nu)$ and a
closed tubular neighborhood $T$ of $X$ in $Y$.
The $\ZZ/2\ZZ$-graded vector bundle $W = \sigma^* p^* (\Sp^{\nu} \otimes E)$
on $T$ is equipped with an isomorphism $J : W_+ \big|_{\partial T} 
\rightarrow W_- \big|_{\partial T}$ on
$\partial T$ given by $\sigma^* {c}$.

By the excision isomorphism in K-theory,
$K^0(Y, \overline{Y-T}) \cong K^0(T, \partial T)$. This means that
after stabilization, $W$ can be extended to a $\ZZ/2\ZZ$-graded vector
bundle $F$ on $Y$ which is equipped with an isomorphism between
$F_+ \big|_{\overline{Y-T}}$ and 
$F_- \big|_{\overline{Y-T}}$. More explicitly,
let $R$ be a vector bundle on $T$ so that
$W_- \oplus R$ is isomorphic to
$\RR^N \times T$, for some $N$.
Then $W_- \oplus R$ extends to a trivial $\RR^N$-vector bundle 
$F_-$ on $Y$. Let $F_+$ be the result of gluing the vector bundle
$W_+ \oplus R$ (on $T$) with $\RR^N \times \overline{Y-T}$
(on $\overline{Y-T}$), using the clutching isomorphism
 \begin{equation}
(W_+ \oplus R) \big|_{\partial T} \stackrel{J \oplus Id}{\longrightarrow} 
(W_- \oplus R) \big|_{\partial T}
\longrightarrow \RR^N \times \partial T
 \end{equation}
along $\partial T$. 

Let $h^R$ be a Hermitian inner product on $R$ and let
$\nabla^R$ be a compatible connection. Choose $h^{F_\pm}$ and 
$\nabla^{F_\pm}$ to
agree with $h^{W_\pm \oplus R}$ and $\nabla^{W_\pm \oplus R}$ on $T$.
Let $V_1 \in \End(F_+, F_-)$ be the result of gluing
$\sigma^* {c} \big|_{W_+} \oplus Id_R$ (on $T$) with the identity map
$\RR^N \rightarrow \RR^N$ (on $\overline{Y-T}$). Put $V = V_1 \oplus
V_1^* \in \End(F)$. Then 
$(F, h^F, \nabla^F, V)$ satisfies the claims of the lemma.
 \end{proof}

Hereafter we assume that $(F, h^F, \nabla^F, V)$ satisfies
Lemma \ref{AH}.
Note that $[F]$ lies in $K^r(Y)$.
For $s > 0$, define a superconnection $C_s$ on $F$ by
 \begin{equation}
C_s = s u^{1/2} V + \nabla^F.
 \end{equation}
Then
 \begin{equation}
\lim_{s \rightarrow 0} 
u^{{r}/{2}} R_u \str \left( e^{- u^{-1} C_s^2} \right)
= \omega(\nabla^F).
 \end{equation}
Also, from \cite[Theorem 1.2]{Bismut-Zhang},
 \begin{equation}
\lim_{s \rightarrow \infty} 
u^{{r}/{2}} R_u \str \left( e^{- u^{-1} C_s^2} \right) = 
\frac{\omega (\nabla^E)}{
\Td \left( \widehat{\nabla}^{\nu} \right)}
\wedge \delta \mstrut _X
 \end{equation}
as currents.

 \begin{definition} \cite[Definition 1.3]{Bismut-Zhang}
 Define $\gamma \in \dO(Y;\R)^{r-1}/\Image(d)$ by
 \begin{equation} \label{gamma}
 \gamma = 
u^{{r}/{2}}
R_u \int_0^\infty \str \left( u^{-1} \frac{dC_s}{ds} e^{- u^{-1} C_s^2} \right)
ds = 
u^{{r}/{2}}
R_u \int_0^\infty \str \left( u^{-1/2} V e^{- u^{-1} C_s^2} \right)
ds.
 \end{equation}
 \end{definition}

The integral on the right-hand side of (\ref{gamma}) is well-defined, as a
current on $Y$, by \cite[Theorem 1.2]{Bismut-Zhang}.  By \cite[Remark
1.5]{Bismut-Zhang}, $\gamma$ is a locally integrable differential form on $Y$
whose wave front set is contained in $\nu^*$.

 \begin{proposition} \cite[Theorem 1.4]{Bismut-Zhang}
We have
   \begin{equation}\label{eq:46}
     d\gamma = \omega(\nabla^F) - 
\frac{\omega (\nabla^E)}{
\Td \left( \widehat{\nabla}^{\nu} \right)}
\wedge \delta \mstrut _X.
   \end{equation}
 \end{proposition}

 \begin{definition} \label{pushforwarddef}
Given a generator ${\mathcal E} = \left(
E, h^E, \nabla^E, \phi \right)$ for $\cK^0(X)$, define
$\check{\iota}_* ({\mathcal E}) 
\in \dcK^r(Y)$ to be the element represented by
the quadruple
 \begin{equation}
\check{\iota}_* ({\mathcal E}) = \left( F, h^F, \nabla^F, \iota_*(\phi) -
\gamma \right).  
 \end{equation}
 \end{definition}

 \begin{lemma} \label{omegaemb}
$\omega(\check{\iota}_* ({\mathcal E})) = \iota_* (\omega({\mathcal E}))$.
 \end{lemma}
 \begin{proof}
We have
 \begin{equation}
\omega(\check{\iota}_* ({\mathcal E})) = 
\omega(\nabla^F) + d(\iota_*(\phi) - \gamma) =
\frac{\omega(\nabla^E) + d\phi}{
\Td \left( \widehat{\nabla}^{\nu} \right)
}
\wedge \delta \mstrut _X = \iota_* (\omega({\mathcal E})),
 \end{equation}
which proves the lemma.
 \end{proof}

 \begin{proposition}
 \begin{enumerate}
 \item
The pushforward $\check{\iota}_* ({\mathcal E}) \in \dcK^r(Y)$ is independent
of the choices of $F$, $h^F$, $\nabla^F$ and $V$, subject to~\eqref{eq:3.5}. 
 \item
The assignment
${\mathcal E} \rightarrow \check{\iota}_* ({\mathcal E})$ factors
through a map $\check{\iota}_* : \cK^0(X) \rightarrow \dcK^r(Y)$.
 \end{enumerate}
 \end{proposition}
 \begin{proof}
Let $G$ be a complex vector bundle on $Y$ with Hermitian metric
$h^G$ and compatible connection $\nabla^G$. If we put
$F^\prime_\pm = F_\pm \oplus G$, $h^{F^\prime_\pm} = h^{F_\pm \oplus G}$,
$\nabla^{F^\prime_\pm} = \nabla^{F_\pm \oplus G}$ and
$V^\prime = V \oplus 
 \begin{pmatrix}
0 & I_G \\
I_G & 0
 \end{pmatrix}$
then it is easy to check that $\gamma$ does not change.
Clearly
$\left( F^\prime, h^{F^\prime}, \nabla^{F^\prime}, \iota_*(\phi) -
\gamma \right)$ equals
$\left( F, h^F, \nabla^F, \iota_*(\phi) -
\gamma \right)$ in $\dcK^r(Y)$.

To prove part (1), suppose that for $i \in \{0,1\}$,
$\left( F_i, h^{F_i}, \nabla^{F_i}, V_i \right)$ are two different
choices of data as in the statement of the proposition.  Since
$F_{i,+} - F_{i,-}$ represent the same class in $K^r(Y)$ for
$i \in \{0,1\}$, the preceding paragraph implies that we can stabilize
to put ourselves into the situation 
that $F_{i,+} = F_+$ and $F_{i,-} = F_-$ for some fixed
$\ZZ/2\ZZ$-graded vector bundle $F_\pm$ on $Y$.

For $t \in [0,1]$, let 
$\left(h^{F}(t), \nabla^{F}(t), V(t) \right)$ be a smooth $1$-parameter
family of data interpolating between
$\left( h^F_0, \nabla^F_0, V_0 \right)$ and
$\left( h^F_1, \nabla^F_1, V_1 \right)$.
Consider the product embedding $\iota^\prime :
[0,1] \times X \rightarrow [0,1] \times Y$.
Let ${\mathcal E}^\prime$ be the pullback of ${\mathcal E}$ to
$[0,1] \times X$.
Let $F^\prime$ be the pullback of $F$ to $[0,1] \times Y$ and 
let $\left(h^{F^\prime}, \nabla^{F^\prime}, V^\prime \right)$ be the
ensuing data on $F^\prime$ coming from the $1$-parameter family.
Construct 
$\gamma^\prime \in 
\dO([0,1] \times Y;\R)^{r-1}/\Image(d)$
from (\ref{gamma}).
Put 
 \begin{equation}
\iota^\prime_*({\mathcal E}^\prime) = 
(F^\prime, h^{F^\prime}, \nabla^{F^\prime},
\iota^\prime_*(\phi^\prime) - \gamma^\prime).
 \end{equation}
By 
Remark \ref{restriction2} 
(or more precisely its extension to currential K-theory),
 \begin{equation}
\left( F, h^F_1, \nabla^F_1, \iota_*(\phi) - \gamma_1 \right) -
\left( F, h^F_0, \nabla^F_0, \iota_*(\phi) - \gamma_0 \right) =
j \left( \int_0^1 \omega(\iota^\prime_*({\mathcal E}^\prime)) \right)
 \end{equation} 
in $\dcK^r(Y)$.
However, using Lemma \ref{omegaemb},
$\omega(\iota^\prime_*({\mathcal E}^\prime)) =
\iota^\prime_* (\omega({\mathcal E}^\prime))$ is the pullback of
$\iota_* (\omega({\mathcal E}))$ from $Y$ to $[0,1] \times Y$.
In particular, $\int_0^1 \omega(\iota^\prime_*({\mathcal E}^\prime)) = 0$.
This proves part (1) of the proposition.

To prove part (2), suppose that we have a relation
${\mathcal E}_2 = {\mathcal E}_1  + {\mathcal E}_3$ in $\cK^0(X)$
coming from a short exact sequence (\ref{shortexact}).
Let $p : [0,1] \times X \rightarrow X$ be the projection map.
Put $E^\prime = p^* E_2$.
Let $\nabla^{E^\prime}$ be a unitary connection on $E$ which is
$p^* \nabla^{E_2}$ near $\{1\} \times X$ and which is
$p^* (\nabla^{E_1} \oplus \nabla^{E_3})$ near $\{0\} \times X$.
Choose $\phi^\prime \in \Omega([0,1] \times X; \R)^{-1}/\Image(d)$
which equals $p^* \phi_2$ near $\{1\} \times X$ and which equals
$p^*(\phi_1 + \phi_3)$ near $\{0\} \times X$. Consider the
product embedding $\iota^\prime : [0,1] \times X \rightarrow [0,1]
\times Y$ and construct $\iota^\prime_*(
{\mathcal E}^\prime) \in
\dcK^0([0,1] \times Y)$ as in Definition \ref{pushforwarddef}.
For $i \in \{0,1\}$, let $A_i : Y \rightarrow [0,1] \times Y$
be the embedding $A_i(y) = (i,y)$. From Lemmas \ref{restriction} and 
\ref{omegaemb},
 \begin{equation}
A_1^*  \iota^\prime_*(
{\mathcal E}^\prime) - A_0^* \iota^\prime_*(
{\mathcal E}^\prime) = j \left( \int_0^1 \iota^\prime_*(
\omega({\mathcal E}^\prime) ) \right) =
j \left( \iota_* \int_0^1 
\omega({\mathcal E}^\prime) \right) = 0,
 \end{equation}
since the relation 
${\mathcal E}_2 = {\mathcal E}_1  + {\mathcal E}_3$ and Lemma
\ref{restriction} imply that
$\int_0^1 \omega({\mathcal E}^\prime)$ vanishes.
Hence $\iota_*({\mathcal E}_2) = \iota_*({\mathcal E}_1) +
\iota_*({\mathcal E}_3)$. 
This proves part (2) of the proposition.
 \end{proof}

\section{Topological index}\label{sec:4}

In this section we define the topological index in differential K-theory.  We
first consider two fiber bundles $X_1 \rightarrow B$ and $X_2 \rightarrow B$,
each equipped with a Riemannian structure and a differential $\spinc$
structure in the sense of the previous section.  We now assume that we have a
fiberwise isometric embedding $\iota : X_1 \rightarrow X_2$.  The preceding
section constructed a pushforward $\check{\iota}_* : \cK^0(X_1) \rightarrow
\dcK^r(X_2)$. To define the topological index we will eventually want to
compose $\check{\iota}_*$ with the pushforward under the fibration $X_2
\rightarrow B$.  However, there is a new issue because the horizontal
distributions on the two fiber bundles $X_1$ and $X_2$ need not be
compatible. Hence we define a correction form $\widetilde{C}$ and, in
Definition~\ref{modified}, a modified embedding pushforward
$\check{\iota}^{mod}_* : \cK^0(X_1) \rightarrow \WFcK^r(X_2)$.

To define the topological index, we specialize to the case when $X_2$ is $S^N
\times B$ for some even $N$, equipped with a Riemannian structure coming from
a fixed Riemannian metric on $S^N$ and the product horizontal distribution.
In this case we show that the pushforward of $\WFcK^r(S^N \times B)$ from
$S^N \times B$ to $B$, 
as defined in Definition~\ref{analindex},
can be written as an explicit map $\check{\pi}^{prod}_* : \WFcK^r(S^N \times
B) \rightarrow \cK^{r-N}(B)$ in terms of a K\"unneth-type formula for
$\WFcK^r(S^N \times B)$. This shows that
$\check{\pi}^{prod}_*$ can be computed without any spectral analysis and, in
particular, can be defined without the assumption about vector bundle
kernel. Relabeling $X_1$ as $X$, we then define the topological index
$\ind^{top} : \cK^0(X) \rightarrow \cK^{-n}(B)$ by $\ind^{top} =
\check{\pi}^{prod}_* \circ \check{\iota}^{mod}_*$.

\subsection{Construction of the topological index}

Let $\pi_1 : X_1 \rightarrow B$ and $\pi_2 : X_2 \rightarrow B$ be fiber
bundles over $B$, with compact fibers $X_{1,b}$ and $X_{2,b}$ of even
dimension $n_1$ and $n_2$, respectively.  Let $\iota : X_1 \rightarrow X_2$
be a fiberwise embedding of even codimension $r$, i.e., $\iota$ is an
embedding, $\pi_2 \circ \iota = \pi_1$, and $n_2=n_1+r$.  Let $\nu$ be the
normal bundle of $X_1$ in $X_2$.  There is a short exact sequence
 \begin{equation} \label{ses}
0 \longrightarrow T^VX_1 \longrightarrow \iota^* T^VX_2 \longrightarrow \nu 
\longrightarrow 0, 
 \end{equation}
of vector bundles on $X_1$.  Suppose that $\pi_1$ and $\pi_2$ have Riemannian
structures.

 \begin{definition} \label{iotacomp}
 The map $\iota$ is {\em compatible with the Riemannian structures on $\pi_1$
and $\pi_2$} if for each $b \in B$, $\iota_b : X_{1,b} \rightarrow X_{2,b}$
is an isometric embedding.
 \end{definition}

\noindent
 The intersection of $T^H X_2$ with $TX_1$ defines a horizontal distribution
$(T^HX_2) \big|_{X_1}$ on $X_1$. We do \emph{not} assume that it coincides
with $T^HX_1$.  It follows that the orthogonal projection of $\iota^*
\nabla^{T^VX_2}$ to $T^VX_1$ is \emph{not} necessarily equal to
$\nabla^{T^VX_1}$.

In the rest of this section we assume that $\iota$ is compatible with the
Riemannian structures on $\pi_1$ and $\pi_2$. Then $\nu$ inherits a
Riemannian structure from \eqref{ses}, which is split by identifying~$\nu $
as the orthogonal complement to~$T^VX_1$ in~$\iota ^*T^VX_2$.  Namely, the
metric $g^\nu$ is the quotient inner product from~$g^{T^VX_2}$ and the
connection $\nabla^\nu$ is compressed from $\iota^* \nabla^{T^VX_2}$.

We also assume a certain 
compatibility of the 
differential $\spinc$-structures on~$\pi _1$, $\pi _2$ and 
$\iota$. To describe this compatibility, recall the discussion of
$\spinc$-structures from Subsection \ref{subsec:1.0}.
Over~$X_1$ we have principal bundles~$\sF_1$, $\sF_\nu$ and 
$\iota^* \sF_2$, with structure groups~$\Spinc(n_1)$, $\Spinc(r)$
and $\Spinc(n_2)$, respectively.
They project to the oriented orthonormal frame bundles $\sB_1$,
$\sB_\nu$ and $\iota^* \sB_2$.
The
embedding~$\iota $ gives a reduction $\sB_1\times \sB_\nu
\hookrightarrow \iota^* \sB_2$ of $\iota^* \sB_2$ 
which is compatible with the inclusion
$SO_{n_1}\times SO_{r}\hookrightarrow SO_{n_2}$.  Then we postulate that
we have a lift
  \begin{equation}\label{eq:1}
     \sF_1\times \sF_\nu \to \iota^*\sF_2 
  \end{equation}
of $\sB_1\times \sB_\nu
\hookrightarrow \iota^* \sB_2$
which is compatible with the 
homomorphism $\Spinc({n_1})\times \Spinc({r})\to
\Spinc({n_2})$.  (The kernel of this homomorphism is a $U(1)$-factor
embedded anti-diagonally.)  Finally,
we suppose that the three $\spinc$-connections 
are compatible in the sense that
the $U(1)$-connection on~$\iota^* \Ker(\sF_2\to\sB_2)$ pulls back
under~\eqref{eq:1} to the tensor product of the $U(1)$-connections 
on~$\Ker(\sF_1 \rightarrow \sB_1)$
and~$\Ker(\sF_\nu \rightarrow \sB_2)$.
Said in terms of the
characteristic line bundles, there is an isomorphism $L^{T^VX_1} \otimes
L^\nu \to \iota^* L^{T^VX_2}$ which is compatible with the metrics and 
connections. 

We now prove a lemma which shows that the elements of the image of 
$\check{\iota}_*$ have good wave front support.

 \begin{lemma} \label{wflemma}
$\nu^* \subset T^*X_2 \big|_{X_1}$ 
intersects $N^*_V X_2 = \pi_2^* T^*B \subset T^*X_2$ only in the zero section.
 \end{lemma}
 \begin{proof}
Suppose that
$x_1 \in X_1$ and $\xi \in \nu^*_{x_1} \cap (N^*_V X_2)_{x_1}$.
Then $\xi$ annihilates $T_{x_1}X_1$ and $(\iota^*T^VX_2)_{x_1}$.
Since $T_{x_1}X_1 \cap (\iota^*T^VX_2)_{x_1} = T^V_{x_1}X_1$, it
follows easily that $T_{x_1}X_1 + (\iota^*T^VX_2)_{x_1} =
(\iota^*TX_2)_{x_1}$. Thus $\xi$ vanishes.
 \end{proof}

Hence for a generator ${\mathcal E}$ of $\cK^0(X_1)$, the element
$\check{\iota}_*({\mathcal E}) \in \dcK^r(X_2)$ is the image of
a unique element in $\WFcK^r(X_2)$, which we will also call 
$\check{\iota}_*({\mathcal E})$.

We now define a certain correction term to take into account the
possible non-compatibility between the horizontal distributions
on $X_1$ and $X_2.$ 
That is, using (\ref{ses}), we construct an explicit form
$\widetilde{C} \in \Omega (X_1; \R)^{-1}/\Image(d)$ so that
 \begin{equation} \label{dC}
d\widetilde{C} = \iota^* 
\Td \left( \widehat{\nabla}^{T^VX_2} \right)
- \Td \left( \widehat{\nabla}^{T^VX_1} \right) \wedge
\Td \left( \widehat{\nabla}^{\nu} \right).
 \end{equation}
Namely, put $W = [0,1] \times X_1$ and let $p : W \rightarrow X_1$
be the projection map.  Put $F =  p^* \iota^* T^VX_2$.
Consider a $\spinc$-connection $\widehat{\nabla}^F$ on $F$ which is
$\iota^* \widehat{\nabla}^{T^VX_2}$ near $\{1\} \times X_1$ and which is
$\widehat{\nabla}^{T^VX_1} \oplus \widehat{\nabla}^{\nu}$ near
$\{0\} \times X_1$. Then
$\widetilde{C} = \int_0^1 \Td \left( \widehat{\nabla}^F \right)
\in \Omega (X_1; \R)^{-1}/\Image(d)$.

 \begin{lemma} \label{vanishes}
Suppose that $(T^HX_2) \big|_{X_1} = T^HX_1$. 
Then 
 \begin{enumerate}
 \item The orthogonal projection of $\iota^* \nabla^{T^VX_2}$ to
$T^VX_1$ equals $\nabla^{T^VX_1}$, and
 \item
$\widetilde{C} = 0$.
 \end{enumerate}
 \end{lemma}
 \begin{proof}
Suppose that $(T^HX_2) \big|_{X_1} = T^HX_1$. 
Choose a Riemannian metric $g^{TB}$ on $B$ and construct
$g^{TX_2}$, $\nabla^{TX_2}$, $\nabla^{T^VX_2}$,
$g^{TX_1}$, $\nabla^{TX_1}$ and $\nabla^{T^VX_1}$ as in
Definition \ref{Bismutconnection}. Let
$P_{12} : \iota^* TX_2 \rightarrow TX_1$ be orthogonal projection.
By naturality,
$\nabla^{TX_1} = P_{12} \circ \iota^* \nabla^{TX_2} \circ P_{12}$ and
$\nabla^{T^VX_1} = P_{12} \circ \iota^* \nabla^{T^VX_2} \circ P_{12}$.
This proves part (1) of the lemma.

As $\iota^* \widehat{\nabla}^{T^VX_2} = \widehat{\nabla}^{T^VX_1} \oplus
\widehat{\nabla}^{\nu}$, it follows that $\widetilde{C} = 0$.
This proves part (2) of the lemma.
 \end{proof}

 \begin{definition} \label{modified}
Define the modified pushforward $\check{\iota}_*^{mod} ({\mathcal E}) \in
\WFcK^r(X_2)$ by
 \begin{equation}
\check{\iota}_*^{mod} ({\mathcal E}) = \check{\iota}_* ({\mathcal E}) -
j \left( \frac{\widetilde{C}
}{
\iota^* 
\Td \left( \widehat{\nabla}^{T^VX_2} \right)
 \wedge
\Td \left( \widehat{\nabla}^{\nu} \right)
} \wedge \omega({\mathcal E}) \wedge \delta_{X_1}\right).
 \end{equation}
 \end{definition}

 \begin{lemma} \label{pushmod}
 \begin{equation}
\omega(\check{\iota}_*^{mod} ({\mathcal E})) =
\frac{\Td \left( \widehat{\nabla}^{T^VX_1} \right)}
{
\iota^* 
\Td \left( \widehat{\nabla}^{T^VX_2} \right)
} \wedge \omega({\mathcal E}) \wedge \delta_{X_1}.
 \end{equation}
 \end{lemma}
 \begin{proof}
We have
 \begin{align}
\omega(\check{\iota}_*^{mod} ({\mathcal E})) = & 
\omega(\check{\iota}_*({\mathcal E})) -
\frac{d\widetilde{C}}{\iota^* 
\Td \left( \widehat{\nabla}^{T^VX_2} \right)
 \wedge
\Td \left( \widehat{\nabla}^{\nu} \right)
} \wedge \omega({\mathcal E}) \wedge \delta_{X_1} \\
= & 
\frac{\omega({\mathcal E})}{
\Td \left( \widehat{\nabla}^{\nu} \right)
}
\wedge \delta \mstrut _{X_1} - \notag \\ 
&  \frac{
\iota^* 
\Td \left( \widehat{\nabla}^{T^VX_2} \right)
- \Td \left( \widehat{\nabla}^{T^VX_1} \right) \wedge
\Td \left( \widehat{\nabla}^{\nu} \right)}{\iota^* 
\Td \left( \widehat{\nabla}^{T^VX_2} \right)
 \wedge
\Td \left( \widehat{\nabla}^{\nu} \right)
} \wedge \omega({\mathcal E}) \wedge \delta_{X_1} \notag \\
= & \frac{
\Td \left( \widehat{\nabla}^{T^VX_1} \right)
}{
\iota^* 
\Td \left( \widehat{\nabla}^{T^VX_2} \right)
} \wedge \omega({\mathcal E}) \wedge \delta_{X_1}. \notag
 \end{align}
This proves the lemma.
 \end{proof}

In the next 
lemma we consider the submersion pushforward in the case
of a product bundle, under the assumption that the differential
K-theory class on the total space has an almost-product form.

 \begin{lemma} \label{productpushforward} Let $Z$ be a compact Riemannian
$\spinc$-manifold of even dimension $n$ with a unitary connection
$\nabla^{L_Z}$ on the characteristic 
line bundle $L^Z$.  Let $\pi^Z : Z \rightarrow \pt$ be
the map to a point.  Let $B$ be any manifold.  Let $\pi^{prod} : Z \times B
\rightarrow B$ be projection on the second factor.  Let $T^H_{prod}(Z \times
B)$ be the product horizontal distribution on the fiber bundle $Z \times B
\rightarrow B$.  Let $p : Z \times B \rightarrow Z$ be projection on the
first factor.  Suppose that ${\mathcal E}^Z=(E^Z, h^{E^Z}, \nabla^{E^Z},
\phi^Z)$ and ${\mathcal E}^B=(E^B, h^{E^B}, \nabla^{E^B}, \phi^B)$ are
generators for $\cK^n(Z)$ and $\cK^{r-n}(B)$, respectively, for some even
integer $r$.  Let $\pi^Z_*(E^Z) \in K^0(\pt)$ denote the K-theory pushforward
of $[E^Z] \in K^n(Z)$ under the map $\pi^Z : Z \rightarrow \pt$.  (We can
identify $\pi^Z_*(E^Z)$ with $\int_Z \Td( \widehat{\nabla}^{TZ}) \wedge
\omega(\nabla^{E^Z}) = \Index(D^{Z,E^Z}) \in \ZZ$.)  Given $\phi \in \WFO(Z
\times B)^{r-1}/\Image(d)$, put ${\mathcal E} = (p^* {\mathcal E}^Z) \cdot
((\pi^{prod})^* {\mathcal E}^B) + j(\phi)$. Then
 \begin{equation}
\check{\pi}^{prod}_* {\mathcal E} = \pi^Z_*(E^Z) \cdot {\mathcal E}^B +
j(\pi^{prod}_*(\phi)) 
 \end{equation}
in $\cK^{r-n}(B)$. 
 \end{lemma}
 \begin{proof}
Using (\ref{product}), we can write
 \begin{align} \label{write}
{\mathcal E} = & \left(
p^* E^Z \otimes (\pi^{prod})^* E^B,
p^* h^{E^Z} \otimes (\pi^{prod})^* h^{E^B},
p^* \nabla^{E^Z} \otimes I + I \otimes (\pi^{prod})^* \nabla^{E^B}, \right. \\
& \left. p^* \phi^Z \wedge (\pi^{prod})^* \omega(\nabla^{E^B}) +
p^* \omega(\nabla^{E^Z}) \wedge (\pi^{prod})^* \phi^B +
p^* \phi^Z \wedge (\pi^{prod})^* d\phi^B + \phi
\right). \notag
 \end{align}
Also, in this product situation, we have 
$\nabla^{T^V(Z \times B)} = p^* \nabla^{TZ}$,
$\Ker(D^V) = \Ker(D^{Z,E^Z}) \otimes E^B$,
$h^{\Ker(D^V)} = h^{\Ker(D^{Z,E^Z})} \otimes h^{E^B}$ and
$\nabla^{\Ker(D^V)} = I_{\Ker(D^{Z,E^Z})} \otimes  \nabla^{E^B}$.
Regarding the eta form, as $[D^V, \nabla^{\mathcal H}] =
(\nabla^{\mathcal H})^2 = T = 0$, 
we have
 \begin{equation}
\tilde{\eta} = 
u^{\frac{r-n}{2}} R_u \int_0^\infty \STr \left( u^{-1} \frac{dA_s}{ds}
e^{- u^{-1} A_s^2} \right) ds =
u^{\frac{r-n}{2}} R_u \int_0^\infty \STr \left( u^{- 1/2} D^V
e^{- s^2 (D^V)^2} \right) ds = 0
\end{equation}
in $\Omega (X;\R)^{r-n-1}/\Image(d)$,
for parity reasons.
Then
 \begin{align} \label{pusheqn}
\check{\pi}^{prod}_* {\mathcal E} = &
\left( \Ker(D^{Z,E^Z}) \otimes E^B,
h^{\Ker(D^{Z,E^Z})} {\otimes} h^{E^B},
I_{\Ker(D^{Z,E^Z})} {\otimes} \nabla^{E^B}, \right. \\
& \left. 
\pi^{prod}_*
\left(
p^* \phi^Z \wedge (\pi^{prod})^* \omega(\nabla^{E^B}) +
p^* \omega(\nabla^{E^Z}) \wedge (\pi^{prod})^* \phi^B +
p^* \phi^Z \wedge (\pi^{prod})^* d\phi^B + \phi
\right) \right) \notag \\
= &
\left( \Ker(D^{Z,E^Z}) \otimes E^B,
h^{\Ker(D^{Z,E^Z})} {\otimes} h^{E^B},
I_{\Ker(D^{Z,E^Z})} {\otimes} \nabla^{E^B}, \right. \notag \\
& \left. 
\pi^Z_*(\omega(\nabla^{E^Z})) \cdot \phi^B +
\pi^{prod}_*(\phi)
\right) \notag \\
= & \pi^Z_*(E^Z) \cdot {\mathcal E}^B +
j(\pi^{prod}_*(\phi) ). \notag 
 \end{align}
This proves the lemma.
 \end{proof}

The next lemma is a technical result, which will be used later,
about the functoriality of reduced eta invariants with respect
to product structures.

 \begin{lemma} \label{func}
Under the hypotheses of Lemma \ref{productpushforward}, suppose in 
addition that $B$ is an odd-dimensional 
closed $\spinc$-manifold, equipped with 
a Riemannian metric $g^{TB}$ and a unitary connection $\nabla^{L^B}$.
Then $\overline{\eta}(B, \check{\pi}_*^{prod} {\mathcal E}) =
\overline{\eta}(Z \times B, {\mathcal E})$ in 
$u^{\frac{r-n-\dim(B)-1}{2}} \cdot (\RR/\ZZ)$.
 \end{lemma}
 \begin{proof}
Using (\ref{pusheqn}), we have
 \begin{align}
\overline{\eta}(B, \check{\pi}_*^{prod} {\mathcal E}) = &
u^{\frac{r-n-\dim(B)-1}{2}} \cdot
\Index(D^{Z,E^Z}) \cdot \overline{\eta}(D^{B,E^B}) + \\
& \int_B \Td(\widehat{\nabla}^{TB}) \wedge
\left( (\pi^Z_*(\omega(\nabla^{E^Z})) \cdot \phi^B + \pi^{prod}_*(\phi)
\right). \notag
\end{align}
By separation of variables, it is easy to show that
\begin{equation}
\Index(D^{Z,E^Z}) \cdot \overline{\eta}(D^{B,E^B}) \, = \,
\overline{\eta}(D^{Z \times B,p^*E^Z \otimes \pi^*E^B}).
\end{equation}
Next,
\begin{align}
\int_B \Td(\widehat{\nabla}^{TB}) \wedge
(\pi^Z_*(\omega(\nabla^{E^Z})) \cdot \phi^B \, & = \,
\left( \int_Z \Td(\widehat{\nabla}^{TZ}) \wedge
\omega(\nabla^{E^Z}) \right) \cdot
\int_B \Td(\widehat{\nabla}^{TB}) \wedge
\phi^B \\
& = \, \int_{Z \times B} 
\Td(\widehat{\nabla}^{T(Z \times B)}) \wedge
p^* \omega(\nabla^{E^Z}) \wedge (\pi^{prod})^* \phi^B. \notag
\end{align}
Also,
\begin{equation}
\int_B \Td(\widehat{\nabla}^{TB}) \wedge
\pi^{prod}_*(\phi) \, = \,
\int_{Z \times B} 
\Td(\widehat{\nabla}^{T(Z \times B)}) \wedge
\phi.
\end{equation}
Hence
\begin{align}
\overline{\eta}(B, \check{\pi}_*^{prod} {\mathcal E})
= & 
u^{\frac{r-n-\dim(B)-1}{2}} \cdot
\overline{\eta}(D^{Z \times B,p^*E^Z \otimes \pi^*E^B}) + \\
& \int_{Z \times B} 
\Td(\widehat{\nabla}^{T(Z \times B)}) \wedge
\left( p^* \omega(\nabla^{E^Z}) \wedge (\pi^{prod})^* \phi^B + \phi \right).
\notag
 \end{align}

On the other hand, from (\ref{write}),
 \begin{align}
\overline{\eta}(Z \times B, {\mathcal E}) = & 
u^{\frac{r-n-\dim(B)-1}{2}} \cdot
\overline{\eta}(D^{Z \times B,p^*E^Z \otimes \pi^*E^B}) + 
\int_{Z \times B} \Td(\widehat{\nabla}^{T(Z \times B)}) \wedge \\
& \left( p^* \phi^Z \wedge (\pi^{prod})^* \omega(\nabla^{E^B}) +
p^* \omega(\nabla^{E^Z}) \wedge (\pi^{prod})^* \phi^B +
p^* \phi^Z \wedge (\pi^{prod})^* d\phi^B + \phi \right) \notag \\
= & u^{\frac{r-n-\dim(B)-1}{2}} \cdot
\overline{\eta}(D^{Z \times B,p^*E^Z \otimes \pi^*E^B}) + \notag \\
& \int_{Z \times B} 
\Td(\widehat{\nabla}^{T(Z \times B)}) \wedge
\left( p^* \omega(\nabla^{E^Z}) \wedge (\pi^{prod})^* \phi^B + \phi \right).
\notag
 \end{align}
This proves the lemma.
 \end{proof}

We now work towards the construction of the topological
index, beginning with a result about embedding in spheres. 

 \begin{lemma} \label{embeddingexists}
Suppose that $\pi : X \rightarrow B$ is a fiber bundle with $X$ compact
and even-dimensional fibers of dimension $n$.
Suppose that $\pi$ has a Riemannian structure.
Given $N$ even, let $\pi^{prod} : S^N \times B \rightarrow B$ be the
product bundle.
Then for large $N$, there are an
embedding $\iota : X \rightarrow S^N \times B$ and a
Riemannian metric on $S^N$ (independent of $b \in B$) so that
$\iota$ is compatible with
the Riemannian structures on $\pi$ and $\pi^{prod}$.
(In applying Definition \ref{iotacomp}, we take
$X_1 = X$ and $X_2 = S^N \times B$.)
 \end{lemma}
 \begin{proof}
Let $g^{TB}$ be any Riemannian metric on $B$. Using the Riemannian
structure on $\pi$, there is a corresponding Riemannian metric
$g^{TX}$ on $X$. 
Let $e : X \rightarrow S^N$ be any isometric embedding of $X$ into
an even-dimensional sphere with some Riemannian metric.
Put $\iota(x) = (e(x), \pi(x)) \in S^N \times B$.
 \end{proof}

Next, we establish a K\"unneth-type formula for the differential K-theory of
$S^N \times B$.  We endow $S^N$ with an arbitrary
Riemannian metric and an arbitrary
unitary connection $\nabla^{L^{S^N}}$ on its
characteristic line bundle $L^{S^N}$. 

 \begin{lemma}\label{thm:88}
 Given ${\mathcal E} \in \WFcK^r(S^N \times B)$ with $r$ and $N$ even,
consider the fibering $\pi^{prod} : S^N \times B \rightarrow B$.  Let $p :
S^N \times B \rightarrow S^N$ be projection onto the first factor.  Then
there are generators $\{ {\mathcal E}^{S^N}_i \}_{i=1}^2$ for $\cK^N(S^N)$,
generators $\{ {\mathcal E}^{B}_i \}_{i=1}^2$ for $\cK^{r-N}(B)$, and  $\phi
\in \WFO(S^N \times B)^{r-1}/\Image(d)$ so that
 \begin{equation} \label{decomposition}
 {\mathcal E} = \sum_{i=1}^2 (p^* {\mathcal E}^{S^N}_i) \cdot ((\pi^{prod})^*
{\mathcal E}^B_i) + j(\phi).
 \end{equation}
 \end{lemma}

        \begin{proof}
 By the K\"unneth formula in K-theory, we can write 
  \begin{equation}\label{eq:102}
     c(\mathcal{E}) = \sum\limits_{i=1}^2 p^*e_i^{S^N}\cdot (\pprod)^*e_i^B 
  \end{equation}
for additive generators $\{e_i^{S^N}\}_{i=1}^2$ of $K^N(S^N)$ and classes 
$\{e_i^B\}_{i=1}^2$ in
$K^{r-N}(B)$.  Lift the $e_i$'s to differential K-theory
classes~$\mathcal{E}_i$.
Then the
exact sequence~\eqref{eq:5.52} implies the existence of~$\phi $. 
        \end{proof}

        \begin{remark}[]\label{thm:2}
It is possible to replace the compact manifold~$S^N$ in Lemma \ref{thm:88}
with the noncompact affine space~$\AA^N$, provided that we use currential
$K$-theory with compact supports.  In that case the summation
in~\eqref{decomposition} would only have a single term, and we could remove
the assumption that $X$~is compact in Lemma~\ref{embeddingexists}.  We chose
to avoid introducing compact supports, at the expense of having a slightly
more complicated lemma.
        \end{remark}

Using Lemma \ref{productpushforward}, we obtain an explicit formula
for the pushforward, under the product submersion $S^N \times B \rightarrow B$,
of a differential K-theory class of the type considered in
Lemma \ref{thm:88}. We now show that the result is independent of the
particular K\"unneth-type representation chosen.  

 \begin{lemma} \label{decpushforward}
Given a generator ${\mathcal E}$ for $\WFcK^r(S^N \times B)$ with
$r$ and $N$ even, write 
${\mathcal E}$ as in (\ref{decomposition}). Apply
the map $\check{\pi}^{prod}_*$ in Lemma \ref{productpushforward}
to ${\mathcal E}$ in the case $Z = S^N$, to get an element of
$\cK^{r-N}(B)$. Then the result factors through a map 
$\check{\pi}_*^{prod} : \WFcK^r(S^N \times B) \rightarrow \cK^{r-N}(B)$, 
which is independent
of the particular decomposition (\ref{decomposition}) chosen.
 \end{lemma}

        \begin{proof}
 We refer to the notation in the proof of Lemma~\ref{thm:88}.  Let 
$\{1,x\}$~be
an additive basis of~$K^0(S^N)$, where $1$~is the trivial bundle of rank~1
and $x$~has rank~0.  Choose~$e_1^{S^N}=u^{{N}/{2}}1$ 
and~$e_2^{S^N}=u^{{N}/{2}}x$,
where $u$ denotes the Bott element in K-theory.
Without loss of generality, we can
assume that $x$ is chosen so that $\pi _*^{S^N}(u^{N/2}x)=1
\in \ZZ$.  Given a differential K-theory class ${\mathcal E}$ as in
\eqref{decomposition}, Lemma \ref{productpushforward} implies that
  \begin{equation}\label{eq:104}
     \check\pi^{prod}_*\mathcal{E} = \mathcal{E}_2^B + j(\pi^{prod}_*(\phi) ). 
  \end{equation}
Now a different decomposition, as in \eqref{decomposition}, of the
same differential K-theory class ${\mathcal E}$, can only arise by the
changes
  \begin{equation}\label{eq:103}
     \begin{aligned} \mathcal{E}_i^{S^N} &\longrightarrow \mathcal{E}_i^{S^N}
      + j(\alpha _i^{S^N}), \\ \mathcal{E}_i^{B} &\longrightarrow
      \mathcal{E}_i^{B} + j(\alpha _i^B), \\ \phi &\longrightarrow \phi
      - \sum\limits_{i=1}^2 \bigl[p^*\alpha _i^{S^N}\wedge (\pprod)^*\omega
      (\mathcal{E}_i^{B}) \;+\; p^*\omega (\mathcal{E}_i^{S^N}) \wedge
      (\pprod)^*\alpha _i^B\bigr]\end{aligned} 
  \end{equation}
for some $\alpha_i^{S^N} \in \Omega (S^N;\R)^{N-1}/\Image(d)$ and
$\alpha_i^{B} \in \Omega (B;\R)^{r-N-1}/\Image(d)$
The ensuing change in the right-hand side of (\ref{eq:104}) is
\begin{equation} \label{expression}
j (\alpha_2^B) - \sum\limits_{i=1}^2 j \left( \pi_*^{prod} \left(
p^*\alpha _i^{S^N}\wedge (\pprod)^*\omega
      (\mathcal{E}_i^{B}) \;+\; p^*\omega (\mathcal{E}_i^{S^N}) \wedge
      (\pprod)^*\alpha _i^B \right) \right).
\end{equation}
As $\pi^{S^N}_* \left( \alpha_i^{S^N} \right) = 
\pi^{S^N}_* \left( \omega (\mathcal{E}_1^{S^N}) \right) = 0$
and
$\pi^{S^N}_* \left( \omega (\mathcal{E}_2^{S^N}) \right) = 1$,
the expression in (\ref{expression}) vanishes. The lemma follows.
        \end{proof}

The point of Lemma \ref{decpushforward} is that it gives us a
well-defined map $\check{\pi}_*^{prod} : \WFcK^{r}(S^N \times B) \rightarrow
\cK^{r-N}(B)$ which agrees with the pushforward defined in
Section \ref{sec:3} when applied to elements of $\WFcK^{r}(S^N \times B)$
that are written in the form (\ref{decomposition}), and which can be
computed explicitly, but does not need
any spectral analysis.  In particular, $\check{\pi}_*^{prod}$ is defined
without 
any condition about vector bundle kernel.  (Note that if $E$ is a
general Hermitian vector bundle on $S^N \times B$ and $\nabla^E$ is
a general compatible connection on $E$ then there is no reason that
$\Ker(D^V)$ should form a vector bundle on $B$.)

We now define the topological index for compact base spaces~$B$; the
extension to proper submersions with noncompact~$B$ is described at the end
of Section~\S\ref{sec:6}.

 \begin{definition} Let $\pi : X \rightarrow B$ be a fiber bundle with $X$
compact.  Put $n = \dim(X) - \dim(B)$, which we assume to be even.  Suppose
that $\pi$ has a Riemannian structure.  Construct $N$ and $\iota$ from Lemma
\ref{embeddingexists}.  Given a generator ${\mathcal E}$ for $\cK^0(X)$,
construct $\check{\iota}_*^{mod} ({\mathcal E}) \in \WFcK^{N-n}(S^N \times B)$
from Definition \ref{modified}. Write $\check{\iota}_*^{mod} ({\mathcal E})$ as
in equation (\ref{decomposition}). Using Lemma \ref{decpushforward}, define the
{\em topological index} $\ind^{top}({\mathcal E}) \in \cK^{-n}(B)$ by
 \begin{equation}
\ind^{top}({\mathcal E}) = \check{\pi}_*^{prod}(\check{\iota}_*^{mod} 
({\mathcal E})).
 \end{equation}
 \end{definition}

 \begin{lemma} \label{omegaagrees}
 \begin{equation}
\omega(\ind^{top}({\mathcal E})) = \pi_*(\omega({\mathcal E})).
 \end{equation}
 \end{lemma}
 \begin{proof}
From Lemmas \ref{omegapush} and \ref{pushmod},
 \begin{align}
\omega(\ind^{top}({\mathcal E})) &=
\omega(\check{\pi}_*^{prod}(\check{\iota}_*^{mod} 
({\mathcal E}))) \\&=
{\pi}_*^{prod}(\omega(\check{\iota}_*^{mod} 
({\mathcal E}))) \notag \\
&=  {\pi}_*^{prod} \left( \frac{
\Td \left( \widehat{\nabla}^{T^VX} \right)
}{
\iota^* 
\Td \left( \widehat{\nabla}^{T^V(S^N \times B)} \right)
} \wedge \omega({\mathcal E}) \wedge \delta_{X} \right) \notag \\
&= \pi_*(\omega({\mathcal E})). \notag
 \end{align}
This proves the lemma.
 \end{proof}

 \begin{proposition}
The following diagram commutes :
  \begin{equation}
     \xymatrix{ 0 \ar[r]& \dfrac{\Omega (X;\R)^{-1}}{\Omega
     (X;\R)_K^{-1}}\ar[d]^{\pi_*}\ar[r]^{\quad
           j}&\cK^0(X)\ar[d]^{\ind^{top}}\ar[r]^{c\quad} &
           K^0(X;\ZZ)\ar[d]^{\ind^{top}}\ar[r] &0 \\ 0 \ar[r]&
           \dfrac{\Omega (B;\R)^{-n-1}}{\Omega
     (B;\R)_K^{-n-1}}\ar[r]^{\quad j}&\cK^{-n}(B)\ar[r]^{c\quad} &
           K^{-n}(B;\ZZ)\ar[r]& 0. } 
  \end{equation}
\end{proposition}
\begin{proof}
The right-hand square commutes from our construction of 
$\ind^{top} : \cK^0(X) \rightarrow \cK^{-n}(B)$; see the discussion 
at the beginning of Section \ref{sec:2} of the K-theory pushforward
under an embedding. To see that the left-hand
square commutes, suppose that
$\phi \in \frac{\Omega (X;\R)^{-1}}{\Omega(X;\R)_K^{-1}}$.
Then 
\begin{align}
\ind^{top}(j(\phi)) = & \check{\pi}^{prod}_*(\check{\iota}^{mod}_*(
j(\phi))) \\
 = & \check{\pi}^{prod}_* \left( j \left(
\frac{\phi}{\Td \left( \widehat{\nabla}^{\nu} \right)} \wedge \delta_X - 
\frac{\widetilde{C}
}{
\iota^* 
\Td \left( \widehat{\nabla}^{T^V(S^N \times B} \right)
 \wedge
\Td \left( \widehat{\nabla}^{\nu} \right)
} \wedge d\phi \wedge \delta_{X} \right) \right) \notag \\
= & j \left( \int_X \left( \frac{\iota^* 
\Td \left( \widehat{\nabla}^{T^V(S^N \times B} \right)
}{
\Td \left( \widehat{\nabla}^{\nu} \right)} \wedge \phi -
\frac{\widetilde{C}}{\Td \left( \widehat{\nabla}^{\nu} \right)} 
\wedge d\phi \right) \right) \notag \\
= & j \left( \int_X \frac{\iota^* 
\Td \left( \widehat{\nabla}^{T^V(S^N \times B} \right) - d\widetilde{C}
}{
\Td \left( \widehat{\nabla}^{\nu} \right)} \wedge \phi \right) \notag \\
= & j \left( \int_X \Td \left( T^VX \right) \wedge \phi \right)
= j(\pi_*(\phi)). \notag
\end{align}
This proves the lemma.
\end{proof} 

From what has been said so far, the map $\ind^{top} : \cK^0(X) \rightarrow
\cK^{-n}(B)$ depends on the Riemannian structure on $\pi$ and, possibly, on
the embedding $\iota$.  We prove in Corollary~\ref{indie} that it is in fact
independent of~$\iota $.

\section{Index theorem: vector bundle kernel}\label{sec:5}

In this section we prove our index theorem for families of Dirac operators,
under the assumption of vector bundle kernel and compact base space.

In terms of the diagram
  \begin{equation} \label{comm5}
     \xymatrix{ 0 \ar[r]&
           K^{-1}(X;\RZ)\ar[d]^{}\ar[r]^{\quad
           j}&\cK^0(X)\ar[d]_{\ind^{an}}\ar@<1ex>[d]^{\ind^{top}}\ar[r]^{\omega\quad}
           & \Omega(X;\R)_K^0\ar[d]^{\pi_*}\ar[r] &0 \\
     0 \ar[r]& 
           K^{-n-1}(B;\RZ)\ar[r]^{\quad j}&\cK^{-n}(B)\ar[r]^{\omega\quad} &
           \Omega(B;\R)_K^{-n}\ar[r]& 0, } 
  \end{equation}
we know that if ${\mathcal E} \in \cK^0(X)$ then $\omega(\ind^{an}({\mathcal
E}) - \ind^{top}({\mathcal E})) = 0$. Hence $\ind^{an}({\mathcal E}) -
\ind^{top}({\mathcal E})$ is the image under $j$ of a unique element in
$K^{-n-1}(B; \RR/\ZZ)$. We now apply the method of proof of \cite[Section
4]{Lott} to prove that the difference vanishes, by computing its pairings
with elements of $K_{-n-1}(B)$. From Lemma \ref{etapairing}(2), such pairings
are given by reduced eta-invariants. As in \cite[Section 4]{Lott}, the
pairing with an element of $K_{-n-1}(B)$ becomes a computation of reduced
$\eta$-invariants on $X$ after taking adiabatic limits. A new ingredient is
the use of the main theorem of \cite{Bismut-Zhang} in order to relate the
reduced eta-invariants of a manifold and an embedded submanifold.

 \begin{theorem} \label{maintheorem}
 Let $\pi \:X\to B$ be a fiber bundle with compact fibers of even dimension.
Suppose that $\pi $~is equipped with a Riemannian structure and a
differential $\spinc$ structure.  Assume that $X$~is compact and that
$\Ker(D^V)\to B$ is a vector bundle.  Then for all~${\mathcal E} \in
\cK^0(X)$ we have $\ind^{an}({\mathcal E}) = \ind^{top}({\mathcal E})$.
 \end{theorem}
 \begin{proof} The short exact sequence (\ref{eq:4.5}), along with Lemmas
\ref{omegapush} and \ref{omegaagrees}, implies that $\ind^{an}({\mathcal E}) -
\ind^{top}({\mathcal E})$ lifts uniquely to an element ${\mathcal T}$ of
$K^{-n-1}(B; \RR/\ZZ)$. We want to show that this element vanishes.  To do
so, we use the method of proof of \cite[Section 4]{Lott}. From the universal
coefficient theorem and the divisibility of $\RR/\ZZ$, it suffices to show
that for all $\alpha \in K_{-n-1}(B; \ZZ)$, the pairing $\langle \alpha,
{\mathcal T} \rangle$ vanishes in $\RR/\ZZ$.  From \cite{Hopkins-Hovey},
$K_{-n-1}(B; \ZZ)$ is generated by elements of the form $\alpha = f_* [M]$
where $M$ is a closed odd-dimensional 
$\spinc$-manifold, $[M] \in K_{-n-1}(M;
\ZZ)$ is the fundamental class of $M$
(shifted from $K_{\dim(M)}(M; \ZZ)$ to $K_{-n-1}(M; \ZZ)$ 
using Bott periodicity)
and $f : M \rightarrow B$ is a smooth map.
(The argument in
\cite[Section 4]{Lott} used instead the Baum-Douglas description of
K-homology \cite{Baum-Douglas}, which essentially involves an additional
vector bundle on $M$.)  As $\langle \alpha, {\mathcal T} \rangle = \langle
[M], f^* {\mathcal T} \rangle$, we can effectively pull everything back to
$M$ and so reduce to considering the case when $B$ is an arbitrary closed
odd-dimensional $\spinc$-manifold.

Now suppose that ${\mathcal E} = \left( E, h^E, \nabla^E, \phi \right)$.
Recall the construction of $\check{\iota}_* ({\mathcal E}) =
\left( F, h^F, \nabla^F, \iota_*(\phi) +
\gamma \right)$ from
Definition \ref{pushforwarddef}.

In the rest of this proof, all equalities will be taken modulo the integers,
so will be written as congruences.
We equip $B$ with a Riemannian metric $g^{TB}$,
and the characteristic line bundle $L^B$ with
a unitary connection $\nabla^{L^B}$.  We equip the fiber bundle $S^N
\times B \rightarrow B$ with the product horizontal connection
$T^H_{prod}(S^N \times B)$. Then $S^N \times B$ has the product Riemannian
metric, from which the submanifold $X$ acquires a Riemannian metric.

By Proposition \ref{etapairing} and Lemma \ref{func},
 \begin{equation}
u^{-\frac{\dim(X)+1}{2}} \langle [B], {\mathcal T} \rangle \equiv C_1 - C_2
 \end{equation}
in $u^{-\frac{\dim(X)+1}{2}} \cdot (\RR/\ZZ)$, 
where
 \begin{align}
C_1 \equiv & \overline{\eta}(B, \ind^{an}({\mathcal E})) \\
\equiv & u^{-\frac{\dim(X)+1}{2}}
\overline{\eta} \left( D^{B,{\Ker(D^V)_+}} \right) -
u^{-\frac{\dim(X)+1}{2}}
\overline{\eta} \left( D^{B,{\Ker(D^V)_-}} \right) +
\int_B \Td \left( \widehat{\nabla}^{TB} \right)
\wedge (\pi_* (\phi) + \widetilde{\eta})
\notag \\
\equiv & u^{-\frac{\dim(X)+1}{2}}
\overline{\eta} \left( D^{B,{\Ker(D^V)_+}} \right) -
u^{-\frac{\dim(X)+1}{2}}
\overline{\eta} \left( D^{B,{\Ker(D^V)_-}} \right) +
\int_B \Td \left( \widehat{\nabla}^{TB} \right) \wedge
\widetilde{\eta} + \notag \\
& 
\int_X \pi^* \Td \left( \widehat{\nabla}^{TB} \right) \wedge
\Td \left( \widehat{\nabla}^{T^VX} \right) \wedge \phi \notag
 \end{align}
and, using Lemma \ref{func},
 \begin{align} \label{C2}
C_2 \equiv & \overline{\eta}(B, \ind^{top}({\mathcal E})) \equiv
\overline{\eta}(S^N \times B, \check{\iota}_*^{mod}({\mathcal E})) \\
\equiv & u^{-\frac{\dim(X)+1}{2}}
\overline{\eta} \left( D^{S^N \times B, F_+} \right) -
u^{-\frac{\dim(X)+1}{2}}
\overline{\eta} \left( D^{S^N \times B, F_-} \right) +
\int_{S^N \times B} \Td \left( \widehat{\nabla}^{T(S^N \times B)} \right) 
\wedge \notag \\
& \left( \iota_* \phi - \gamma -
\frac{\widetilde{C}
}{
\iota^* 
\Td \left( \widehat{\nabla}^{T^V(S^N \times B)} \right) \wedge
\Td \left( \widehat{\nabla}^{\nu} \right)
}
\wedge \omega({\mathcal E}) \wedge \delta_{X} \right)
\notag \\
\equiv & 
u^{-\frac{\dim(X)+1}{2}}
\overline{\eta} \left( D^{S^N \times B, F_+} \right) -
u^{-\frac{\dim(X)+1}{2}}
\overline{\eta} \left( D^{S^N \times B, F_-} \right) + \notag \\
& 
\int_{X} 
\frac{\iota^* \Td \left( \widehat{\nabla}^{T(S^N \times B)} \right)
}{
\Td \left( \widehat{\nabla}^{\nu} \right)
}
\wedge \phi - 
\int_{S^N \times B} \Td \left( \widehat{\nabla}^{T(S^N \times B)} 
\right) \wedge
\gamma - 
\notag \\
&
\int_{X} 
\frac{\pi^* \Td \left( \widehat{\nabla}^{TB} \right)
}{
\Td \left( \widehat{\nabla}^{\nu} \right)}
\wedge \widetilde{C}
\wedge (\omega(\nabla^E) + d\phi). \notag
 \end{align}
From \cite[Theorem 2.2]{Bismut-Zhang},
 \begin{equation}
\overline{\eta} 
\left( D^{S^N \times B,{F_+}} \right) -
\overline{\eta} 
\left( D^{S^N \times B,{F_-}} \right) \equiv 
\overline{\eta} 
\left( D^{X,E} \right) +
u^{\frac{\dim(X)+1}{2}}
\int_{S^N \times B} \Td \left( \widehat{\nabla}^{T(S^N \times B)} \right)
\wedge \gamma.
 \end{equation}
Thus
 \begin{align}
C_2 \equiv & 
u^{-\frac{\dim(X)+1}{2}} \overline{\eta} \left( D^{X,E} \right) -
\int_{X} 
\frac{\pi^* \Td \left( \widehat{\nabla}^{TB} \right)
}{
\Td \left( \widehat{\nabla}^{\nu} \right)}
\wedge
\widetilde{C}
\wedge \omega(\nabla^E)
+ \int_{X} 
\frac{\iota^* \Td \left( \widehat{\nabla}^{T(S^N \times B)} \right)
}{
\Td \left( \widehat{\nabla}^{\nu} \right)}
\wedge \phi - \\ 
& \int_{X} 
\frac{\pi^* \Td \left( \widehat{\nabla}^{TB} \right)
}{
\Td \left( \widehat{\nabla}^{\nu} \right)}
\wedge
\widetilde{C}
\wedge d\phi. \notag
 \end{align}
Now
 \begin{align} \label{C3}
& \int_{X} 
\frac{\iota^* \Td \left( \widehat{\nabla}^{T(S^N \times B)} \right)
}{
\Td \left( \widehat{\nabla}^{\nu} \right)}
\wedge
\phi - 
\int_{X} 
\frac{\pi^* \Td \left( \widehat{\nabla}^{TB} \right)
}{
\Td \left( \widehat{\nabla}^{\nu} \right)}
\wedge
\widetilde{C}
\wedge d\phi  \equiv \\
&
\int_{X} 
\frac{\iota^* \Td \left( \widehat{\nabla}^{T(S^N \times B)} \right)
}{
\Td \left( \widehat{\nabla}^{\nu} \right)}
\wedge \phi - 
\int_{X} 
\frac{\pi^* \Td \left( \widehat{\nabla}^{TB} \right)
}{
\Td \left( \widehat{\nabla}^{\nu} \right)}
\wedge
d\widetilde{C}
\wedge \phi \equiv \notag \\
& \int_{X} 
\frac{\iota^* \Td \left( \widehat{\nabla}^{T(S^N \times B)} \right)
}{
\Td \left( \widehat{\nabla}^{\nu} \right)}
\wedge
\phi - \notag \\
& \int_{X} 
\frac{\pi^* \Td \left( \widehat{\nabla}^{TB} \right)
}{
\Td \left( \widehat{\nabla}^{\nu} \right)}
\wedge
\left(
\iota^* 
\Td \left( \widehat{\nabla}^{T^V(S^N \times B)} \right)
- \Td \left( \widehat{\nabla}^{T^VX} \right) \wedge
\Td \left( \widehat{\nabla}^{\nu} \right)
\right)
\wedge \phi \equiv \notag \\
& \int_X \pi^* \Td \left( \widehat{\nabla}^{TB} \right) \wedge
\Td \left( \widehat{\nabla}^{T^VX} \right) \wedge
\phi. \notag
 \end{align}
Then 
 \begin{align}
C_1 - C_2 \equiv & u^{-\frac{\dim(X)+1}{2}}
\overline{\eta} \left( D^{B,{\Ker(D^V)_+}} \right) -
u^{-\frac{\dim(X)+1}{2}}
\overline{\eta} \left( D^{B,{\Ker(D^V)_-}} \right) + 
\int_B \Td \left( \widehat{\nabla}^{TB} \right) \wedge
\widetilde{\eta}
- \\
& u^{-\frac{\dim(X)+1}{2}}
\overline{\eta} \left( D^{X,E} \right) + 
\int_{X} 
\frac{\pi^* \Td \left( \widehat{\nabla}^{TB} \right)
}{
\Td \left( \widehat{\nabla}^{\nu} \right)}
\wedge \widetilde{C}
\wedge \omega(\nabla^E). \notag
 \end{align}

 The next lemma, stated in terms of bordisms, shows that $C_1 - C_2$ is
unchanged by certain perturbations.

 \begin{lemma} \label{variation}
Suppose that $B = \partial B^\prime$ for some even-dimensional compact
$\spinc$-manifold $B^\prime$.
Suppose that the structures,
$g^{TB}$, $\nabla^{L^B}$, $\pi : X \rightarrow B$, $T^HX$, 
$\iota : X \rightarrow S^N \times B$, $E \rightarrow X$ and
$\nabla^E$ extend to structures
$g^{TB^\prime}$, $\nabla^{L^{B^\prime}}$, 
$\pi^\prime : X^\prime \rightarrow B^\prime$, $T^HX^\prime$, 
$\iota^\prime : X^\prime \rightarrow S^N \times B^\prime$, 
$E^\prime \rightarrow X^\prime$ and
$\nabla^{E^\prime}$ over $B^\prime$,
which are product-like near $B = \partial B^\prime$.
Suppose that $\Ker(D^V)^\prime$ forms a $\ZZ/2\ZZ$-graded
vector bundle on $B^\prime$.
Then $C_1 - C_2 \equiv 0$.
 \end{lemma}
 \begin{proof}
From Lemma \ref{etapairing},
 \begin{equation} \label{first}
u^{-\frac{\dim(X)+1}{2}}
\overline{\eta} \left( D^{B,{\Ker(D^V)_+}} \right) -
u^{-\frac{\dim(X)+1}{2}}
\overline{\eta} \left( D^{B,{\Ker(D^V)_-}} \right) \equiv
\int_{B^\prime} \Td (\widehat{\nabla}^{TB^\prime}) \wedge
\omega \left( \nabla^{\Ker(D^V)^\prime} \right)
 \end{equation}
and
 \begin{equation}
u^{-\frac{\dim(X)+1}{2}}
\overline{\eta} \left( D^{X,E} \right) \equiv
\int_{X^\prime} \Td (\widehat{\nabla}^{TX^\prime}) \wedge
\omega \left( \nabla^{E^\prime} \right).
 \end{equation}
Also,
 \begin{align}
\int_B \Td \left( \widehat{\nabla}^{TB} \right) \wedge
\widetilde{\eta} \equiv &
\int_{B^\prime} \Td \left( \widehat{\nabla}^{TB^\prime} \right) \wedge
d\widetilde{\eta}^\prime \\
\equiv &
\int_{X^\prime} (\pi^\prime)^*
\Td \left( \widehat{\nabla}^{TB^\prime} \right) \wedge
\Td \left( \widehat{\nabla}^{T^V X^\prime} \right) \wedge
\omega(\nabla^{E^\prime}) - \notag \\
& \int_{B^\prime} \Td \left( \widehat{\nabla}^{TB^\prime} \right) \wedge
\omega \left( \nabla^{\Ker(D^V)^\prime} \right) \notag
 \end{align}
and
 \begin{align} \label{last}
& \int_{X} 
\frac{\pi^* \Td \left( \widehat{\nabla}^{TB} \right)
}{
\Td \left( \widehat{\nabla}^{\nu} \right)}
\wedge \widetilde{C}
\wedge \omega(\nabla^E) \equiv \\
& \int_{X^\prime} 
\frac{(\pi^\prime)^* \Td \left( \widehat{\nabla}^{TB^\prime} \right)
}{
\Td \left( \widehat{\nabla}^{\nu^\prime} \right)}
\wedge d\widetilde{C}^\prime
\wedge \omega(\nabla^{E^\prime}) \equiv \notag \\
& 
\int_{X^\prime} 
\frac{(\pi^\prime)^* \Td \left( \widehat{\nabla}^{TB^\prime} \right)
}{
\Td \left( \widehat{\nabla}^{\nu^\prime} \right)}
\wedge 
\left(
(\iota^\prime)^* 
\Td \left( \widehat{\nabla}^{T^V (S^N \times B^\prime)} \right)
- \Td \left( \widehat{\nabla}^{T^VX^\prime} \right) \wedge
\Td \left( \widehat{\nabla}^{\nu^\prime} \right)
\right)
\wedge \omega(\nabla^{E^\prime}) \equiv \notag \\
& \int_{X^\prime} 
\Td \left( \widehat{\nabla}^{TX^\prime} \right)
\wedge \omega(\nabla^{E^\prime}) -
\int_{X^\prime} 
(\pi^\prime)^* \Td \left( \widehat{\nabla}^{TB^\prime} \right)
\wedge \Td \left( \widehat{\nabla}^{T^VX^\prime} \right) 
\wedge \omega(\nabla^{E^\prime}).
\notag
 \end{align}
The lemma follows from combining equations (\ref{first})-(\ref{last}).
 \end{proof}

Continuing with the proof of Theorem \ref{maintheorem},
we apply Lemma \ref{variation} with $B^\prime = [0,1] \times B$,
so $\partial B^\prime = B_1 - B_0$.
If $p : [0,1] \times B \rightarrow B$ is the projection map then we take
all of the structures on $B^\prime$ to be pullbacks under $p$
of the corresponding structures on 
$B$, except for the horizontal distribution $T^H X^\prime$.
Note that the property of having vector bundle kernel is
independent of the choice of horizontal distribution. 
We choose $T^HX^\prime$ to equal $p^* T^H X$ near
$\{1\} \times B$, and to equal
$p^* (T^H_{prod}(S^N \times B)) \big|_{X}$ near $\{0\} \times B$.
Then Lemma \ref{variation} implies the computation of 
$C_1 - C_2$ for $B_1$ equals that for $B_0$. Thus
without loss of generality, we can assume that
$T^H X = (T^H_{prod}(S^N \times B)) \big|_{X}$. 
In this case, $\widetilde{C}$ vanishes from
Lemma \ref{vanishes}.

Next, we apply Lemma \ref{variation} with $B^\prime = [0,1] \times B$
and with all of the structures on $B^\prime$ pulling back from $B$,
except for the Riemannian metrics. 
Given $\epsilon > 0$, let $\rho : [0,1] \rightarrow \RR^+$
be a smooth function which is $\epsilon$ near $\{0\}$ and which
is $1$ near $\{1\}$. Multiply the fiberwise metrics for the Riemannian
structures $\pi^\prime : [0,1] \times X \rightarrow [0,1] \times B$ and
$\pi_{prod}^\prime : [0,1] \times S^N \times B \rightarrow [0,1] \times B$
by a factor $\rho(t)$, for $t \in [0,1]$. 
By doing so, we do not alter the property of having
vector bundle kernel. 
Then Lemma \ref{variation} implies the computation of 
$C_1 - C_2$ for $B_1$ equals that for $B_0$. That is,
$C_1 - C_2$ is unchanged after
scaling the metrics by $\epsilon$.

Hence it suffices to compute $C_1 - C_2$ in the limit when 
$\epsilon \rightarrow 0$. In this case, it is
known \cite[Theorem 0.1]{Dai}, \cite[Section 4]{Lott} that 
 \begin{equation}\label{eq:90}
\lim_{\epsilon \rightarrow 0}
\overline{\eta} \left( D^{X,E} \right) \equiv 
\overline{\eta} \left( D^{B,{\Ker(D^V)_+}} \right) -
\overline{\eta} \left( D^{B,{\Ker(D^V)_-}} \right) + 
u^{\frac{\dim(X)+1}{2}}
\int_B \Td \left( \widehat{\nabla}^{TB} \right) \wedge
\widetilde{\eta}
 \end{equation}
in $\RR/\ZZ$. Thus $C_1 - C_2 \equiv 0$. The theorem follows.
 \end{proof}

 \begin{corollary}
 \begin{enumerate}
 \item The assignment ${\mathcal E} \rightarrow \ind^{an}({\mathcal E})$
factors through a map $\cK^0(X) \rightarrow \cK^{-n}(B)$.
 \item The map $\ind^{top} : \cK^0(X) \rightarrow \cK^{-n}(B)$
is independent of the choice of embedding $\iota$.
 \end{enumerate}
 \end{corollary}
 \begin{proof}
Part (1) follows from Theorem \ref{maintheorem} and the fact that
the assignment ${\mathcal E} \rightarrow \ind^{top}({\mathcal E})$
factors through a map $\cK^0(X) \rightarrow \cK^{-n}(B)$.
Part (2) follows from Theorem \ref{maintheorem} and the fact that
$\ind^{an}$ is independent of the choice of embedding $\iota$.
 \end{proof}

\section{Index theorem: general case}\label{sec:6}

In this section 
we complete the proof of the differential $K$-theory index theorem.  

In general, the kernels of a family of Dirac operators need not form a vector
bundle. In such a case, 
the basic idea is to perform a finite-rank perturbation of
the operators, in order to effectively 
reduce to the case of vector bundle kernel.
One way to do this, used in \cite{Atiyah-Singer} is to enlarge the domain of
$(D^V)_+$ by the sections of a 
trivial bundle over $B$, in order to make a finite rank
change so that $(D^V)_+$ becomes surjective; this implies
vector bundle kernel.  We instead follow the method of
\cite[Section 5]{Lott}, which uses a lemma of
Mischenko-Fomenko (Lemma~\ref{thm:99}) to find
a finite rank subbundle of the infinite rank bundle $\sH$
which captures the index. Adding on this finite rank subbundle,
with the opposite grading, allows one to alter the operator to make
it invertible.

An additional technical issue arises in trying to construct the eta form.
We want to make the $D^V$-term in the integrand 
invertible for large $s$, but we want to keep the small-$s$
asymptotics of the unperturbed Bismut superconnection. As in
\cite[Section 5]{Lott}, we use the trick of ``time-varying $\eta$-forms'', 
which originated in \cite{Melrose-Piazza}.

In Subsection \ref{recall} we recall some facts about ``time-varying
$\eta$-invariants'' and ``time-varying $\eta$-forms''.  In Subsection
\ref{subsec:6.12} we review the Mischenko-Fomenko result and and construct
the analytic pushforward in the general case (Definition~\ref{analindex2}).
After these preliminaries, in Subsection \ref{general} we prove the general
index theorem along the lines of the argument in the previous section.
Finally, in Subsection~\ref{subsec:6.1} we use the limit theorem in the
appendix to extend the theorem to proper fiber bundles with arbitrary base.

\subsection{Eta invariants and eta forms} \label{recall}

We first review some material from \cite{Lott} about eta invariants and
eta forms, which is an adaptation of
\cite{Bismut-Cheeger} to the time-varying case. 

Let $B$ be a closed odd-dimensional manifold.
Let ${\mathcal D}$ be a smooth $1$-parameter family of
first-order self-adjoint elliptic pseudodifferential operators
$D(s)$ on $B$, such that
 \begin{itemize}
 \item There are a $\delta > 0$ and a 
first-order self-adjoint elliptic pseudodifferential operator $D_0$ on 
$X$ such that for $s \in (0, \delta)$, we have $D(s) = s D_0$.
 \item There are a $\Delta > 0$ and a 
first-order self-adjoint elliptic pseudodifferential operator $D_\infty$ on 
$X$ such that for $s \in (\Delta, \infty)$, we have $D(s) = 
s D_\infty$.
 \end{itemize}

For $z \in \CC$ with $Re(z) >>0$, put
 \begin{equation}
\eta({\mathcal D})(z) = \frac{2}{\sqrt{\pi}} \int_0^\infty
s^z \Tr \left( \frac{dD(s)}{ds} e^{-D(s)^2} \right) ds.
 \end{equation}

 \begin{lemma} \cite[Lemma 2]{Lott}
$\eta({\mathcal D})(z)$ extends to a meromorphic function on $\CC$
which is holomorphic near $z=0$.
 \end{lemma}

\noindent
 Define the eta-invariant of ${\mathcal D}$ by 
 \begin{equation}
\eta({\mathcal D}) = \eta({\mathcal D})(0)
 \end{equation}
and define the reduced eta-invariant of 
${\mathcal D}$ by 
 \begin{equation}
\overline{\eta}({\mathcal D}) = 
\frac{\eta({\mathcal D}) + \dim(\Ker(D_\infty))}{2}
\, \, \, \, (mod \, \, \ZZ).
 \end{equation}

 \begin{lemma} \cite[Lemma 3]{Lott} \label{onlydepends}
$\eta({\mathcal D})$ only depends on $D_0$ and $D_\infty$, and
$\overline{\eta}({\mathcal D})$ only depends on $D_0$.
 \end{lemma}

Now suppose that $B$ additionally is a Riemannian $\spinc$-manifold,
equipped with a $\spinc$-connection 
$\widehat{\nabla}^{TB}$ on the spinor bundle $\Sp^B$.
Let $E$ be a $\ZZ/2\ZZ$-graded vector bundle over $B$.
We think of $[E]$ as defining an element of $K^{-n}(B)$, for
some even $n$.
If $A$ is a superconnection on $E$ and $s \in \RR^+$, let
$A_s$ denote the result of multiplying each factor of $u$ in
$A$ by $s^2$.
 
Let ${\mathcal A} = \{A(s)\}_{s \ge 0}$ be a smooth $1$-parameter family
of superconnections on $E$ such that
 \begin{itemize}
 \item There are a $\delta > 0$ and a superconnection $A_0$ on
$E$ such that for $s \in (0, \delta)$, we have $A(s) = (A_0)_s$.
 \item There are a $\Delta > 0$ and a superconnection $A_\infty$ on
$E$ such that for $s \in (\Delta, \infty)$, we have $A(s) = (A_\infty)_s$.
 \end{itemize}

Suppose that $A_\infty$ is invertible. For $z \in \CC$, $Re(z) >> 0$, define
$\widetilde{\eta}({\mathcal A})(z) \in \Omega(B;
\R)^{-n-1}/\Image(d)$ by
 \begin{equation}\label{eq:99}
\widetilde{\eta}({\mathcal A})(z) =
u^{- {n}/{2}} R_u \int_0^\infty z^s \str \left( u^{-1} \frac{dA(s)}{ds}
e^{- u^{-1} A(s)^2} \right) ds.
 \end{equation}

 \begin{lemma} \cite[Lemma 4]{Lott} 
$\widetilde{\eta}({\mathcal A})(z)$ extends to a meromorphic
vector-valued function on $\CC$ with simple poles.  Its residue
at zero vanishes in $\Omega(B; \R)^{-n-1}/\Image(d)$.
 \end{lemma} 

\noindent
 Define the eta form of ${\mathcal A}$ by
 \begin{equation}
\widetilde{\eta}({\mathcal A}) = \widetilde{\eta}({\mathcal A})(0).
 \end{equation}
As in Lemma \ref{onlydepends}, $\widetilde{\eta}({\mathcal A})$ only
depends on $A_0$ and $A_\infty$.

Given a superconnection
$A$ on $E$, let $\overline{A}$ denote the associated first-order differential
operator~\cite[Section 3.3]{Berline-Getzler-Vergne}. It is the essentially
self-adjoint operator on $C^\infty(E \otimes \Sp^B)$ obtained by
replacing the Grassmann variables in $A$ by
Clifford variables and replacing $u$ by $1$.
Now given a family
${\mathcal A}$ of superconnections as above and a parameter $\epsilon > 0$,
define a family of operators ${\mathcal D}^{(\epsilon)}$ by
 \begin{equation}
D^{(\epsilon)}(s) = \overline{A(s)_{\epsilon^{-1}}}.
 \end{equation}
(In the fiber bundle situation, this corresponds to multiplying the
fiber lengths by a factor of $\epsilon$.
The paper \cite{Bismut-Cheeger} instead expands the base,
but the two approaches are equivalent.)
Let $\eta({\mathcal D}^{(\epsilon)})$ be the corresponding eta invariant. 
Then a generalization of \cite[(A.1.7)]{Bismut-Cheeger} says that
 \begin{equation}
\lim_{\epsilon \rightarrow 0} \overline{\eta}({\mathcal D}^{(\epsilon)}) =
u^{\frac{\dim(B)+n+1}{2}}
\int_B \Td(\widehat{\nabla}^{TB}) \wedge \widetilde{\eta}({\mathcal A})
\, \, \, \, (mod \, \, \ZZ).
 \end{equation}

\subsection{Analytic pushforward}\label{subsec:6.12}

We continue with the setup of Section \ref{sec:3}, namely a family of
Dirac-type operators, except that we no longer assume that $\Ker(D^V)$ forms
a smooth vector bundle on $B$.  In order to deal with this more general
situation, we will use a perturbation argument, following the approach of
\cite[Section 5]{Lott}.  For this, we need to assume that $B$~is compact.

We first recall a technical lemma of
Mischenko-Fomenko, along with its proof.

 \begin{lemma}[\cite{Mischenko-Fomenko}]\label{thm:99}
Suppose that $B$ is compact.
Then there are finite-dimensional vector subbundles $L_\pm \subset
\sH_\pm$ and complementary closed subbundles
$K_\pm \subset \sH_\pm$, i.e.
 \begin{equation} \label{directsum}
\sH_\pm = K_\pm \oplus L_\pm,
 \end{equation}
so that
$D^V_+ \in \Hom(\sH_+,\sH_-)$ is block diagonal as a map
 \begin{equation}
D^V_+ : K_+ \oplus L_+ \rightarrow K_- \oplus L_-
 \end{equation}
and $D^V_+$ restricts to an isomorphism between
$K_+$ and $K_-$.
(Note that $K_\pm$ may not be orthogonal to $L_\pm$.)
 \end{lemma}
 \begin{proof}
This is proved in \cite[Lemma 2.2]{Mischenko-Fomenko}.
For completeness, we sketch the argument.  One first finds
finite-dimensional vector subbundles $L^\prime_\pm \subset
\sH_\pm$ so that the projected map $(L^\prime_+)^\perp 
\stackrel{D^V_+}{\rightarrow} \sH_- \rightarrow (L^\prime_-)^\perp$
is an isomorphism. With respect to the orthogonal decomposition
$\sH_\pm = (L^\prime_\pm)^\perp \oplus L^\prime_\pm$, write
 \begin{equation}
D^V_+ =
 \begin{pmatrix}
A & B \\
C & D
 \end{pmatrix}
 \end{equation}
where $A : (L^\prime_+)^\perp 
\rightarrow (L^\prime_-)^\perp$ is an isomorphism.
Set 
 \begin{align}
K_+ & = (L^\prime_+)^\perp, \\
L_+ & = 
\Image((-A^{-1}B + I) : L^\prime_+ \rightarrow \sH_+), \notag \\
K_- & = 
\Image((I + CA^{-1}) : (L^\prime_-)^\perp \rightarrow \sH_-), \notag \\
L_- & = L^\prime_-. \notag
 \end{align}
This proves the lemma.
 \end{proof}

Let $i_- : L_- \rightarrow \sH_-$ be the inclusion map and let $p_+ : \sH_+
\rightarrow L_+$ be the projection map coming from (\ref{directsum}).  Put
$\widetilde{\sH}_\pm = \sH_\pm \oplus L_\mp$.  Given $\alpha \in \CC$, define
$\widetilde{D}^V_+(\alpha) \in \Hom(\widetilde{\sH}_+, \widetilde{\sH}_-)$ by
the matrix
 \begin{equation}
\widetilde{D}^V_+(\alpha) = 
 \begin{pmatrix}
D^V_+ & \alpha i_- \\
\alpha p_+ & 0
 \end{pmatrix}.
 \end{equation}
That is,
 \begin{equation}
\widetilde{D}^V_+(\alpha) (h_+ \oplus l_-) =
(D^V_+ h_+ + \alpha i_- l_- \oplus \alpha p_+ h_+).
 \end{equation} 

 \begin{lemma} \label{invertible}
If $\alpha \neq 0$ then 
$\widetilde{D}^V_+(\alpha)$ is invertible.
 \end{lemma}
 \begin{proof}
Suppose that $\widetilde{D}^V_+(\alpha) (h_+ \oplus l_-) = 0$.
As $p_+ h_+ = 0$, we know that $h_+ \in K_+$.
Then $D^V_+ h_+ \in K_-$. As $D^V_+ h_+ + \alpha i_- l_- = 0$,
we conclude that $D^V_+ h_+ = 0$ and $l_- = 0$. As $D^V_+$ is
injective on $K_+$, it follows that $h_+ = 0$. Hence
$\widetilde{D}^V_+(\alpha)$ is injective.

Now suppose that $h^\prime_- \oplus l^\prime_+ \in \widetilde{\sH}_-$.
With respect to (\ref{directsum}), write $h^\prime_- = k^\prime_- +
l^\prime_-$.
Put 
 \begin{align}
h_+ = & \left( D^V_+ \big|_{K_+} \right)^{-1} k^\prime_-
+ \alpha^{-1} l^\prime_+,  \\
l_- = & \alpha^{-1} l^\prime_- -
\alpha^{-2} D^V_+ l^\prime_+. \notag
 \end{align}
One can check that
$\widetilde{D}^V_+ ( h_+ \oplus l_- ) =  h^\prime_- \oplus l^\prime_+$.
Thus $\widetilde{D}^V_+(\alpha)$ is surjective.
 \end{proof}

Define $\widetilde{D}^V(\alpha) \in \End(\widetilde{\sH})$ by
$\widetilde{D}^V(\alpha) = \widetilde{D}^V_+(\alpha) \oplus \left(
\widetilde{D}^V_+(\alpha) \right)^*$, an essentially self-adjoint operator on
each fiber $\widetilde{\sH}_b$. As $\widetilde{D}^V(\alpha)$ is a finite-rank
perturbation of $D^V \oplus I_L$, 
and $(I + (D^V)^2)^{-1}$ is compact on each fiber
$\sH_b$, it follows that
$(I + (\widetilde{D}^V(\alpha))^2)^{-1}$ is compact on each fiber
$\widetilde{\sH}_b$.
Lemma \ref{invertible} now implies that if
$\alpha \neq 0$ then $\left( \widetilde{D}^V(\alpha) \right)^2$ has strictly
positive spectrum.  

Give $L$ the projected Hermitian inner product $h^L$ and projected
compatible connection $\nabla^L$ from $\sH$.  Put
$\nabla^{\widetilde{\sH}_\pm} = \nabla^{{\sH}_\pm} \oplus
\nabla^{L_\mp}$.  Let $\alpha : [0, \infty) \rightarrow [0,1]$ be a smooth
function for which $\alpha(s) = 0$ if $s$ is near $0$, and $\alpha(s)=1$ if
$s \ge 1$.  
We view $[L]$~as defining an element of $K^{-n}(B)$.

For $s > 0$, define a superconnection $\widetilde{A}_s$ by
   \begin{equation} \label{perturbedsuper}
     \widetilde{A}_s= s u^{1/2} \widetilde{D}^V(\alpha(s)) + 
\nabla^{\widetilde{\sH}}  - s^{-1} u^{-1/2}\frac{c(T)}{4}. 
   \end{equation}
Then
 \begin{equation}
 \lim_{s \rightarrow 0} u^{-n/2}R_u \STr \left( e^{- u^{-1}
\widetilde{A}_s^2} \right) = \pi_*(\omega(\nabla^E)) - \omega(\nabla^L),
 \end{equation}
while
 \begin{equation}
 \lim_{s \rightarrow \infty} u^{-n/2} R_u \STr \left( e^{- u^{-1} 
\widetilde{A}_s^2} \right) =
0.
 \end{equation}
Note that unlike in Subsection \ref{recall}, we do not have to
use zeta-function regularization because when $s \rightarrow 0$, the
$\sH$ and $L$ factors in $\widetilde{\sH}$ decouple and so we are
reduced to the short-time asymptotics of the Bismut superconnection
(\ref{short}).

Put 
 \begin{equation}
 \tilde{\eta} = u^{-n/2}R_u \int_0^\infty \STr \left( u^{-1} 
\frac{d\widetilde{A}_s}{ds}
e^{- u^{-1} \widetilde{A}_s^2} \right) ds \in \Omega (B;\R)^{-n-1}/\Image(d).
 \end{equation}
It is independent of the particular choice of the function $\alpha$. Also,
 \begin{equation}\label{eq:98}
d\tilde{\eta} = \pi_*(\omega(\nabla^E)) - \omega(\nabla^L).
 \end{equation}

 \begin{definition}\label{analindex2}
Given a generator ${\mathcal E} = \left( E, h^E, \nabla^E, \phi \right)$ for
$\cK^0(X)$, we define the \emph{analytic index}
 \begin{equation} \label{submerformula2}
 \ind^{an}({\mathcal E}) =
\left( L, h^{L}, 
\nabla^{L}, \pi_*(\phi) + \widetilde{\eta}
\right)
 \end{equation}
as an element of $\cK^{-n}(B)$, where $L$~is chosen as in Lemma~\ref{thm:99}.
 \end{definition}
Given a generator ${\mathcal E}$ of $\WFcK^0(X)$, we define
$\ind^{an}({\mathcal E}) \in \cK^{-n}(B)$ by the same formula
(\ref{submerformula2}).  We prove in Corollary~\ref{indie} that this
definition is independent of the choice of~$L$.

 \begin{lemma} \label{omegapush2}
If ${\mathcal E}$ is a generator for $\WFcK^0(X)$ then
$\omega(\ind^{an}({\mathcal E})) = \pi_*(\omega({\mathcal E}))$
in $\Omega(B; \R)^{-n}$.
 \end{lemma}
 \begin{proof}
We have
 \begin{equation}
\omega(\ind^{an}({\mathcal E})) =
\omega(\nabla^L) + d(\pi_*(\phi) + \widetilde{\eta}) =
\pi_*(\omega(\nabla^E) + d\phi) = 
\pi_*(\omega({\mathcal E})),
 \end{equation}
which proves the lemma.
 \end{proof}

\subsection{General index theorem} \label{general}

Continuing with the assumptions of the previous subsection, suppose that $B$
is a closed odd-dimensional Riemannian manifold with a
$\spinc$-structure. Let $\widehat{\nabla}^{TB}$ be a $\spinc$-connection on
$\Sp^B$.  Combining with the Riemannian structure on $\pi$ and the
differential $\spinc$-structure on $\pi$, we obtain a Riemannian metric
$g^{TX}$ on $X$ and a $\spinc$-connection $\widehat{\nabla}^{TX}$ on~ $\Sp^X$.

As in~\ref{recall}, given a parameter $\epsilon > 0$, we define a family of
pseudodifferential operators ${\mathcal D}^{(\epsilon)}$ (living on $X$) by
 \begin{equation}
D^{(\epsilon)}(s) = \overline{(\widetilde{A}_s)_{\epsilon^{-1}}}.
 \end{equation}
Then the family ${\mathcal D}^{(\epsilon)}$ satisfies the formalism of~
\ref{recall}.

To identify the operators $D_0$ and $D_\infty$ corresponding to the family
${\mathcal D}^{(\epsilon)}$, let $X_\epsilon$ denote the Riemannian structure
on $X$ coming from multiplying $g^{TX}$ in the vertical direction by
$\epsilon$.  If $s$ is near zero then $\alpha(s)$ vanishes and the
superconnection $\widetilde{A}_s$ of (\ref{perturbedsuper}) just becomes the
direct sum of the Bismut superconnection on $\sH$ and the connection on
$\Pi L$, the latter being $L$ with the opposite
grading.  Therefore, $D_0 = D^{X_\epsilon,E} \oplus D^{B, \Pi L}$ is the sum
of ordinary Dirac-type operators on $X_\epsilon$ and $B$.  On the other hand,
if $s > \Delta$ then $\alpha(s) = 1$ and $\widetilde{D}^V(\alpha(s))$ is
$L^2$-invertible.  From (\ref{perturbedsuper}), $D_\infty$ is the
Dirac operator on~$B$ coupled to the superconnection
$\epsilon^{-1} \widetilde{D}^V(1) + \nabla^{\sH} - \epsilon \frac{c(T)}{4}$.
If $\epsilon$ is small then the term $\epsilon^{-1} \widetilde{D}^V(1)$
dominates when computing the spectrum of $D_\infty$, so $D_\infty$ is an
invertible first-order self-adjoint elliptic pseudodifferential operator on
the disjoint union $X \sqcup B$.

Let $\overline{\eta}({\mathcal
D}^{(\epsilon)})$ be the reduced eta invariant of the rescaled family
${\mathcal D}^{(\epsilon)}$.  As in Lemma \ref{onlydepends},
$\overline{\eta}({\mathcal D}^{(\epsilon )})$ only
depends on $D_0$.  It follows that 
 \begin{equation} \label{diff}
\overline{\eta}({\mathcal D}^{(\epsilon)}) =
\overline{\eta}(D^{X_\epsilon,E}) +
\overline{\eta}(D^{B, \Pi L}) = 
\overline{\eta}(D^{X_\epsilon,E}) -
\overline{\eta}(D^{B, L}).
 \end{equation}
A generalization of \cite[Theorem 4.35]{Bismut-Cheeger} says that
 \begin{equation} \label{neweta}
\lim_{\epsilon \rightarrow 0} 
\overline{\eta}({\mathcal D}^{(\epsilon)}) =
u^{\frac{\dim(X)+1}{2}}
\int_B \Td(\widehat{\nabla}^{TB}) \wedge \widetilde{\eta}
(\widetilde{\mathcal A})
\, \, \, \, (mod \, \, \ZZ).
 \end{equation}

We can now go through the proof of Theorem \ref{maintheorem}, using
\eqref{eq:98} and (\ref{diff})+(\ref{neweta}) in place of~\eqref{eq:2.10}
and~\eqref{eq:90}, respectively, to derive the following result.

 \begin{theorem} \label{maintheorem2}
 Suppose that $\pi : X \rightarrow B$ is a smooth fiber bundle with compact
fibers of even dimension. Suppose that $\pi$ is equipped with a Riemannian
structure and a differential $\spinc$-structure.  Assume that $X$~is compact.
Then for all ${\mathcal E} \in \cK^0(X)$, we have $\ind^{an}({\mathcal E}) =
\ind^{top}({\mathcal E})$.
 \end{theorem}

 \begin{corollary} \label{indie}
 \begin{enumerate}
 \item The homomorphism $\ind^{top} : \cK^0(X) \rightarrow \cK^{-n}(B)$
is independent of the choice of embedding $\iota$.
 \item The assignment ${\mathcal E} \rightarrow \ind^{an}({\mathcal E})$
factors through a homomorphism $\ind^{an}\:\cK^0(X) \rightarrow \cK^{-n}(B)$.
\item The map $\ind^{an} : \cK^0(X) \rightarrow \cK^{-n}(B)$
is independent of the choice of finite-dimensional vector subbundle
$L_\pm$.\
\item  If $D^V$ has vector bundle kernel then the
analytic index defined in Definition \ref{analindex} equals the
analytic index defined in Definition \ref{analindex2}.
 \end{enumerate}
 \end{corollary}

Aside from its intrinsic interest, the next proposition will be
used in Section~\ref{sec:7}.

\begin{proposition}
The following diagrams commute :
  \begin{equation} \label{comm1}
     \xymatrix{ 0 \ar[r]& \dfrac{\Omega (X;\R)^{-1}}{\Omega
     (X;\R)_K^{-1}}\ar[d]^{\pi_*}\ar[r]^{\quad
           j}&\cK^0(X)\ar[d]^{\ind^{an}}\ar[r]^{c\quad} &
           K^0(X;\ZZ)\ar[d]^{\ind^{an}}\ar[r] &0 \\ 0 \ar[r]&
           \dfrac{\Omega (B;\R)^{-n-1}}{\Omega
     (B;\R)_K^{-n-1}}\ar[r]^{\quad j}&\cK^{-n}(B)\ar[r]^{c\quad} &
           K^{-n}(B;\ZZ)\ar[r]& 0 } 
  \end{equation}
and
  \begin{equation} \label{comm2}
     \xymatrix{ 0 \ar[r]&
           K^{-1}(X;\RZ)\ar[d]^{\ind^{an}}\ar[r]^{\quad
           j}&\cK^0(X)\ar[d]^{\ind^{an}}\ar[r]^{\omega\quad} &
           \Omega(X;\R)_K^0\ar[d]^{\pi_*}\ar[r] &0 \\
     0 \ar[r]& 
           K^{-n-1}(B;\RZ)\ar[r]^{\quad j}&\cK^{-n}(B)\ar[r]^{\omega\quad} &
           \Omega(B;\R)_K^{-n}\ar[r]& 0. } 
  \end{equation}
\end{proposition}
\begin{proof}
The commuting of (\ref{comm1}) follows immediately from the
definition of $\ind^{an}$.  The
left-hand square of (\ref{comm2}) commutes from the definition of
$\ind^{an} : K^{-1}(X; \RR/\ZZ) \rightarrow K^{-n-1}(B; \RR/\ZZ)$ in
\cite{Lott}. The right-hand square of (\ref{comm2}) commutes
from Lemma \ref{omegapush2}.
\end{proof}

 \subsection{Noncompact base}\label{subsec:6.1}

We use the limit theorem in the appendix to define the index maps for proper
submersions and extend Theorem~\ref{maintheorem2}.  

Suppose that $\pi : X \rightarrow B$ is a
 proper submersion of relative dimension $n$, 
with $n$ even.
Suppose that $\pi$ is equipped with a Riemannian
structure and a differential $\spinc$-structure. 
Let $B_1\subset B_2\subset \cdots$ be an exhaustion of~$B$ by compact
codimension-zero submanifolds-with-boundary.
From Theorem~\ref{thm:10}, there is an isomorphism
  \begin{equation}\label{eq:121}
      \cK^{-n}(B) \cong \varprojlim_i\,\cK^{-n}(B_i) .
  \end{equation}

Put $X_i=\pi \inv (B_i)$.
Given ${\mathcal E} \in \cK^0(X)$, we can define 
$\ind^{top}({\mathcal E} \big|_{X_i}) \in \cK^{-n}(B_i)$ as in
Section \ref{sec:4}, after making a choice of 
embedding $\iota : X_i \rightarrow S^N \times B_i$. Clearly if
$B_j \subset B_i$ then $\ind^{top}({\mathcal E} \big|_{X_j})$,
as defined using the restriction of $\iota$ to $X_j$, is
the restriction of $\ind^{top}({\mathcal E} \big|_{X_i})$ to
$B_j$. Using the fact from Corollary \ref{indie} that
$\ind^{top}({\mathcal E} \big|_{X_j})$ is independent of the
choice of embedding, it follows that we have defined
a topological index $\ind^{top}({\mathcal E})$ in
$\cK^{-n}(B) \cong \varprojlim_i\,\cK^{-n}(B_i)$.

Similarly, we can define
$\ind^{an}({\mathcal E} \big|_{X_i}) \in \cK^{-n}(B_i)$ as in the
earlier part of this section,
after making a choice of the finite-dimensional vector subbundle
$L_\pm$ over $B_i$. Clearly if
$B_j \subset B_i$ then $\ind^{an}({\mathcal E} \big|_{X_j})$,
as defined using the restriction of $L_\pm$ to $B_j$, is
the restriction of $\ind^{an}({\mathcal E} \big|_{X_i})$ to
$B_j$. Using the fact from Corollary \ref{indie} that
$\ind^{an}({\mathcal E} \big|_{X_j})$ is independent of the
choice of vector subbundle, it follows that we have defined
an analytic index $\ind^{an}({\mathcal E})$ in
$\cK^{-n}(B) \cong \varprojlim_i\,\cK^{-n}(B_i)$.
 
        \begin{theorem}[]\label{thm:13}
  Suppose that $\pi : X \rightarrow B$ is a 
 proper submersion with even relative dimension. 
Suppose that $\pi$ is equipped with a Riemannian
structure and a differential $\spinc$-structure.  Then for all ${\mathcal E}
\in \cK^0(X)$, we have $\ind^{an}({\mathcal E}) = \ind^{top}({\mathcal E})$.
        \end{theorem}

        \begin{proof}
By Theorem~\ref{maintheorem2}, we know that for each $i$,
$\ind^{an}({\mathcal E} \big|_{X_i}) = \ind^{top}({\mathcal E} \big|_{X_i})$
in $\cK^{-n}(B_i)$. 
Along with (\ref{eq:121}), the theorem follows.
        \end{proof}

\section{Relationships to earlier work}\label{sec:7}

In this section we illustrate how our main index theorem relates to other
work in the geometric index theory of Dirac operators.  We first treat the
determinant line bundle, using the holonomy theorem of~\cite{Bismut-Freed} to
show that the determinant of the analytic pushforward is the determinant
bundle.  (For a different approach to this question
see~\cite[Chapter~9]{Klonoff}).  As a consequence, our main theorem gives a
``topological'' construction of the determinant line bundle, equipped with
its connection, up to isomorphism.

Next, in Subsection~\ref{subsec:7.9} we remark that when specialized to
$\RR/\ZZ$-valued K-theory, our theorem implies that
the topological pushforward constructed in Section 
\ref{sec:2}  coincides with the topological pushforward constructed
from generalized cohomology theory.

The Chern character map, from topological K-theory to rational
cohomology, has a differential refinement, going from
differential K-theory to rational differential cohomology.
 In Subsection \ref{rational} we apply
this refined Chern character map to our main theorem
and recover the Riemann-Roch formula of~\cite{Bunke-Schick}.

Finally, under certain assumptions,
there are geometric invariants of families of Dirac operators which
live in higher-degree integral differential cohomology \cite{Lott2}.
In Subsection \ref{Deligne} we point out that our index theorem computes them 
in terms of the
topological pushforward.

\subsection{Determinant line bundle} \label{detline}

There is a map $\Det$ from $\cK^0(X)$ to isomorphism classes of line bundles
on $X$, equipped with a Hermitian metric and a compatible connection.
(The latter group may be identified with the 
integral
differential cohomology group~$\breve H^2(X)$.)  
Given a generator ${\mathcal E} = (E, h^E,
\nabla^E, \phi)$ for $\cK^0(X)$, its image $\Det({\mathcal E})$ is represented
by the line bundle $\Lambda^{max}(E)$, equipped with the Hermitian metric
$h^{\Lambda^{max}(E)}$ and the connection \begin{equation}
\nabla^{\Det({\mathcal E})} = \nabla^{\Lambda^{max}(E)} - 2 \pi i \phi_{(1)}.
\end{equation}
Here $\phi_{(1)} \in \Omega^1(X)/\Image(d)$ denotes $u$ times the
component of $\phi \in \Omega(X; \R)^{-1}/\Image(d)$ in
$u^{-1} \Omega^1(X)/\Image(d)$.
Note that changing
a particular representative $\widehat{\phi} \in \Omega^1(X)$ for $\phi$ by an
exact form $df$ amounts to acting on the connection
$\nabla^{\Lambda^{max}(E)} - 2 \pi i \widehat{\phi}$ by a gauge
transformation $g = e^{2 \pi i f}$. 

Suppose that $\pi : X \rightarrow B$ is a compact fiber
bundle with fibers of even dimension  $n$,
endowed with a Riemannian structure and
a differential $\spinc$ structure.  If ${\mathcal E}$ is a generator for
$\cK^0(X)$ of the form $(E, h^E, \nabla^E, 0)$ then there is a corresponding
determinant line bundle $\Det_{\Quillen}$ on $B$, which is equipped with a
Hermitian metric $h^{\Quillen}$ (due to Quillen~\cite{Quillen}) and a
compatible connection~$\nabla^{\Quillen}$ (due to Bismut-Freed
\cite{Bismut-Freed}); see \cite[Chapter 9.7]{Berline-Getzler-Vergne}.
The
construction is analytic;
for example, the construction of
$h^{\Quillen}$ uses $\zeta $-functions built from the spectrum
of the fiberwise Dirac-type operators $D^V$. 

\begin{proposition} \label{determinant}
 After using periodicity to shift $\ind^{an}({\mathcal E}) 
\in \cK^{-n}(B)$ into $\cK^0(B)$, we have that 
$\Det(\ind^{an}({\mathcal E}))$ is the inverse of $\Det_{\Quillen}$.
\end{proposition}
\begin{proof}
Suppose first that $\Ker(D^V)$ is a $\ZZ/2\ZZ$-graded vector
bundle on $B$. Using Definition \ref{analindex} for 
$\ind^{an}({\mathcal E})$, it follows that
\begin{equation}
\Det(\ind^{an}({\mathcal E})) = \Lambda^{max}(\Ker(D^V)_+) \otimes
\left( \Lambda^{max}(\Ker(D^V)_-) \right)^{-1}
\end{equation}
is the inverse of
the determinant line bundle. The connection 
$\nabla^{\Det(\ind^{an}({\mathcal E}))}$ is
\begin{equation}
\nabla^{\Det(\ind^{an}({\mathcal E}))} =
\nabla^{\Lambda^{max}(\Ker(D^V)_+) \otimes
\left( \Lambda^{max}(\Ker(D^V)) \right)^{-1}} - 2\pi i
\widetilde{\eta}_{(1)}.
\end{equation}
From (\ref{etaform}),
\begin{equation}
\widetilde{\eta}_{(1)} = - \frac{1}{2\pi i} \int_0^\infty
\STr \left( D^V [\nabla^{\sH}, D^V] e^{-s^2 (D^V)^2} \right) \, s \, ds.
\end{equation}

In this case of vector bundle kernel, $h^{\Quillen}$ differs from the
$L^2$-metric $h^{L^2}$ by a factor involving the Ray-Singer analytic torsion.
Let $T \in \End(\Det_{\Quillen})$ be multiplication by
$\sqrt{\frac{h^{L^2}}{h^{\Quillen}}}$, so that $T$ is an isometric isomorphism
from $\left( \Det_{\Quillen}, h^{L^2} \right)$ to $\left( \Det_{\Quillen},
h^{\Quillen} \right)$. Then \begin{equation} T^* \left( \Det_{\Quillen},
h^{\Quillen}, \nabla^{\Quillen} \right) = \left( \Det(\ind^{an}({\mathcal
E}))^{-1}, h^{\Det(\ind^{an}({\mathcal E}))^{-1}},
\nabla^{\Det(\ind^{an}({\mathcal E}))^{-1}} \right); \end{equation} see
\cite[Proof of Proposition 9.45]{Berline-Getzler-Vergne}.

If one does not assume vector bundle kernel then a direct proof of
the proposition is trickier for the following reason.  The usual
construction of the determinant line bundle proceeds by making
spectral cuts over suitable open subsets of $B$, constructing 
a line bundle with Hermitian metric and compatible connection over
each open set, and then showing
that these local constructions are compatible on overlaps.  As
$\Det(\ind^{an}({\mathcal E}))$ is only defined up to isomorphism,
it cannot be directly recovered from its restrictions to the
elements of an open cover of $B$. For this reason, a direct
comparison of $\Det(\ind^{an}({\mathcal E}))$ and $\Det_{\Quillen}$ is
somewhat involved.  Instead, we will just compare their holonomies.

Without loss of generality, we can assume that $B$ is connected.  Let $\star
\in B$ be a basepoint.  Let $PB$ denote the smooth maps $c : [0,1]
\rightarrow B$ with $c(0) = \star$. Let $\Omega B$ denote the elements of
$PB$ with $c(1) = \star$.  A unitary connection on a line bundle over $B$
gives rise to a homomorphism $H : \Omega B \rightarrow U(1)$, the holonomy
map. Given $H$, we can construct a line bundle
$\mathcal{L}\to B$ as $\mathcal{L}=(PB \times \CC)/\sim$, where $(c_1, z_1)
\sim (c_2, z_2)$ if $c_1(1) = c_2(1)$ and $z_2 = H(c_2^{-1} \cdot c_1)
z_1$. There is an evident notion of parallel transport on $\mathcal{L}\to B$
and a corresponding unitary connection. There is also a unique Hermitian
inner product on $\mathcal{L}\to B$, up to overall scaling, with which the
unitary connection is compatible.  Hence, given two Hermitian line bundles on
$B$ with compatible connections, if their holonomies are the same then they
are isomorphic.

To compare the holonomies of the connections on $\Det(\ind^{an}({\mathcal
E}))$ and $\Det_{\Quillen}$ around a closed curve, we pull the fiber bundle
structure back to the curve and thereby reduce to the case when $B =
S^1$. From Definition \ref{analindex2}, the holonomy around $S^1$ of
$\nabla^{\Det(\ind^{an}({\mathcal E}))}$ is $e^{2\pi i \int_{S^1}
\widetilde{\eta}_{(1)}}$ times the holonomy around $S^1$ of
$\nabla^{\Lambda^{max} L_+ \otimes (\Lambda^{max} L_-)^{-1}}$.  From
(\ref{diff}) and (\ref{neweta}),
 \begin{equation}  
 e^{2\pi i \int_{S^1} \widetilde{\eta}_{(1)}} = e^{2\pi i (\lim_{\epsilon
\rightarrow 0} \overline{\eta}(D^{X^\epsilon, E})
-\overline{\eta}(D^{S^1,L}))}.
 \end{equation}
By a standard computation, the holonomy around $S^1$ of
$\nabla^{\Lambda^{max} L_+ \otimes (\Lambda^{max} L_-)^{-1}}$ equals $e^{2
\pi i \overline{\eta}(D^{S^1,L})}$. Thus the holonomy around $S^1$ of
$\nabla^{\Det(\ind^{an}({\mathcal E}))}$ is $e^{2\pi i \lim_{\epsilon
\rightarrow 0} \overline{\eta}(D^{X^\epsilon, E})}$.  On the other hand, the
holonomy around $S^1$ of $\nabla^{\Det_{\Quillen}}$ is $e^{- 2\pi i
\lim_{\epsilon \rightarrow 0} \overline{\eta}(D^{X^\epsilon, E})}$
\cite[Theorem 3.16]{Bismut-Freed}.  This proves the proposition.
\end{proof}

As a consequence of Theorem  \ref{thm:13}  
and Proposition
\ref{determinant}, the determinant line bundle with its Hermitian metric and
compatible connection can be constructed up to isomorphism without using any
spectral analysis.  This was also derived in~\cite[Chapter~9]{Klonoff},
though with a different model of differential $K$-theory.

\subsection{$\RR/\ZZ$-index theory}\label{subsec:7.9}

Under the assumptions of Theorem  \ref{thm:13},  
there is a topological index
$\ind^{top} : K^{-1}(X; \RR/\ZZ) \rightarrow K^{-1}(B; \RR/\ZZ)$
which can be constructed from a general procedure in generalized cohomology
theory.

\begin{proposition}
The following diagram commutes :
  \begin{equation} \label{comm4}
     \xymatrix{ 0 \ar[r]&
           K^{-1}(X;\RZ)\ar[d]^{\ind^{top}}\ar[r]^{\quad
           j}&\cK^0(X)\ar[d]^{\ind^{top}}\ar[r]^{\omega\quad} &
           \Omega(X;\R)_K^0\ar[d]^{\pi_*}\ar[r] &0 \\
     0 \ar[r]& 
           K^{-n-1}(B;\RZ)\ar[r]^{\quad j}&\cK^{-n}(B)\ar[r]^{\omega\quad} &
           \Omega(B;\R)_K^{-n}\ar[r]& 0. } 
  \end{equation}
\end{proposition}
\begin{proof}
The right-hand square commutes by Lemma \ref{omegaagrees}. 
The left-hand square commutes from the fact that the diagram
(\ref{comm2}) commutes, along with the facts that the analytic and
topological indices agree in differential K-theory 
(Theorem  \ref{thm:13}),
and in $\RR/\ZZ$-valued K-theory \cite{Lott}.
\end{proof}

\begin{remark}
In a similar vein, one might think that Theorem 
 \ref{thm:13},  along
with the commuting of the
right-hand squares in (\ref{comm1}) and (\ref{comm4}), gives a
new and purely analytic proof of the Atiyah-Singer families index theorem
\cite{Atiyah-Singer}. However, such is
not the case.  The proof of Theorem  \ref{thm:13}  uses
Proposition \ref{etapairing}(2), whose proof uses 
\cite[Theorem (5.3)]{Atiyah-Patodi-SingerIII}, whose proof in turn uses the
Atiyah-Singer families index theorem
\cite[Section 8]{Atiyah-Patodi-SingerIII}.  
\end{remark}

\subsection{Rational index theorem} \label{rational}

For a subring $\Lambda \subset
\RR$ we define~$\R_\Lambda =\Lambda [u,u\inv ]$, analogous to~\eqref{eq:2}.
There is a notion of differential cohomology $\breve{H}(X;\R_\Lambda
)^{\bullet}$, the generalized differential cohomology theory attached to
ordinary cohomology with coefficients in the graded ring~$\RL $.
It fits into exact sequences
  \begin{equation}\label{Hses}
     0 \longrightarrow
     H^{\bullet -1}(X;(\RR/\Lambda)[u,u\inv])\xrightarrow{\;\;i\;\;}
\breve{H}^\bullet (X;\RL)\xrightarrow{\;\;\omega \;\;}
     \Omega(X;\R)_\Lambda^\bullet \longrightarrow 0, 
   \end{equation}
  \begin{equation}\label{Hses2}
      0 \longrightarrow 
      \dfrac{\Omega (X;\R)^{\bullet -1}}{\Omega
      (X;\R)_{\Lambda }^{\bullet -1}}\xrightarrow{\;\;j\;\;}
      \breve{H}(X;\RL)^\bullet \xrightarrow{\;\;c\;\;}
      H^\bullet (X;\RL)\longrightarrow  0,
   \end{equation}
where $\Omega(X;\R)_\Lambda^\bullet$ denotes the closed
$\R$-valued forms on $X$
with periods in $\RL=\Lambda [u,u\inv ]$.

The differential cohomology theory
$\breve{H}(X;\RL)^\bullet $ is essentially the same as the
Cheeger-Simons theory of differential characters \cite{Cheeger-Simons}.
Namely, let $C_k(X)$ and $Z_k(X)$ denote the groups of smooth singular
$k$-chains and $k$-cycles in $X$, respectively.  Then an element of 
$\breve{H}^k(X;\RL)^{l}$ corresponds to a homomorphism $F
: Z_{k-1}(X) \rightarrow u^{\frac{l-k}{2}} \cdot (\RR/\Lambda)$ with the
property that there is some $\alpha \in \Omega(X; \R)^{l}$ so that for all $c
\in C_{k}(X)$, we have $F(\partial c) = \int_c \alpha \mod{u^{\frac{l-k}{2}}
\cdot \Lambda}$.

Given a proper submersion $\pi : X \rightarrow B$ of relative
dimension $n$ which is oriented
in ordinary cohomology, there is an ``integration over the fiber'' map
$\int_{X/B} : \breve{H}(X;\RL)^{\bullet}
\rightarrow \breve{H}(B;\RL)^{\bullet - n}$
\cite[Section 3.4]{Hopkins-Singer}.
In short, if $F$ is a differential character on $X$ and
$z \in Z_*(B)$ then the evaluation of $\int_{X/B}$ on
$z$ is $F(\pi^{-1}(z)) \in (\RR/\Lambda)[u,u\inv]$; see also
\cite{Bunke-Kreck-Schick}.

There is a Chern character $\breve{\ch} : \cK^{\bullet}(X) \rightarrow
\breve{H}(X; \RQ)^{\bullet}$.  When acting on generators of
$\cK^0(X)$ of the form ${\mathcal E} = (E, h^E, \nabla^E, 0)$, the Chern
character $\breve{\ch}({\mathcal E})$ was defined in \cite[\S
2]{Cheeger-Simons} as a differential character.  It then suffices to
additionally define $\breve{\ch}$ on $\Image(j)$,
where $j$~is the map in~\eqref{eq:5}.  For this, we may note that
there is a natural map of the domain of~$j$ in~\eqref{eq:5} to the domain
of~$j$ in~\eqref{Hses2}, since $\Omega (X;\R)^{\bullet -1}_K\subset \Omega
(X;\R)^{\bullet -1}_\QQ $.  Or, in the language of differential
characters, we define the evaluation of $\breve{\ch}(j(\phi))$ on a cycle
$z \in Z_*(X)$ to be $\int_z \phi \mod{\RQ}$.  The Chern character on
differential K-theory was also considered in \cite[Section 6]{Bunke-Schick}.
 
In the proof of the next proposition we will make use of the Chern character 
  \begin{equation}\label{eq:125}
     \ch_*\:K_{\bullet }(X; \ZZ)\to H(X;\RQ)_{\bullet } 
  \end{equation}
on the $K$-homology of a space~$X$.  Recall from the proof of
Theorem~\ref{maintheorem} that every $K$-homology class can be written
as~$u^kf_*[M]$ for some integer~$k$ and some continuous map $f\:M\to X$ of a
$\spinc$ manifold~$M$ into~$X$, where $[M]\in K_q(M; \ZZ)$ is the fundamental
class.  Let $\Td(M)\in H(M;\RQ)^0$ be the Todd class and let $\Td(M)^{\vee}\in
H(M;\RQ)_q$ be its Poincar\'e dual.  Then
  \begin{equation}\label{eq:123}
     \ch_*\bigl(u^kf_*[M] \bigr) = u^kf_*\bigl(\Td(M)^{\vee} \bigr)\in
     H(X;\RQ)_{q+2k}. 
  \end{equation}
The Chern character maps on cohomology and homology are compatible
with the natural pairings, meaning that for $a\in K_\ell (X; \ZZ)$ and 
$\alpha \in K^{\ell }(X; \ZZ)$, we have
  \begin{equation}\label{eq:124}
     \langle a,\alpha \rangle = \langle \ch_*(a),\ch(\alpha ) \rangle. 
  \end{equation}
The Chern character on homology 
is an isomorphism after tensoring~$K_{\bullet }(X; \ZZ)$
with~$\RQ$.

Recall from \cite[\S 2]{Cheeger-Simons} 
that there are characteristic classes in
$\breve{H}(X;\RQ)^{\bullet}$.
In particular, if $W$ is a real vector bundle with a
$\spinc$-structure and $\widehat{\nabla}^W$ is a $\spinc$-connection
on the associated spinor bundle then there is a refined Todd class
$\breve{\Td} \left( \widehat{\nabla}^W \right) 
\in \breve{H}(X;\RQ)^{0}$. 

\begin{proposition}
Let $\pi : X \rightarrow B$ be a proper submersion of relative dimension
$n$, with $n$ even. Suppose that $\pi$ has a Riemannian structure and
a differential $\spinc$-structure. Then for all
${\mathcal E} \in \cK^0(X)$,
\begin{equation} \label{ratindex}
\breve{\ch}(\ind^{top}({\mathcal E})) =
\int_{X/B} \breve{\Td} \left( \widehat{\nabla}^{T^VX} \right) 
\cup \breve{\ch}({\mathcal E})
\in \breve{H}(B; \RQ)^{-n}.
\end{equation}
\end{proposition}

\begin{proof}
Put 
\begin{equation}
\Delta = \breve{\ch}(\ind^{top}({\mathcal E})) -
\int_{X/B} \breve{\Td} \left( \widehat{\nabla}^{T^VX} \right) 
\cup \breve{\ch}({\mathcal E}).
\end{equation}
From Lemma \ref{omegaagrees}, we have $\omega(\Delta) = 0$. Then (\ref{Hses})
implies that $\Delta = i({\mathcal U})$ for some unique
${\mathcal U} \in H(B;(\RR/\QQ)[u,u\inv])^{-n -1}$.  Using the universal
coefficient theorem and the fact that $\RR/\QQ$ is divisible, to show that
${\mathcal U}$ vanishes it suffices to prove the vanishing of its pairings
with $H_*(B; \RQ)$.
We use the fact that the Chern character~\eqref{eq:125} is surjective after
tensoring $K_*(X; \ZZ)$ with~$\RQ$ to argue, 
as in the proof of Theorem~\ref{maintheorem},
that we can reduce to the case when $B$ is a closed odd-dimensional
$\spinc$-manifold and we are evaluating on the Chern character of its
fundamental $K$-homology class.

We equip $B$ with a 
Riemannian metric and a unitary connection $\nabla^{L^B}$ on the
characteristic line bundle $L^B$.
In the rest of this proof, we work modulo $\RQ=\QQ[u,u\inv]$.
From~\eqref{eq:124} and~ \cite[\S 9]{Cheeger-Simons}, 
\begin{equation}
\langle \ch_*[B], {\mathcal U} \rangle \equiv
\overline{\eta}(B; \ind^{top}({\mathcal E})) -
\int_B \int_{X/B} 
\pi^* \breve{\Td} \left( \widehat{\nabla}^{TB} \right) \cup
\breve{\Td} \left( \widehat{\nabla}^{T^VX} \right) 
\cup \breve{\ch}({\mathcal E}).
\end{equation}
From (\ref{C2})-(\ref{C3}),
\begin{align} \label{rest}
\overline{\eta}(B, \ind^{top}({\mathcal E})) \equiv &
u^{-\frac{\dim(X)+1}{2}} \overline{\eta}(D^{X,E}) -
\int_{X} 
\frac{\pi^* \Td \left( \widehat{\nabla}^{TB} \right)
}{
\Td \left( \widehat{\nabla}^{\nu} \right)}
\wedge
\widetilde{C}
\wedge \omega(\nabla^E) + \\
&  \int_X \pi^* \Td \left( \widehat{\nabla}^{TB} \right) \wedge
\Td \left( \widehat{\nabla}^{T^VX} \right) \wedge
\phi \notag \\
\equiv &
\int_X \breve{\Td}(\nabla^{TX}) \cup \breve{\ch}(\nabla^E) -
\int_{X} 
\frac{\pi^* \Td \left( \widehat{\nabla}^{TB} \right)
}{
\Td \left( \widehat{\nabla}^{\nu} \right)}
\wedge
\widetilde{C}
\wedge \omega(\nabla^E) + \notag \\
&  \int_X \pi^* \Td \left( \widehat{\nabla}^{TB} \right) \wedge
\Td \left( \widehat{\nabla}^{T^VX} \right) \wedge
\phi. \notag 
\end{align}
Here $X$ has the induced Riemannian metric from its embedding
in $S^N \times B$.
Thus
\begin{align} \label{rhmess}
\langle \ch_*[B], {\mathcal U} \rangle 
\equiv &
\int_X \breve{\Td}(\nabla^{TX}) \cup \breve{\ch}(\nabla^E) -
\int_{X} 
\frac{\pi^* \Td \left( \widehat{\nabla}^{TB} \right)
}{
\Td \left( \widehat{\nabla}^{\nu} \right)}
\wedge
\widetilde{C}
\wedge \omega(\nabla^E) + \\
&  \int_X \pi^* \Td \left( \widehat{\nabla}^{TB} \right) \wedge
\Td \left( \widehat{\nabla}^{T^VX} \right) \wedge
\phi - \notag \\
& \int_X \breve{\Td} \left( \widehat{\nabla}^{T^VX} \right) 
\cup \breve{\ch}({\mathcal E})  \notag \\
\equiv & 
\int_X \breve{\Td}(\nabla^{TX}) \cup \breve{\ch}(\nabla^E) -
\int_{X} 
\frac{\pi^* \Td \left( \widehat{\nabla}^{TB} \right)
}{
\Td \left( \widehat{\nabla}^{\nu} \right)}
\wedge
\widetilde{C}
\wedge \omega(\nabla^E) - \notag \\
& \int_X
\pi^* \breve{\Td} \left( \widehat{\nabla}^{TB} \right) \cup
 \breve{\Td} \left( \widehat{\nabla}^{T^VX} \right) 
\cup \breve{\ch}(\nabla^E).  \notag
\end{align}

As in the proof of Theorem \ref{maintheorem}, we can deform to the
case $T^HX = (T^H(S^N \times B)) \big|_X$ without changing
the right-hand side of (\ref{rhmess}). In this case,
\begin{equation}
\breve{\Td} \left( \widehat{\nabla}^{TX} \right) = 
\pi^* \breve{\Td} \left( \widehat{\nabla}^{TB} \right) \cup
 \left( \widehat{\nabla}^{T^VX} \right)
\end{equation}
and Lemma \ref{vanishes} says that 
$\widetilde{C} = 0$. The proposition follows.
\end{proof}

\begin{corollary} \label{ratcor}
Let $\pi : X \rightarrow B$ be a proper submersion of relative dimension
$n$, with $n$ even. Suppose that $\pi$ has a Riemannian structure and
a differential $\spinc$-structure. Then for all
${\mathcal E} \in \cK^0(X)$,
\begin{equation} \label{ratindex2}
\breve{\ch}(\ind^{an}({\mathcal E})) =
\int_{X/B} \breve{\Td} \left( \widehat{\nabla}^{T^VX} \right) 
\cup \breve{\ch}({\mathcal E})
\in \breve{H}(B; \RQ)^{-n}.
\end{equation}
\end{corollary}

Corollary \ref{ratcor} was proven by different means in
\cite[Section 6.4]{Bunke-Schick}. 

\subsection{Index in Deligne cohomology} \label{Deligne}

In general, the image of $\breve{\ch}$ lies in the rational differential
cohomology group $\breve{H}(X; \RQ)^{\bullet}$ but not in the
integral differential cohomology group $\breve{H}(X; \ZZ[u,u\inv
])^{\bullet}$. However, in some special cases one gets integral differential
cohomology classes.  Recall that there is a filtration of the usual K-theory
$K^{\bullet}(X; \ZZ) = K^{\bullet}_{(0)}(X; \ZZ) \supset 
K^{\bullet}_{(1)} (X; \ZZ)\supset
\ldots$, where $K^{\bullet}_{(i)}(X; \ZZ)$ consists of the elements $x$ of
$K^{\bullet}(X; \ZZ)$ 
with the property that for any finite simplicial complex $Y$
of dimension less than $i$ and any continuous map $f : Y \rightarrow X$, the
pullback $f^* x$ vanishes in $K^{\bullet}(Y; \ZZ)$ \cite[Section
1]{Atiyah-Hirzebruch}.

Let ${H}^{(i)}(X;
\Lambda[u,u\inv ])^{\bullet}$ denote the subgroup of
${H}(X;
\Lambda[u,u\inv ])^{\bullet}$ consisting of terms of $X$-degree
equal to $i$, and similarly for
$\breve{H}^{(i)}(X; \Lambda[u,u\inv ])^{\bullet}$. 
Given $[E] \in K^{\bullet}_{(i)}(X; \ZZ)$, one can refine 
the component of 
$\ch([E])\in {H}(X; \QQ[u,u\inv
])^{\bullet}$ in ${H}^{(i)}(X; \QQ[u,u\inv
])^{\bullet}$
to an integer class $\ch^{(i)}([E])\in {H}^{(i)}(X; \ZZ[u,u\inv
])^{\bullet}$. 
Similarly, if $[{\mathcal E}] \in \cK^\bullet(X)$ and $c([{\mathcal
E}]) \in K^{\bullet}_{(i)}(X; \ZZ)$ 
then one can refine the 
component of $\breve{\ch}([{\mathcal E}]) \in \breve{H}(X; \QQ[u,u\inv
])^{\bullet}$ in $\breve{H}^{(i)}(X;\RQ)^{\bullet}$
to an element $\breve{\ch}^{(i)}([{\mathcal E}])\in
\breve{H}^{(i)}(X;
\ZZ[u,u\inv ])^{\bullet}$.
In general there is more than one such refinement,
but if $X$ is $(i-2)$-connected then there is a canonical choice.

Under the hypotheses of Theorem  \ref{thm:13}, 
suppose in addition that $B$ is $(2k-2)$-connected and
the Atiyah-Singer index $\ind^{top}([E]) = \ind^{an}([E]) \in 
K^{-n}(B; \ZZ)$
lies in the subset 
$K^{-n}_{(2k)}(B; \ZZ)\subset K^{-n}(B; \ZZ)$.
Then an explicit cocycle in
a certain integral Deligne cohomology group is
constructed in \cite[Section 4]{Lott2}.
More precisely, the cocycle is a
$2k$-cocycle for the Cech-cohomology of the complex of sheaves
\begin{equation} \ZZ \longrightarrow \Omega^0 \longrightarrow \ldots
\longrightarrow \Omega^{2k-1}
\end{equation}
on $B$. From the viewpoint of the present paper, the cocycle constructed in
\cite[Section 4]{Lott2} represents $\breve{\ch}^{(2k)}(\ind^{an}({\mathcal
E})) \in \breve{H}^{(2k)}(B; \ZZ[u,u\inv ])^{-n}$.  Then Theorem
 \ref{thm:13}  
implies that the same 
integral differential cohomology class
can be computed as
$\breve{\ch}^{(2k)}(\ind^{top}({\mathcal E}))$.

\section{Odd index theorem}\label{sec:8}

In this section we extend the index-theoretic 
results of the previous sections to
the case of odd differential K-theory classes. 
Because some of the arguments in the section are similar to what
was already done in the even case, 
we state some results without proof.

In Subsection \ref{oddsubsec1} we give a model for odd differential K-theory
whose generators consist of a Hermitian vector bundle, a compatible
connection, a unitary automorphism and an even differential form.  
We then construct suspension and desuspension maps between the even and
odd differential K-theory classes.  These are used to prove that
the odd groups defined here are isomorphic to those defined from the general
theory in~\cite{Hopkins-Singer}.

In Subsection \ref{oddsubsec2} we define the analytic and topological indices
in the case of an odd differential K-theory class on the total space of a
fiber bundle with even-dimensional fibers.  The definition uses suspension
and desuspension to reduce to the even case.  It is likely that the
topological index map agrees with that constructed using a more topological
model~\cite{Klonoff}.  We do not attempt to relate the odd analytic
pushforward with the odd Bismut superconnection.

In Subsection \ref{oddsubsec3} we define the analytic and topological indices
in the case of an even differential K-theory class on the total space of a
fiber bundle with odd-dimensional fibers. The definition uses the trick,
taken from \cite[Proof of Theorem 2.10]{Bismut-Freed}, of multiplying both
the base and fiber by a circle and then tensoring with the Poincar\'e line
bundle on the ensuing torus. 

In Subsection \ref{degreeone} we look at the result of applying a determinant
map to the analytic and topological index, to obtain a map from $B$ to
$S^1$. We show that this map is given by the reduced eta-invariants of the
fibers $X_b$. In particular if $Z$ is a closed Riemannian $\spinc$-manifold
of odd dimension $n$, and $\pi : Z \rightarrow \pt$ is the map to a point
then for any ${\mathcal E} \in \cK^0(Z)$, the topological and analytic
indices $\ind^{top}_{odd}({\mathcal E}), \ind^{an}_{odd}({\mathcal E}) \in
\cK^{-n}(\pt)$ equal the reduced eta-invariant $\overline{\eta}(Z, {\mathcal
E})$.  This is a version of the main theorem in~\cite{Klonoff}.

In Subsection \ref{indexgerbe} we indicate the relationship between
the odd differential K-theory index and the index gerbe of
\cite{Lott2}.

\subsection{Odd differential K-theory} \label{oddsubsec1}

Let $X$~be a smooth manifold.  We can describe $K^{-1}(X; \ZZ)$ as an abelian
group in terms of generators and relations.  The generators are complex
vector bundles $G$ over $X$ equipped with Hermitian metrics $h^G$ and unitary
automorphisms $U^G$.  The relations are that \begin{enumerate} \item $(G_2,
h^{G_2}, U^{G_2}) = (G_1, h^{G_1}, U^{G_1}) + (G_3, h^{G_3}, U^{G_3})$
whenever there is a short exact sequence of Hermitian vector bundles
\begin{equation} \label{shortexact2} 0 \longrightarrow G_1 {\longrightarrow}
G_2 {\longrightarrow} G_3 \longrightarrow 0
\end{equation}
so that the diagram
  \begin{equation} \label{commutes}
     \xymatrix{ 0 \ar[r]& G_1
\ar[d]^{U^{G_1}}\ar[r]^{}&G_2\ar[d]^{U^{G_2}}\ar[r]^{} &
G_3\ar[d]^{U^{G_3}}\ar[r] &0 \\ 0 \ar[r]&
           G_1\ar[r]^{}&G_2\ar[r]^{} &
           G_3\ar[r]& 0 } 
  \end{equation}
commutes.
\item $(G, h^G, U^G_1 \circ U^G_2) = (G, h^G, U^G_1) + 
(G, h^G, U^G_2)$.
\end{enumerate}

Given a generator $(G, h^G, U^G)$ for $K^{-1}(X; \ZZ)$,
let $\nabla^G$ be a unitary connection on $G$. Put
\begin{equation}
A(t) = (1-t) \nabla^G + t U^G \cdot \nabla^G \cdot (U^G)^{-1}
\end{equation}
and
\begin{equation}
\omega(\nabla^G, U^G) = \int_0^1 R_u 
\tr \left( e^{- u^{-1} (dt \, \partial_t + A(t))^2}
\right) \in \Omega (X;\R)^{-1}.
\end{equation}
Then $\omega(\nabla^G, U^G)$ is a closed form whose de Rham cohomology class 
$\ch(G,U^G) \in H (X;\R)^{-1}$ 
is independent of $\nabla^G$. The assignment
$(G, U^G) \rightarrow \ch(G, U^G)$ factors through 
a map $\ch : K^{-1}(X; \ZZ) \longrightarrow H (X;\R)^{-1}$ which
becomes an isomorphism after tensoring the left-hand side with $\RR$.

We can represent $K^{-1}(X; \ZZ)$
using $\ZZ/2\ZZ$-graded vector bundles.
A generator of $K^{-1}(X; \ZZ)$ is then a
$\ZZ/2\ZZ$-graded complex vector bundle $G = G_+ \oplus G_-$ on $X$, equipped 
with a Hermitian metric $h^G = h^{G_+} \oplus h^{G_-}$ and a
unitary automorphism $U^G_{\pm} \in \Aut(G_\pm)$.
Choosing compatible connections $\nabla^{G_\pm}$, and putting
\begin{equation}
A_\pm(t) =  (1-t) \nabla^G_\pm + t U^G_\pm \cdot \nabla^G_\pm \cdot 
(U^G)^{-1}_\pm,
\end{equation}
we put
\begin{equation} \label{Chern4}
\omega(\nabla^G, U^G) = \int_0^1 R_u 
\str \left( e^{- u^{-1} (dt \, \partial_t + t A(t))^2}
\right) \in \Omega (X;\R)^{-1}.
\end{equation}

If $\nabla^G_1$ and $\nabla^G_2$ are two 
metric-compatible connections on a Hermitian vector
bundle $G$ with a unitary automorphism $U^G$ then
there is an explicit form
$CS(\nabla^G_1,\nabla^G_2, U^G) \in 
\Omega (X;\R)^{-2}/\Image(d)$ so that
\begin{equation}
dCS(\nabla^G_1,\nabla^G_2, U^G) =
\omega(\nabla^G_1, U^G) - \omega(\nabla^G_2, U^G).
\end{equation}
More generally, if we have a short exact sequence (\ref{shortexact2}) of
Hermitian vector bundles, unitary automorphisms $U^{G_i} \in 
\Aut(G_i)$,
a commutative diagram (\ref{commutes}) 
and metric-compatible connections
$\{\nabla^{G_i}\}_{i=1}^3$ then there is an explicit form
$CS(\nabla^{G_1},\nabla^{G_2},\nabla^{G_3}, U^{G_1}, U^{G_2}, U^{G_3}) \in 
\Omega (X;\R)^{-2}/\Image(d)$ so that
\begin{equation}
dCS(\nabla^{G_1},\nabla^{G_2},\nabla^{G_3}, U^{G_1}, U^{G_2}, U^{G_3}) =
\omega(\nabla^{G_2}, U^{G_2}) - \omega(\nabla^{G_1},U^{G_1}) - 
\omega(\nabla^{G_3},U^{G_3}).
\end{equation}
To construct 
$CS(\nabla^{G_1},\nabla^{G_2},\nabla^{G_3}, U^{G_1}, U^{G_2}, U^{G_3})$,
put $W = [0,1] \times X$ and let $p : W \rightarrow X$ be the
projection map.  Put $F = p^* G_2$, $h^F = p^* h^{G_2}$ and
$U^F = p^* U^{G_2}$. Let $\nabla^F$ be a
unitary connection on $F$
which equals $p^* \nabla^{G_2}$ near
$\{1\} \times X$ and which equals $p^*(\nabla^{G_1} \oplus \nabla^{G_3})$
near $\{0\} \times X$. Then
\begin{equation}
CS(\nabla^{G_1},\nabla^{G_2},\nabla^{G_3}) = \int_0^1
\omega(\nabla^F, U^F) \in \Omega (X;\R)^{-2}/\Image(d).
\end{equation}

Also, if $G$ is a Hermitian vector bundle with a unitary connection $\nabla^G$
and two unitary automorphisms $U^G_1, U^G_2$ then there is an explicit form
$CS(\nabla^{G}, U^G_1, U^G_2) \in \Omega (X;\R)^{-2}/\Image(d)$ so that
\begin{equation} dCS(\nabla^{G}, U^G_1, U^G_2) = \omega(\nabla^{G}, U^G_1 \circ
U^G_2) - \omega(\nabla^{G}, U^G_1) - \omega(\nabla^{G}, U^G_2).  \end{equation}
To construct $CS(\nabla^{G}, U^G_1, U^G_2)$, let $\Delta \subset \AA^2$ be the
simplex \begin{equation} \Delta = \{(t_1, t_2) \in \AA^2 : t_1 \ge 0, t_2 \ge
0, t_1 + t_2 \le 1\}, \end{equation} with the orientation induced from the
canonical orientation of the affine plane~ $\AA^2$.  Put \begin{equation}
A(t_1, t_2) = \nabla^G + t_1 U^G_1 \cdot \nabla^G \cdot (U^G_1)^{-1} + t_2
(U^G_1 \circ U^G_2) \cdot \nabla^G \cdot (U^G_1 \circ U^G_2)^{-1}.
\end{equation} Then \begin{equation} CS(\nabla^{G}, U^G_1, U^G_2) = \int_\Delta
R_u \tr \left( e^{- u^{-1} (dt_1 \, \partial_{t_1} + dt_2 \, \partial_{t_2} +
A(t_1,t_2))^2} \right).  \end{equation}

 \begin{definition}\label{def:8.15}
The differential K-theory group $\cK^{-1}(X)$ is the abelian group 
defined by the following generators and relations.
The generators are quintuples ${\mathcal G} = 
(G, h^G, \nabla^G, U^G, \phi)$ where
\begin{itemize}
\item $G$ is a complex vector bundle on $X$.
\item $h^G$ is a Hermitian metric on $G$.
\item $\nabla^G$ is an $h^G$-compatible connection on $G$.
\item $U^G$ is a unitary automorphism of $G$.
\item $\phi \in \Omega (X;\R)^{-2}/\Image(d)$.
\end{itemize}

The relations are 
\begin{enumerate}
\item
${\mathcal G}_2 = {\mathcal G}_1 + {\mathcal G}_3$
whenever there is a short exact sequence (\ref{shortexact2}) of
Hermitian vector bundles, along with a commuting diagram
(\ref{commutes}) and
$\phi_2 = \phi_1 + \phi_3 - CS(\nabla^{G_1},\nabla^{G_2},\nabla^{G_3},
U^{G_1}, U^{G_2}, U^{G_3})$.
\item 
$(G, h^G, \nabla^G, U^G_1 \circ U^G_2, - CS(\nabla^G, U^G_1, U^G_2)) = 
(G, h^G, \nabla^G, U^G_1, 0) + 
(G, h^G, \nabla^G, U^G_2, 0)$.
\end{enumerate}
\end{definition}

By a generator of $\cK^{-1}(X)$ we mean a quadruple ${\mathcal G} = (G, h^G,
\nabla^G, U^G, \phi)$ as above. There is a homomorphism $\omega : \cK^{-1}(X)
\rightarrow \Omega(X; \R)^{-1}$ given on generators by $\omega({\mathcal G})
= \omega(\nabla^G, U^G) + d\phi$.

There is a similar model of $\cK^{r}(X)$ for any odd~$r$.  A
generator is a quintuple ${\mathcal G} = (G, h^G, \nabla^G, U^G, \phi)$ as
above with only a change in degree: $\phi \in \Omega (X;\R)^{r-1}/\Image(d)$.
Then $\omega({\mathcal G}) = u^{(r+1)/2}\omega(\nabla^G, U^G) + d\phi\in
\Omega (X;\R)^{r}$.  Also, the exact 
sequences \eqref{eq:4} and \eqref{eq:5} hold 
in odd degrees. 

There is a suspension map $S : \cK^{-1}(X) \rightarrow \cK^0(S^1 \times X)$
given on generators as follows. Let ${\mathcal G} = 
(G, h^G, \nabla^G, U^G, \phi)$ be a generator for
$\cK^{-1}(X)$. 
Let $p : [0,1] \times X \rightarrow X$ be the projection map.
Put $F = p^* G$, $h^F = p^* h^G$ and 
\begin{equation}
\nabla^F = dt \, \partial_t + (1-t) \nabla^G + t U^G \cdot \nabla^G 
\cdot (U^G)^{-1}.
\end{equation}
Let $(E, h^E, \nabla^E)$ be the
Hermitian vector bundle with connection on $S^1 \times X$ obtained
by gluing $F \big|_{\{0\} \times X}$ with $F \big|_{\{1\} \times X}$
using the automorphism $U^G$. Put 
\begin{equation}
\Phi = dt \wedge \phi \in \Omega(S^1 \times X; \R)^{-1}/\Image(d).
\end{equation}
Then $S({\mathcal G}) = 
(E, h^E, \nabla^E, \Phi ) \in \cK^0(S^1 \times X)$.
Equivalently, $S$ is multiplication by a certain element of
$\cK^{1}(S^1)$.

There is also a suspension map
$S : \cK^{0}(X) \rightarrow \cK^1(S^1 \times X)$
given on generators as follows. Let ${\mathcal E} = 
(E, h^E, \nabla^E, \Phi)$ be a generator for
$\cK^{0}(X)$. 
Let $p_1 : S^1 \times X \rightarrow S^1$ and
$p_2 : S^1 \times X \rightarrow X$ be the projection maps.
Let ${\mathcal L}$ be the trivial complex line bundle over
$S^1$, equipped with the product Hermitian metric $h^{\mathcal L}$,
the product connection $\nabla^{\mathcal L}$ and the automorphism 
$U^{\mathcal L}$ which multiplies
the fiber ${\mathcal L}_{e^{2\pi i t}}$ over 
$e^{2\pi i t} \in S^1$ by $e^{2\pi i t}$.
Put $G = p_1^* {\mathcal L} \otimes p_2^* E$, 
$h^G = p_1^* h^{\mathcal L} \otimes p_2^* h^E$,
$\nabla^G = p_1^* \nabla^{\mathcal L} \otimes I +
I \otimes p_2^* \nabla^E$, 
$U^G = p_1^* U^{\mathcal L}$ and 
\begin{equation}
\phi = p_1^* dt \wedge
p_2^* \Phi \in \Omega(S^1 \times X; \R)^{0}/\Image(d).
\end{equation}
Then $S({\mathcal E}) = 
(G, h^G, \nabla^G, U^G, \phi ) \in \cK^0(S^1 \times X)$.
Again, $S$ is multiplication by an element of
$\cK^{1}(S^1)$.

The double suspension $S^2 : \cK^0(X) \rightarrow \cK^2(T^2 \times X)$ can be
described explicitly as follows. Let $p_1 : T^2 \times X \rightarrow T^2$ and
$p_2 : T^2 \times X \rightarrow X$ be the projection maps.
Consider the trivial complex line bundle ${\mathcal M}$ on $[0,1]
\times [0,1]$ with product Hermitian metric and connection $dt_1 \,
\partial_{t_1} + dt_2 \, \partial_{t_2} - 2 \pi i t_1 dt_2$.  Define the
hermitian line bundle $P\to T^2$ by making identifications of ${\mathcal M}$
along the boundary of $[0,1] \times [0,1]$; it is the ``Poincar\'e'' line
bundle on the torus.  Let $h^P$ and $\nabla^P$ be the ensuing Hermitian
metric and compatible connection on $P$.  Given a generator ${\mathcal E} =
(E, h^E, \nabla^E, \phi)$ for $\cK^0(X)$, its double suspension is
\begin{equation} \label{eq:8.19}
 S^2({\mathcal E}) = \left( p_1^* P \otimes p_2^* E, p_1^*
h^P \otimes p_2^* h^E, p_1^* \nabla^P \otimes I + I \otimes p_2^* \nabla^E,
dt_1 \wedge dt_2 \wedge \phi \right) \in \cK^2(T^2 \times X).
\end{equation}

Finally, there is a desuspension map $D : \cK^0(S^1 \times X) \rightarrow
\cK^{-1}(X)$
given on generators as follows. Let ${\mathcal E} = 
(E, h^E, \nabla^E, \Phi)$ be a generator for
$\cK^0(S^1 \times X)$.
Picking a basepoint $\star \in S^1$,
define $A : X \rightarrow S^1 \times X$ by
$A(x) = (\star,x)$. Put $G = A^*E$, $h^G = A^* h^E$ and
$\nabla^G = A^* \nabla^E$. Let $U^G \in \Aut(G)$ be the
map given by parallel transport around the circle fibers on $S^1 \times X$,
starting from $\{\star\} \times X$.
Put 
\begin{equation}
\phi = \int_{S^1} \Phi \in \Omega(X; \R)^{-2}/\Image(d).
\end{equation}
Then 
  \begin{equation}\label{eq:105}
     D({\mathcal E}) = (G, h^G, \nabla^G, U^G, \phi) \in \cK^{-1}(X) .
  \end{equation}
It is independent of the choice of basepoint $\star \in S^1$.
Also, $D \circ S$ is the identity on $\cK^{-1}(X)$.

There are analogous suspension and
desuspension maps in other degrees. 

We now show how to use 
the suspension and desuspension maps
to relate Definition~\ref{def:8.15} to
the notion of odd differential K-theory groups from
\cite{Hopkins-Singer}.

        \begin{proposition}[]\label{thm:4}
 The odd differential $K$-groups defined here are isomorphic to those defined
by the theory of generalized differential cohomology.
        \end{proposition}

        \begin{proof}
 Temporarily denote the geometrically defined groups in
Definition~\ref{def:1.14} and Definition~\ref{def:8.15} 
as~$\breve{L}^{\bullet}(X)$.
Then,  looking at 
degree~$-1$ for definiteness, the composition 
  \begin{equation}\label{eq:107}
     \xymatrix@1{\breve{L}\inv (X)\ar@<.5ex>[r]^<(.15)S &\breve{L}^0(\cir\times
     X)\cong \cK^0(\cir\times X)\ar@<.5ex>[l]^<(.25)D\ar[r]^<(.2){\int_{\cir}}&
     \cK\inv (X)}
  \end{equation}
defines a homomorphism we claim is an isomorphism.  (See Remark~\ref{thm:3}
for the isomorphism in the middle of~\eqref{eq:107}.)  This follows from the
5-lemma applied to the diagram  
  \begin{equation}\label{eq:108}
     \xymatrix{ 0 \ar[r]& \dfrac{\Omega (X;\R)^{-2}}{\Omega
     (X;\R)_K^{-2}}\ar[d]\ar[r]^{\quad
     j}&\breve{L}^{-1}(X)\ar[d]\ar[r]^{c\quad} &
     K^{-1}(X;\ZZ)\ar[d]\ar[r] &0 \\ 
    0 \ar[r]& \dfrac{\Omega (X;\R)^{-2}}{\Omega
     (X;\R)_K^{-2}}\ar[r]^{\quad
     j}&\cK^{-1}(X)\ar[r]^{c\quad} &
     K^{-1}(X;\ZZ)\ar[r] &0} 
  \end{equation}
in which the rows are exact and the outer vertical arrows are the identity. 
        \end{proof}
 
Note that under the isomorphism 
$\breve{L}^0 (\cir\times X)\cong \cK^0 (\cir\times
X)$, the desuspension map corresponds to integration over the circle.

\subsection{Index theorem : even-dimensional fibers, odd classes}
\label{oddsubsec2}

Suppose that $\pi : X \rightarrow B$ is a 
 proper submersion of relative dimension $n$, with
$n$ even. 
Suppose that $\pi$ is 
equipped with a Riemannian structure
and a differential $\spinc$-structure.
The  product submersion 
$\pi^\prime : S^1 \times X \rightarrow S^1 \times B$
inherits a product Riemannian structure and differential
$\spinc$-structure, so we have the analytic and
topological indices $\ind^{an}, \ind^{top} :
\cK^0(S^1 \times X) \rightarrow \cK^{-n}(S^1 \times B)$.

Define the analytic index
\begin{equation}
\ind^{an}_{odd} :  \cK^{-1}(X) \rightarrow
\cK^{-n-1}(B)
\end{equation}
by $\ind^{an}_{odd} = D \circ \ind^{an} \circ S$.
Define the
topological index $\ind^{top}_{odd} :  \cK^{-1}(X) \rightarrow
\cK^{-n-1}(B)$ by 
\begin{equation}
\ind^{top}_{odd} = D \circ \ind^{top} \circ S.
\end{equation}
As an immediate consequence of Theorem  \ref{thm:13}, 
$\ind^{top}_{odd} = \ind^{an}_{odd}$.

\subsection{Index theorem : odd-dimensional fibers}
\label{oddsubsec3}

Suppose that $\pi : X \rightarrow B$ is a 
 proper submersion of relative dimension $n$, 
with $n$ odd. 
Suppose that $\pi$ is 
equipped with a Riemannian structure
and a differential $\spinc$-structure.

Following \cite[Pf. of Theorem 2.10]{Bismut-Freed}, we construct
a new  submersion 
$\pi^\prime : T^2 \times X \rightarrow S^1 \times B$
by multiplying the base and fibers by $S^1$.
Given $a > 0$, we endow $\pi^\prime$
with a product Riemannian structure and differential
$\spinc$-structure, so that the circle fibers have length $a$.
Then we have analytic and
topological indices $\ind^{an}, \ind^{top} :
\cK^2(T^2 \times X) \rightarrow \cK^{-n+1}(S^1 \times B)$.

Define the analytic index
$\ind^{an}_{odd} :  \cK^{0}(X) \rightarrow
\cK^{-n}(B)$
by 
\begin{equation}\label{eq:8.26}
\ind^{an}_{odd} = D \circ \ind^{an} \circ S^2.
\end{equation}
Define the
topological index $\ind^{top}_{odd} :  \cK^{0}(X) \rightarrow
\cK^{-n}(B)$ by 
\begin{equation}\label{eq:8.30}
\ind^{top}_{odd} = D \circ \ind^{top} \circ S^2.
\end{equation}
As an immediate consequence of Theorem  \ref{thm:13}, 
$\ind^{top}_{odd} = \ind^{an}_{odd}$.

\begin{lemma}
The indices
$\ind^{top}_{odd}$ and $\ind^{an}_{odd}$ are independent
of the choice of fiber circle length $a$. 
\end{lemma}
\begin{proof}
Given $a_0, a_1 > 0$ and  ${\mathcal E} \in \cK^{0}(X)$, let 
$\ind^{top}_{odd,0}({\mathcal E})$ and 
$\ind^{top}_{odd,1}({\mathcal E})$ denote the indices in
$\cK^{-n}(B)$ as computed using circle fibers of length
$a_0$ and $a_1$, respectively.
Consider the fiber bundle
$\pi^{\prime \prime} :
[0,1] \times X \rightarrow [0,1] \times B$, equipped with the
product Riemannian structure and differential
$\spinc$-structure.
Let $a : [0,1] \rightarrow \RR^+$ be a smooth function so that
$a(t)$ is $a_0$ near $t = 0$, and $a(t)$ is $a_1$ near $t=1$.
Given $t \in [0,1]$, let the circle
fiber length over $\{t\} \times B$
in the fiber bundle $[0,1] \times T^2 \times X \rightarrow
[0,1] \times B$ be
$a(t)$.
Let $\ind^{top}_{odd,[0,1] \times B}({\mathcal E}) 
\in \cK^{-n}([0,1] \times B)$
be the topological index of ${\mathcal E}$ as computed using
the fiber bundle $\pi^{\prime \prime}$. 
Using Lemmas \ref{restriction} and \ref{omegaagrees},
\begin{equation}
\ind^{top}_{odd,1} - \ind^{top}_{odd,0} =
j(\int_{[0,1]} \omega(\ind^{top}_{odd,[0,1] \times B})) =
j \left( \int_{[0,1]} \Td \left( \widehat{\nabla}^{T^V([0,1] 
\times T^2 \times X)} \right)
\wedge \omega({\mathcal E}) \right).
\end{equation}
One can check that
$\widehat{\nabla}^{T^V([0,1] 
\times T^2 \times X)}$ pulls back from a connection
$\widehat{\nabla}^{T^V(T^2 \times X)}$ on the vertical
tangent bundle of $\pi^\prime : T^2 \times X \rightarrow S^1 \times B$.
Thus 
\begin{equation}
\int_{[0,1]} \Td \left( \widehat{\nabla}^{T^V([0,1] 
\times T^2 \times X)} \right)
\wedge \omega({\mathcal E}) =0.
\end{equation} 
This proves the lemma.
\end{proof}

\subsection{Degree-one component of the index theorem} \label{degreeone}

There is a homomorphism $\Det$ from $\cK^{-1}(X)$ to the space of smooth maps
$[X, S^1]$. Namely, given a generator
${\mathcal E} = (E, h^E, \nabla^E, U^E, \phi)$ of $\cK^{-1}(X)$, its image
$\Det({\mathcal E}) \in [X, S^1]$ sends $x \in X$ to $e^{2\pi i \phi_{(0)}(x)}
\det(U^E(x))$, where $\phi_{(0)} \in \Omega^0(X)$ denotes $u$
times the component of $\phi \in \Omega(X; \R)^{-2}/\Image(d)$ in
$u^{-1} \Omega^0(X)/\Image(d) = u^{-1} \Omega^0(X)$.

\begin{proposition}
Let $\pi : X \rightarrow B$ be a  proper submersion of relative
dimension $n$, with $n$ odd. 
Suppose that $\pi$ is equipped with a
Riemannian structure and a differential $\spinc$-structure.
Given ${\mathcal E} \in \cK^0(X)$, after using periodicity to shift
$\ind^{an}({\mathcal E}) \in \cK^{-n}(B)$ into $\cK^{-1}(B)$, the map
$\Det(\ind^{an}_{odd}({\mathcal E}))$ sends $b \in B$ to
$u^{\frac{n+1}{2}}
\overline{\eta} \left( X_b, {\mathcal E} \big|_{X_b} \right) \in 
\RR/\ZZ$, where $X_b = \pi^{-1}(b)$.
\end{proposition}
\begin{proof}
From \eqref{eq:8.26}, we want to apply $\Det \circ D$ to
$(\ind^{an} \circ S^2)({\mathcal E})$
and evaluate the result at $b$. Using \eqref{eq:8.19},
$S^2({\mathcal E})$ is a certain element of
$\cK^2(T^2 \times X)$ and then 
$(\ind^{an} \circ S^2)({\mathcal E})$ is its analytic index
in $\cK^{-n+1}(S^1 \times B)$. In order to apply
$\Det \circ D$ to this, and then compute the result at $b$,
it suffices to just use the restriction of 
$(\ind^{an} \circ S^2)({\mathcal E})$
to $S^1 \times \{b\}$. After doing so,
the proof of Proposition \ref{determinant} implies the result
of applying $\Det \circ D$ is
\begin{equation}
\lim_{\epsilon \rightarrow 0}
\overline{\eta} \left( D^{(T^2 \times X_b)^\epsilon, 
p_1^* P \otimes p_2^* E \big|_{X_b}} \right)
+ u^{\frac{n+1}{2}} \int_{S^1} \int_{S^1 \times X_b} 
\Td \left( \widehat{\nabla}^{T^V(T^2 \times X_b)} \right)
\wedge p_2^* \phi.
\end{equation}
For any $\epsilon > 0$, separation of variables gives
\begin{equation}
\overline{\eta} \left( D^{(T^2 \times X_b)^\epsilon, 
p_1^* P \otimes p_2^* E \big|_{X_b}} \right) =
\Index(D^{T^2, P}) \cdot
\overline{\eta} \left( D^{X_b,E \big|_{X_b}} \right) = 
\overline{\eta} \left( D^{X_b,E \big|_{X_b}} \right).
\end{equation}
Then the evaluation of $\Det(\ind^{an}_{odd}({\mathcal E}))$ at $b$ is
\begin{equation}
\overline{\eta} \left( D^{X_b,E \big|_{X_b}} \right)
+ u^{\frac{n+1}{2}} \int_{X_b} \Td \left( \widehat{\nabla}^{TX_b} \right)
\wedge \phi = u^{\frac{n+1}{2}} \overline{\eta}(X_b, {\mathcal E}\big|_{X_b}).
\end{equation}
This proves the proposition.
\end{proof}

\begin{corollary}[\cite{Klonoff}] \label{etapush}
Suppose that $Z$ is a closed Riemannian
$\spinc$-manifold of odd dimension $n$
with a $\spinc$-connection $\widehat{\nabla}^{TZ}$.
Let $\pi : Z \rightarrow \pt$ be the mapping to a point.
Given ${\mathcal E}=(E,h^E,\nabla ^E,\phi ) \in \cK^0(Z)$, we have
\begin{equation}
\ind^{an}_{odd}({\mathcal E}) = \ind^{top}_{odd}({\mathcal E}) = 
\overline{\eta}(Z, {\mathcal E})
\end{equation}
in $\cK^{-n}(\pt) = u^{-\frac{n+1}{2}} \cdot (\RR/\ZZ)$.
\end{corollary}

\subsection{Index gerbe} \label{indexgerbe}

Let $\pi : X \rightarrow B$ be a 
proper submersion of relative dimension $n$, 
with $n$ odd. 
Suppose that $\pi$ is equipped with
a Riemannian structure and a differential $\spinc$-structure.
If $E$ is a Hermitian vector bundle on $X$ with 
compatible connection $\nabla^E$ then one can construct an
index gerbe \cite{Lott2}, which is an abelian gerbe-with-connection.
Isomorphism classes of such
gerbes-with-connection are in bijection with
the differential cohomology group $\breve{H}^3(X; \ZZ)$.
On the other hand, taking ${\mathcal E} = (E, h^E, \nabla^E, 0) \in
\cK^0(X)$, the component of
$\breve{\ch}(\ind^{an}_{odd}({\mathcal E})) = 
\breve{\ch}(\ind^{top}_{odd}({\mathcal E})) \in
\breve{H}(B; \QQ[u,u^{-1}])^{\bullet}$ in
$\breve{H}^3(B; \QQ)$ comes from an element of $\breve{H}^3(B; \ZZ)$.
Presumably this is the class of the index gerbe.  It should be
possible to prove this by comparing  the holonomies around
surfaces in $B$, along the lines of the proof of
Proposition \ref{determinant}.

\appendix

  \section*{Appendix: Limits in differential $K$-theory}\label{sec:A}
\setcounter{section}{1}\setcounter{equation}{0}

Let $X$~be a topological
space and let $X_1\subset X_2\subset \cdots$ be an increasing sequence
of compact subspaces whose union is~$X$.  Milnor~\cite{Milnor} proved that
for every cohomology theory~$h$ and every integer~$q$, there is an exact
sequence
  \begin{equation}\label{eq:113}
     0\longrightarrow {\varprojlim\limits_{i}}^1\,h^q(X_i)\longrightarrow
     h^q(X)\longrightarrow \varprojlim\limits_{i}\,h^q(X_i)
     \longrightarrow 0, 
  \end{equation}
in which the quotient is the (inverse) limit and the kernel is its first
derived functor; see~\cite[\S3.F]{Hatcher}, \cite[\S19.4]{May} for modern
expositions and~\cite[\S4]{Atiyah-Segal} for the specific case of $K$-theory.
In this appendix we prove that the \emph{differential} cohomology of the
union of compact manifolds is isomorphic to the inverse limit of the
differential cohomologies: there is no $\varprojlim^1$ term.  So as to not
introduce new notation, we present the argument for differential $K$-theory,
which is the case of interest for this paper.  We remark that for integral
differential cohomology the theorem is immediate if we use the isomorphism with
Cheeger-Simons differential characters, as a differential character is
determined by its restriction to compact submanifolds.  The universal
coefficient theorem for $K$-theory~\cite {Yosimura} plays an analogous role
in the following proof.

        \begin{theorem}[]\label{thm:10}
 Let $X$~be a smooth manifold and $X_1\subset X_2\subset \cdots$ an
increasing sequence of compact codimension-zero submanifolds-with-boundary
whose union is~$X$.  Then for any integer~$q$, the restriction maps induce an
isomorphism
  \begin{equation}\label{eq:114}
     \cK^q(X)\longrightarrow \varprojlim\limits_i\,\cK^q(X_i). 
  \end{equation}
        \end{theorem}

The proof is based on the exact sequence~\eqref{eq:4} and the following
lemmas.  

        \begin{lemma}[]\label{thm:11}
 The restriction maps induce an isomorphism
  \begin{equation}\label{eq:115}
     \Omega(X;\R)_K^q\longrightarrow
     \varprojlim\limits_i\,\Omega(X_i;\R)_K^q.
  \end{equation}
        \end{lemma}

        \begin{proof}
We first note that the
universal coefficient theorem for $K$-theory implies that
$K^q(X)/\textnormal{torsion}\cong \Hom\bigl(K_q(X),\ZZ \bigr)$.  Representing
$K$-homology classes by $\spinc$-manifolds, as in the proof of
Theorem~\ref{maintheorem}, we see that the cohomology class of a closed
differential form $\omega \in \Omega (X;\R)^q$ is in the image of the Chern
character if and only if for every closed Riemannian $\spinc$-manifold
$W$ and every smooth map $f : W \rightarrow X$, we have
  \begin{equation}\label{eq:122}
     \int_{W}\Td(W)\wedge f^* \omega \;\in u^{\frac{q-\dim(W)}{2}} \cdot \ZZ .
  \end{equation}
For each such choice of $W$ and $f$, we know that $f(W) \subset X_i$
for some~$i$.

Returning to the map (\ref{eq:115}),
a differential form is determined pointwise, so the map is injective.  For
the same reason, an element~$\{\omega _i\}$ in the inverse limit on the
right-hand side glues to a global
differential form~$\omega $, and $d\omega =0$ since $d$~is local. 
Since $\omega \big|_{X_i} \in \Omega(X_i;\R)_K^q$ for each $i$,
the previous paragraph implies that
$\omega \in \Omega(X;\R)^q_K$. This proves the lemma. 
        \end{proof}

        \begin{lemma}[]\label{thm:12}
 The restriction maps induce an isomorphism 
  \begin{equation}\label{eq:116}
     K^{q-1}(X;\RZ)\longrightarrow \varprojlim\limits_i\,K^{q-1}(X_i;\RZ). 
  \end{equation}
        \end{lemma}

        \begin{proof}
 The Ext term in the universal coefficient theorem for $K$-theory vanishes
since $\RR/\ZZ$ is divisible.  Applying the universal coefficient theorem
twice and the fact that homology commutes with colimits, we obtain
  \begin{equation}\label{eq:117}
  \begin{split}
     K^{q-1}(X;\RZ)&\cong \Hom\bigl(K_{q-1}(X),\RZ \bigr)\\&\cong
     \Hom\bigl(\varinjlim\limits_iK_{q-1}(X_i),\RZ \bigr)\\&\cong
     \varprojlim\limits_i\Hom\bigl(K_{q-1}(X_i),\RZ \bigr)\\&\cong
     \varprojlim\limits_i\,K^{q-1}(X_i;\RZ).
  \end{split}
  \end{equation}
        \end{proof}

        \begin{proof}[Proof of Theorem~\ref{thm:10}]
 Milnor's exact sequence~\eqref{eq:113} and Lemma~\ref{thm:12} imply that 
  \begin{equation}\label{eq:118}
     {\varprojlim\limits_i}^1\,K^{q-1}(X_i;\RZ)=0. 
  \end{equation}
Thus the limit of the short exact sequence~\eqref{eq:4} is a short exact
sequence~\cite[\S3.5]{Weibel}.  The restriction maps then fit into a
commutative diagram
  \begin{equation}\label{eq:120}
     \xymatrix{0 \ar[r] &K^{q -1}(X;\RZ)\ar[r]\ar[d]&\cK^q
     (X)\ar[r]\ar[d] &\Omega(X;\R)_K^q \ar[r]\ar[d] &0\\
      0 \ar[r] &\varprojlim\limits_iK^{q
     -1}(X_i;\RZ)\ar[r]&\varprojlim\limits_i\cK^q 
     (X_i)\ar[r]&\varprojlim\limits_i\Omega(X_i;\R)_K^q \ar[r] &0
 }
  \end{equation}
in which the rows are exact and, by the lemmas, the outer vertical arrows are
isomorphisms.  The theorem now follows from the 5-lemma. 
        \end{proof}


\begin{thebibliography}{99}

\bibitem{Atiyah-Hirzebruch0} M. Atiyah and F. Hirzebruch,
``Riemann-Roch theorems for differentiable manifolds'',
Bull. Amer. Math. Soc. 65, p. 276-281 (1959)

\bibitem{Atiyah-Hirzebruch} M. Atiyah and F. Hirzebruch,
``Vector bundles and homogeneous spaces'', Proc. Sympos. Pure Math.
Vol. III, American Math. Soc., Providence, p. 7-38 (1961)

\bibitem{Atiyah-Patodi-Singer} M. Atiyah, V. Patodi and I.M. Singer,
``Spectral asymmetry and Riemannian geometry I'', Math. Proc. Cambridge
Philos. Soc. 77, p. 43-69 (1975)

\bibitem{Atiyah-Patodi-SingerIII} M. Atiyah, V. Patodi and I.M. Singer,
``Spectral asymmetry and Riemannian geometry III'', Math. Proc. Cambridge
Philos. Soc. 79, p. 71-99 (1976)

\bibitem{Atiyah-Segal}
M. F. Atiyah, G. B. Segal, ``Equivariant $K$-theory and completion'',
J. Diff. Geom. 3, p. 1-18  (1969)

\bibitem{Atiyah-Singer0} M. Atiyah and I. Singer,
``The index of elliptic operators on compact manifolds''
Bull. Amer. Math. Soc. 69, p. 422-433 (1963)

\bibitem{Atiyah-Singer} M. Atiyah and I. Singer,
``The index of elliptic operators IV'', Ann. Math. 93, p. 119-138 (1971)

\bibitem{Baum-Douglas} P. Baum and R. Douglas,
``K-homology and index theory'', in 
\underline{Operator algebras and applications} I, Proc. Symp. Pure Math.
38, ed. R. Kadison, AMS, Providence, p. 117-173 (1982)

\bibitem{Berline-Getzler-Vergne} N. Berline, E. Getzler and M. Vergne,
\underline{Heat kernels and Dirac operators}, Springer, New York (2004)

\bibitem{Bismut} J.-M. Bismut, ``The Atiyah-Singer index theorem for
families of Dirac operators: two heat equation proofs'', Invent. Math.
83, p. 91-151 (1985)

\bibitem{Bismut-Cheeger} J.-M. Bismut and J. Cheeger,
``$\eta$-invariants and their adiabatic limits'', J. Amer. Math. Soc. 2,
p. 33-70 (1989)

\bibitem{Bismut-Freed} J.-M. Bismut and D. Freed,
``The analysis of elliptic families II. Dirac operators, eta invariants
and the holonomy theorem'', Comm. Math. Phys. 107, p. 103-163 (1986)

\bibitem{Bismut-Zhang} J.-M. Bismut and W. Zhang,
``Real embeddings and eta invariants'', Math. Ann. 295, p. 661-684 (1993)

\bibitem{Bunke-Kreck-Schick} U. Bunke, M. Kreck and T. Schick,
``A geometric description of smooth cohomology'',
preprint, http://arxiv.org/abs/0903.5290 (2009)

\bibitem{Bunke-Schick} U. Bunke and T. Schick,
``Smooth K-Theory'', to appear, Ast\'erisque,
http://arxiv.org/abs/0707.0046

\bibitem{Cheeger-Simons} J. Cheeger and J. Simons,
``Differential characters and geometric invariants'', in
\underline{Geometry and topology}, Lecture Notes in Math. 1167, 
Springer, Berlin (1985)

\bibitem{Dai} X. Dai, ``Adiabatic limits, non-multiplicativity of
signature and Leray spectral sequence'', J. Amer. Math. Soc. 4, 
p. 265-321 (1991)

\bibitem{Faltings} G. Faltings,
\underline{Lectures on the arithmetic Riemann-Roch theorem},
Ann. Math. Studies 127, Princeton University Press, Princeton (1992)

\bibitem{Freed} D. S. Freed, ``Dirac charge quantization and generalized
differential cohomology'', in \underline{Surveys in differential
geometry VII}, International Press, Somerville, p. 129-194 (2000)

\bibitem{Freed2} D. S. Freed,  ``Pions and Generalized Cohomology'', J.\
Diff.\ Geom. 80, p. 45-77 (2008)
 
\bibitem{Freed3} D. S. Freed, ``On determinant line bundles'', in
\underline{Mathematical Aspects of String Theory}, ed. S. T. Yau, 
World Scientific
Publishing (1987)

\bibitem{Freed-Hopkins} D. Freed and M. Hopkins,
``On Ramond-Ramond fields and K-theory'', J. High Energy Phys. 2000,
no. 5, paper 44 (2000)

\bibitem{Gillet-Soule} H. Gillet and C. Soul\'e,
``Characteristic classes for algebraic vector bundles with
Hermitian metric II'', Ann. Math. 131, p. 205-238 (1990)

\bibitem{Gillet-Rossler-Soule} H. Gillet, D. R\"ossler and
C. Soul\'e, Ann. Inst. Fourier (Grenoble) 58, p. 2169-2189 (2008)

\bibitem{Hatcher}  
A. Hatcher, \underline{Algebraic Topology}, Cambridge University Press,
Cambridge (2002) 

\bibitem{Hirzebruch} F. Hirzebruch, \underline{Topological methods in
algebraic geometry}, Springer-Verlag, New York, 3rd edition (1978)

\bibitem{Hopkins-Hovey} M. Hopkins and M. Hovey,
``Spin cobordism determines real $K$-theory'', Math. Z. 210, p. 181-196
(1992)

\bibitem{Hopkins-Singer} M. Hopkins and I. Singer,
``Quadratic functions in geometry, topology and M-theory'',
J. Diff. Geom. 70, p. 329-452 (2005)

\bibitem{Hormander} L. H\"ormander, \underline{The analysis of linear
partial differential operators I}, Springer, New York (2003)

\bibitem{Karoubi} M. Karoubi, \underline{Homologie cyclique et K-th\'eorie},
Ast\'erisque 149 (1987)

\bibitem{Karoubinew}M. Karoubi, $K$-th\'orie
multiplicative. C. R. Acad. Sci. Paris S\'er. I Math. 302 (1986), no. 8,
321--324 MR 0838584 


\bibitem{Klonoff} K. Klonoff, ``An index theorem in differential
$K$-theory'', 2008~Ph.D. thesis, available electronically at {\tt
http://hdl.handle.net/2152/3912}. 

\bibitem{Lott} J. Lott, ``$\RR/\ZZ$-index theory'', Comm. Anal. Geom. 2,
p. 279-311 (1994)

\bibitem{Lott3} J. Lott, ``Secondary analytic indices'', in
\underline{Regulators in analysis, geometry and number theory},
Progr. Math. 171, Birkh\"auser Boston, Boston, p. 231-293 (2000)

\bibitem{Lott2} J. Lott, ``Higher degree analogs of the
determinant line bundle'', Comm. Math. Phys. 230, p. 41-69 (2002)

\bibitem{May} J. P. May, \underline{A concise course in algebraic topology},
Chicago Lectures in Mathematics, University of Chicago Press, Chicago
(1999) 

\bibitem{Melrose-Piazza} R. Melrose and P. Piazza,
``Families of Dirac operators, boundaries and the b-calculus'',
J. Diff. Geom. 46, p. 99-180 (1997)

\bibitem{Milnor} J. Milnor, \underline{Characteristic Classes},
Princeton University Press, Princeton (1974)
 
\bibitem{Mischenko-Fomenko}
A. Mischenko and A. Fomenko, ``The Index of Elliptic Operators over 
$C^*$-Algebras'', Izv. Akad. Nauk SSSR, Ser. Mat. 43,
p. 831-859 (1979)

\bibitem{Ortiz} M. Ortiz, ``Differential equivariant K-theory'',
http://arxiv.org/abs/0905.0476 (2009)

\bibitem{Quillen} 
D. Quillen, ``Determinants of Cauchy-Riemann operators over a Riemann
surface'', Funk. Anal. iprilozen 19, p. 37-41 (1985)

\bibitem{Quillen2} 
D. Quillen, ``Superconnections and the Chern character'',
Topology 24, p. 89-95 (1985)

\bibitem{Simons-Sullivan} J. Simons and D. Sullivan, 
``Structured vector bundles define differential K-theory'',
http://arxiv.org/abs/0810.4935 (2008)

\bibitem{Weibel}
C. A. Weibel, \underline{An introduction to homological algebra}. Cambridge
Studies in Advanced Mathematics, 38, Cambridge University Press, Cambridge
(1994)  
 
\bibitem{Yosimura}
Z. Yosimura, ``Universal coefficient sequences for cohomology theories of
${\rm CW}$-spectra'', Osaka J. Math. 12, p. 305-323 (1975)

 \end{thebibliography}
 \end{document}